\documentclass{article}

\usepackage{amsmath, amssymb, graphicx, parskip, tikz-cd, color, bm, physics, enumerate,  upgreek}
\usepackage{hyperref}
\hypersetup{
     colorlinks=true,
     linkcolor=blue!60!black,
     filecolor=blue!75!black,
     citecolor = green!50!black,      
     urlcolor=magenta,
     }
\usepackage[mathscr]{euscript}
\usepackage[margin=1in]{geometry}
\usepackage[utf8]{inputenc}
\usepackage{calligra}
\usepackage{mathrsfs}
\usetikzlibrary{cd}
\usepackage[fleqn,tbtags]{mathtools}

\counterwithin{equation}{section}

\usepackage{tikz}
\usepackage{tikz-cd}
\usetikzlibrary{trees}
\usetikzlibrary[shapes]
\usetikzlibrary[arrows]
\usetikzlibrary{patterns}
\usetikzlibrary{fadings}
\usetikzlibrary{backgrounds}
\usetikzlibrary{decorations.pathreplacing}
\usetikzlibrary{decorations.pathmorphing}
\usetikzlibrary{positioning}
\usetikzlibrary{shapes.geometric}
\tikzstyle{box1} = [rectangle, rounded corners, minimum width=7cm, minimum height=1cm, text centered, text width=6.5cm, draw=black]
\tikzstyle{box2} = [rectangle, minimum width=7cm, minimum height=1cm, text centered, text width=6.8cm, draw=black]
\tikzstyle{arrow} = [thick,->,>=stealth]
\tikzstyle{arrow2} = [thick,->,>=stealth,dotted]

\usepackage{amsthm}

\theoremstyle{plain}
\newtheorem{Th}{Theorem}[section]
\newtheorem{Lemma}[Th]{Lemma}
\newtheorem{Cor}[Th]{Corollary}
\newtheorem{Prop}[Th]{Proposition}
\newtheorem{Claim}[Th]{Claim}

\newtheorem{thmx}{Theorem}
 
\newtheorem{conjx}{Conjecture}

 \theoremstyle{definition}
\newtheorem{Def}[Th]{Definition}

\newtheorem{Rem}[Th]{Remark}
\newtheorem{Rec}[Th]{Recollection}

\newtheorem{?}[Th]{Problem}
\newtheorem{Ex}[Th]{Example}

\newtheorem{convention}[Th]{Convention}
\newtheorem{construction}[Th]{Construction}

\DeclareMathOperator{\coker}{coker}

\DeclareMathOperator{\im}{im}

\DeclareMathOperator{\id}{id}

\DeclareMathOperator{\Hom}{Hom}

\DeclareMathOperator{\Nat}{Nat}

\DeclareMathOperator{\Aut}{Aut}

\DeclareMathOperator{\Ext}{Ext}
\DeclareMathOperator{\Spec}{Spec}

\DeclareMathOperator{\Ann}{Ann}
\DeclareMathOperator{\ind}{ind}

\DeclareMathOperator{\End}{End}
\DeclareMathOperator{\even}{even}

\DeclareMathOperator{\Sym}{Sym}

\DeclareMathOperator{\tors}{tors}

\DeclareMathOperator{\SHom}{\mathscr{H}\text{\kern -3pt {\calligra\large om}}\,}
\DeclareMathOperator{\SEnd}{\mathscr{E}\text{\kern -3pt {\calligra\large nd}}\,}
\DeclareMathOperator{\res}{res}
\newcommand{\mf}{\mathfrak}
\newcommand{\GL}{\mathrm{GL}}

\usepackage{listings}
\definecolor{codegreen}{rgb}{0,0.6,0}
\definecolor{codegray}{rgb}{0.5,0.5,0.5}
\definecolor{codepurple}{rgb}{0.58,0,0.82}
\definecolor{backcolour}{rgb}{0.95,0.95,0.92}
 
\lstdefinestyle{mystyle}{
    backgroundcolor=\color{backcolour},   
    commentstyle=\color{codegreen},
    keywordstyle=\color{magenta},
    numberstyle=\tiny\color{codegray},
    stringstyle=\color{codepurple},
    basicstyle=\footnotesize,
    breakatwhitespace=false,         
    breaklines=true,                 
    captionpos=b,                    
    keepspaces=true,                 
    numbers=left,                    
    numbersep=5pt,                  
    showspaces=false,                
    showstringspaces=false,
    showtabs=false,                  
    tabsize=2
}
\usepackage{hieroglf, oands}

\usepackage{titletoc,tocloft}
\makeatletter
\renewcommand*\l@section{\@dottedtocline{1}{1.5em}{1.5em}}
\makeatother
\begin{document}

\title{Notes on Cohomological Finite Generation for Finite Group Schemes}

\author{ Juan Omar G\'omez \& Chris J. Parker}
\maketitle

\begin{abstract}
These are extended notes based on lectures given by Vincent Franjou, Paul Sobaje, Peter Symonds and Antoine Touz\'e at the Master Class on \textit{New Developments in Finite Generation of Cohomology} that took place at Bielefeld University in September 2023. Their aim to give a panoramic overview of van der Kallen’s recent result on the finite generation of cohomology for finite group schemes over an arbitrary Noetherian base, bringing together the many papers that this theorem relies on, and supplying the necessary background and exposition.
\end{abstract}

\vspace*{\fill}
\begin{figure}[h]
\centering
\includegraphics[scale=0.7]{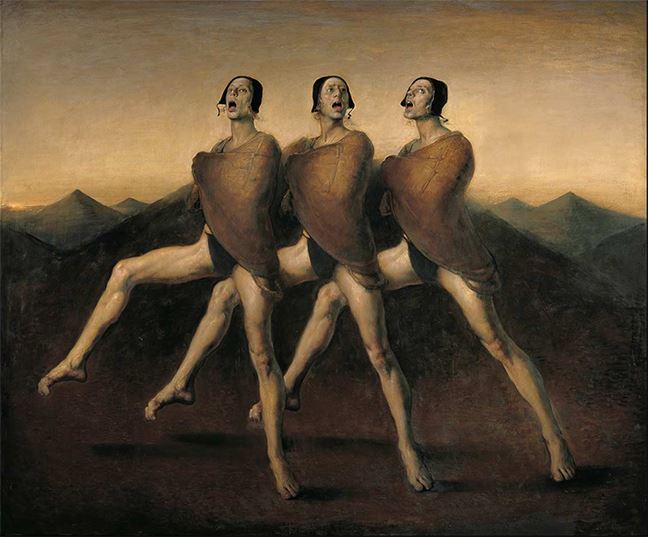}
\caption{Odd Nerdrum, \textit{Three Singers}, Oil on Canvas}
\end{figure}
\vspace*{\fill}

\pagebreak

\tableofcontents

\pagebreak
\part*{Introduction}
\addcontentsline{toc}{part}{Introduction}

Let $k$ be a commutative ring. Given a flat affine group scheme $G$ acting on a finitely generated $k$--algebra $A$, one asks if the ring of invariants $A^G$ is always finitely generated over $k$. This is known as \textit{Hilbert's fourteenth problem}, and has been extensively studied. Indeed, the answer to the above question is negative without any restrictions on $G$ or $k$ - in particular, Nagata \cite{Nag59} (see Example \ref{Ex: Nagata}) showed that the additive group scheme $\mathbb{G}_a$ gives us a counterexample. 

Such a $G$ with the property that for all finitely generated $G$--algebras $A$, the ring of invariants $A^G$ is finitely generated over $k$, will be said to have the \textit{finite generation (FG)} property. In spite of Nagata's counterexample, Franjou and van der Kallen \cite{FvdK10,vdK15} showed that when the base ring $k$ is Noetherian, two large classes of these group schemes satisfy the (FG) property, the \textit{Chevalley group schemes} and the \textit{finite group schemes}.

More generally, for $G$ and $A$ as above, one can ask whether the entire graded cohomology algebra $H^*(G, A)$ is finitely generated. Similarly, such a $G$ with the property that for all finitely generated $G$--algebras $A$, the graded cohomology $H^*(G, A)$ is finitely generated over $k$, will be said to have the \textit{cohomological finite generation (CFG)} property. Of course, (CFG) is a stronger property than (FG), since $A^G=H^0(G,A)$, which is a quotient of $H^\ast(G,A)$, and thus: 
\[  \left\{ \begin{array}{l}
          \mbox{Group schemes over $k$}\\
         \mbox{satisfying (CFG)}.\end{array} \right\} \subseteq \left\{ \begin{array}{l}
          \mbox{Group schemes over $k$}\\
         \mbox{satisfying (FG)}.\end{array} \right\}.  \] 
Once again, the counterexample of Nagata shows us that one needs to place some restrictions on the group scheme $G$ for the (CFG) property to hold, even over a field. On the other hand, one immediately sees that some finiteness assumptions on our base ring $k$ are also needed (see \cite[Section 4.2]{BenII}):

   \textit{ Let $C_p$ be the cyclic group with $p$-elements, $k=\mathbb{Z}\rtimes (\mathbb{Z}/p)^\infty$ and $A=\mathbb{Z}\times (\mathbb{Z}/p)^\infty$ with trivial $C_p$ action. In this case,  
   \[
   H^1(C_p,A)=\mathrm{Hom}_{\mathbb{Z}}(C_p,A)=(\mathbb{Z}/p)^\infty
   \]
   which is not  finitely generated over $k$. Hence $H^\ast(C_p,A)$ cannot be a finitely generated $k$--algebra. }

In fact, it is a remarkably arduous task to determine whether a given group scheme $G$ has the (CFG) property. Notably, Evens \cite{Eve61} has shown that this property holds for \textit{finite groups} over a Noetherian base ring. Later, Friedlander and Suslin \cite{FS97} extended the (CFG) property to \textit{finite group schemes} over a field, and a further generalization was given by Touzé and van der Kallen \cite{TvdK10} for \textit{reductive linear algebraic group schemes} over a field. Indeed, over a field Touz\'e and van der Kallen even showed that (FG) and (CFG) are equivalent properties. Hence, it is natural to ask how general our group scheme $G$ can be in order to have the (CFG) property.

Recently and most generally, a striking result of van der Kallen \cite{vdK23} gives us a positive answer in the case of finite group schemes over any Noetherian ring: 

\begin{thmx}\label{Thm A}
    Let $G$ be a finite flat group scheme over a commutative Noetherian ring $k$. Then $G$ has the (CFG) property.
\end{thmx}

These notes provide a gentle introduction to the subject, and the main aim is to give a panoramic overview of van der Kallen's proof of cohomological finite generation for finite group schemes over an arbitrary Noetherian base. The only originality we claim is in the mistakes.

Let us now give the outline of this document, as well as highlight the main steps in the proof of van der Kallen's theorem. Part I of this document contains the general theory we will need to cover to even approach the above theorem. In Part II  we begin the actual proof of van der Kallen's theorem, which requires many intermediate results, each comprising their own section.

Let $G$ be a finite flat group scheme over $k$, and fix a finitely generated $G$--algebra $A$, where the base $k$ is Noetherian. In Section \ref{section: reduction lemma} we show that one can reduce the (CFG) problem to the general linear group $\GL_n$ as follows: First, the  Embedding Lemma \ref{embedding of G}, tells us that we can consider $G$ as closed subgroup of $\GL_n$ for some $n$, and then we may use the Reduction Lemma \ref{vdK's Reduction Lemma} to obtain that \[H^\ast(\GL_n,\mathrm{ind}_G^{\GL_n}A)\cong H^\ast(G,A)\] and since $\mathrm{ind}_G^{\GL_n}A$ is a finitely generated $\GL_n$--algebra, we obtain the desired reduction. 

So now van der Kallen's theorem would be proved if we could show (CFG) for $\GL_n$ over a Noetherian base, but it is not quite that simple. We do not actually prove full blown (CFG) for $\GL_n$ over an arbitrary Noetherian base, but rather a provisional form of (CFG) which gives us a criteria for detecting the finite generation of the cohomology ring via the existence of a uniform bound on its torsion as an abelian group. This provisional (CFG) theorem appears in Section \ref{section:provisionalCFG} as Theorem \ref{torsion}, and the precise formulation is as follows:

\begin{thmx}\label{Thm B}
    Let $k$ be a Noetherian ring, $G'$ be a Chevalley group scheme over $k$, and $A$ be a $G'$--algebra over $k$ which is finitely generated over $k$. Then $H^\ast(G',A)$ is a finitely generated algebra if and only if the underlying abelian group $H^\ast(G',A)$ has bounded $\mathbb{Z}$--torsion.
\end{thmx}

Thus, the (CFG) property for $G$ has been reduced to showing that $H^\ast(\GL_n,\mathrm{ind}_G^{\GL_n}A)\cong H^\ast(G,A)$ has bounded torsion when considered as an abelian group. Showing the existence of a uniform bound on its torsion is the aim of Section \ref{section:boundedtorsion}, and corresponds to the main result of that section, Theorem \ref{bounded torsion}:

\begin{thmx}\label{Thm C}
    Let $G$ be a finite group scheme over a field, and $A$ be a finitely generated $G$--algebra. Then $H^\ast(G,A)$ has bounded torsion. 
\end{thmx}

These ingredients now essentially complete the proof of Theorem \ref{Thm A}! In fact, in Section \ref{sec: alternative bounded}, we present an alternative proof of this theorem that is due to Eike Lau and appears in \cite{Lau2023}. 

The hard part in this story is really in proving Theorem \ref{Thm B}, so let us say a bit about its proof: Sections \ref{section: FG}, \ref{section:GrosshansFiltrations}, and \ref{section:CFGforGL} establish the intermediate results needed to approach this theorem, in particular, the main ingredient in Theorem \ref{Thm B} is proving (CFG) for $\GL_n$ over a Noetherian base $k$ which contains a field, as seen in Theorem \ref{CFG for GLn}:

\begin{thmx}\label{Thm D}
    Let $k$ be a commutative Noetherian ring containing a field. Then $\GL_n$ over $k$  satisfies (CFG).  
\end{thmx}

To prove this result, in Section \ref{section: FG} we crucially establish (FG) for $\GL_n$ over a Noetherian base, and in Section \ref{section:GrosshansFiltrations} we set up the main technical tools needed for Theorem \ref{Thm D}. To proceed, we must recall for ourselves Friedlander and Suslin's result \cite{FS97}:

\begin{thmx}\label{Thm E}
    Finite group schemes over a field satisfy (CFG). 
\end{thmx}

In reality, we do not use Friedlander and Suslin's  result as stated above, but rather a stronger version of it. A precise formulation appears later as Theorem \ref{FS full strength}. 

As we already mentioned, Touzé and van der Kallen \cite{TvdK10} proved that (CFG) and (FG) are equivalent over a field. In particular, the main ingredient in their proof is the existence of \textit{universal classes in cohomology} \cite{Tou10}. Let us stress that their method was not extended to prove Theorem \ref{Thm D}; instead, a different approach was proposed by van der Kallen. The following theorem is discussed in Section \ref{section:GrosshansFiltrations}:

\begin{thmx}\label{Thm F}
    Let $k$ be a Noetherian ring.  Suppose that there is a morphism of $\GL_n$--algebras $f\colon B\to A$ such that both $B$ and $A$ are  finitely generated $k$--algebras,
         $B$ has a good filtration, and
         $f$ makes $A$ a Noetherian $B$--module.
     Then the following properties hold.
\begin{itemize}
    \item[$(a)$] $H^i(\GL_n,A)=0$ for $i$ large enough. 
    \item[$(b)$] $H^i(\GL_n,A)$ is a Noetherian $A^G$--module for every $i$.
\end{itemize}
In particular, this implies that $H^\ast(\GL_n,A)$ is a finitely generated $k$--algebra.
\end{thmx}

This is, in fact, the key step to adapting the proof of (CFG) for $\GL_n$ over a field given in \cite{TvdK10} to the case of a Noetherian base ring containing a field. The rough idea to prove Theorem \ref{Thm D} is as follows: First, let $A$ be a finitely generated $\GL_n$--algebra, one considers the so-called \textit{Grosshans filtration} on $A$, and with respect to this filtration, one proceeds by analysis on the associated spectral sequence 
\[
E_1^{s,t}(A)=H^{s+t}(\GL_n,\mathrm{gr}_{-s}A)\Rightarrow H^{s+t}(\GL_n,A).
\]
The abstract result on spectral sequences that we need in order to show that $H^\ast(\GL_n,A)$ is finitely generated is Lemma \ref{Lemma analysis of ss} due to Evens. In particular, one needs to verify that  $E^{\ast,\ast}_1(A)$ is a finitely generated $k$--algebra and that and that the spectral sequence collapses at a finite stage. Showing the collapse of this spectral sequence involves the \textit{full strength} of Theorem \ref{Thm E} together with the existence of universal classes of Touz\'e \cite{Tou10}. And to show that $E^{\ast,\ast}_1(A)$ is a finitely generated $k$--algebra, one shows first that $H^\ast(\GL_n,\mathrm{gr}A)$  is a finitely generated $k$--algebra, where $\mathrm{gr}A$ denotes associated graded with respect to the Grosshans filtration of $A$, this is done in Theorem \ref{cohomology of grA is finitely generated}. This itself requires another layer of analysis on the spectral sequence 
\[
{}^{LHS}E^{s,t}_2(\mathrm{gr}A)=H^s(\GL_n/G_r,H^t(G_r,\mathrm{gr}A))\Rightarrow H^{s+t}(\GL_n,\mathrm{gr}A)
\] associated to a high enough $r$th Frobenius kernel $G_r$ of $\GL_n$. The collapse on a finite stage of this spectral sequence is a consequence of Theorem \ref{Thm F}, the rest of the proof uses again  the \textit{full strength} of Theorem \ref{Thm E}.  

In particular, we have shown that the proof of Theorem \ref{Thm A} is independent from the classic result of Evens, and so we recover it as a corollary: 

\begin{thmx}\label{Thm G}
    Finite groups satisfy (CFG) over commutative Noetherian rings. 
\end{thmx}

Let us conclude with a flowchart roughly containing the logical dependence of the aforementioned results:\\

\begin{figure}[h]
\centering
\begin{tikzpicture}[node distance=2cm]
\node (A1) [box2] { Analysis of LHS-spectral sequence, \\ Grosshans filtration and properties (Section \ref{section:GrosshansFiltrations}), and Theorems \ref{Thm E} and \ref{Thm F}.};

\node (A2) [box2, below of=A1] {(FG) for $\GL_n$ over a Noetherian base (Corollary \ref{FG for Chevalley}), Existence of universal classes (Theorem \ref{Thm Touze}), \\  and Theorems \ref{Thm E} and \ref{Thm F}.};

\node (A3) [box2, below of=A2] {(FG) for $\GL_n$ over a Noetherian base (Corollary \ref{FG for Chevalley}), and Proposition \ref{Prop: power surjectivity on cohomology at a prime}.};

\node (A4) [box2, below of=A3] {Reduction and Embedding Lemma (Section \ref{section: reduction lemma}), and Theorem \ref{Thm C}.
};

\node (B1) [box1, right of=A1, xshift=6.5cm] {$\forall$ $\GL_n$--alg $A$, $H^\ast(\GL_n,\mathrm{gr}A)$ is f.g over $k$  \\ (Theorem \ref{cohomology of grA is finitely generated}) };

\node (B2) [box1, below of=B1] {(CFG) for $\GL_n$ over a Noetherian ring containing a field\\ (Theorem \ref{Thm D}). };

\node (B3) [box1, below of=B2] {$H^\ast(\GL_n,A)$ is a f.g. $k$--alg. $\Leftrightarrow$ it has bounded $\mathbb{Z}$--torsion \\ (Theorem \ref{Thm B}).};

\node (B4) [box1, below of=B3] {Finite group schemes over a Noetherian ring satisfy (CFG) \\ (Theorem \ref{Thm A}).};

\node (B5) [box1, below of=B4] {Finite groups over a Noetherian base satisfy (CFG) \\ (Theorem \ref{Thm G}).};

\draw [arrow] (B1) -- (B2);
\draw [arrow] (B2) -- (B3);
\draw [arrow] (B3) -- (B4);
\draw [arrow] (B4) -- (B5);
\draw [arrow] (A1) -- (B1);
\draw [arrow] (A2) -- (B2);
\draw [arrow] (A3) -- (B3);
\draw [arrow] (A4) -- (B4);

\end{tikzpicture}
\end{figure}

 \subsection*{Further directions and applications}
As mentioned above, over a field the cohomological finite generation property is equivalent to the finite generation of invariants property. In particular, this equivalence holds for reductive algebraic groups. It is therefore natural to ask whether the reductive algebraic groups over a Noetherian base satisfy the cohomological finite generation property. Antoine Touzé has conjectured that this is indeed the case, and there is supporting evidence: for example, the conjecture has been verified for the groups $\mathrm{SL}_2$ and $\mathrm{SL}_3$ (see \cite[Section 4]{AiM}).

\begin{conjx}[Touz\'e]
    Let $G$ be a reductive  algebraic group scheme over a commutative Noetherian ring $k$. Then $G$ satisfies  the cohomological finite generation property.
\end{conjx}

Let us now highlight some applications of the cohomological finite generation property. Bounded derived categories of rational representations of finite group schemes are arguably one of the most illuminating examples in tensor-triangular geometry, making them a rich subject of study. Of special interest is  the bounded derived category $\mathbf{D}^b(\mathrm{lat}(G,k))$ of finitely generated lattices, together with its Ind-completion $\mathbf{Rep}(G,k)$, which can be realized as a subcategory of the homotopy category of projective representations. The monoidal structure is given by tensoring over the base ring, with the diagonal group scheme action induced by the Hopf algebra structure of its coordinate algebra.

Recently, Barthel–Benson–Iyengar–Krause–Pevtsova \cite{BBIKP} established a classification of localizing tensor ideals in $\mathbf{Rep}(G,k)$ and of thick tensor ideals in $\mathbf{D}^b(\mathrm{lat}(G,k))$ in terms of subsets of the homogeneous spectrum of the cohomology ring $H^\ast(G,k)$. A key ingredient in their work is van der Kallen’s theorem, which ensures that $H^\ast(G,k)$ is Noetherian, thereby enabling the application of BIK-stratification \cite{BIK11}. A more tt-geometric proof of this classification was later obtained by the first-named author \cite{Gom25}, which again relies on the cohomological finite generation property for finite group schemes over Noetherian bases.

We also mention the forthcoming work of Eike Lau, announced in \cite{Lau2023}, where van der Kallen’s provisional cohomological finite generation theorem is used to prove that for an algebraic stack $X$ admitting a surjective finite flat morphism $Y \to X$ with $Y$ affine, the cohomology ring $H^\ast(X,\mathcal{O}_X)$ is Noetherian. This result serves as a crucial input in Lau’s classification of thick tensor ideals of $\mathbf{Perf}(X)$, the category of perfect complexes on $X$. In the special case where $X = BG = [\mathrm{Spec}(k)/G]$ for a finite group or a finite group scheme $G$ over a Noetherian ring $k$, the category $\mathbf{Perf}(BG)$ coincides with $\mathbf{D}^b(\mathrm{lat}(G,k))$.

\subsection*{Notation and conventions}
 Let us now highlight some of the most frequently used notation, terminology, and conventions in this text: 
 
 \begin{enumerate} 
     \item All rings in this document are unital.  We use $k$ to denote a commutative ring. All algebras over $k$ are assumed to be commutative.
     \item When the context is clear, we suppress $k$ from the notation. For instance, we write $\otimes$ instead of $\otimes_k$.
     \item We freely use standard terminology from commutative algebra. In particular, we write $\Spec(-)$ to denote the Zariski spectrum functor.
     \item  All group schemes are assumed to be flat and affine, except in the first part of the document, where we allow more generality.
     \item By default, we work with right comodules and left modules, referring to them simply as comodules and modules, respectively. When left comodules or right modules are involved, we indicate this explicitly. 
     \item For a Hopf algebra $H$ and an $H$--module $M$, we denote the invariants by $M^H$. For an $H$--comodule $N$, the coinvariants are denoted by $N^{\text{co}H}$. If $H = k[G]$ for a group scheme $G$, we also write $N^G$ for $N^{\text{co}k[G]}$ and refer to them as invariants. While potentially confusing, this is justified since, in practice, $k[G]$--comodules and $G$--representations are used interchangeably.  
     \item Categories are denoted in bold. For example, the category of groups is written as $\textbf{Grps}$; the category of left (resp. right) $k$--modules is denoted by $_k\textbf{Mod}$ (resp. $\textbf{Mod}_k$); and the category of affine schemes over $k$ is written as $\textbf{AffSch}_k$.
    \end{enumerate}

\subsection*{Acknowledgments} 
We are deeply grateful to Eike Lau for explaining to us his proof on bounded torsion, and to Antoine Touz\'e for  helpful comments on an earlier version of this manuscript. We would also like to thank the organizers of the master class \textit{New Developments in Finite Generation of Cohomology}, and especially Henning Krause and Julia Pevtsova for their encouragement in writing these notes. While preparing them, both  JOG and CJP were   supported   by the Deutsche Forschungsgemeinschaft (Project-ID 491392403 – TRR 358).

\pagebreak

\part*{Part I - General Theory}\label{part I}
\addcontentsline{toc}{part}{Part I - General Theory}

\section{Outline of the First Part}

In this part we cover some baseline background material needed for the rest of the document. We introduce basic notions on group schemes, coalgebras, Hopf algebras and their relations. In particular, for a finite group scheme $G$ over a commutative ring $k$, we  make explicit the relation between comodules over  $k[G]$, modules over $k[G]^\ast$, and $G$--representations, and define invariants and (rational) cohomology right after.  In the last section, we provide a quick overview of costandard modules, primarily with the intention of reusing several of the notions and terminology in later sections.

It is worth highlighting that we aim to keep the first part as general as possible, indicating where restrictions on the group schemes, algebras, or base ring are necessary.

Throughout this notes, we write $k$ to denote a commutative ring, and $\textbf{Mod}_k$ to denote the category of $k$--modules equipped with the symmetric monoidal structure given by the tensor product $\otimes_k$ of $k$--modules and monoidal unit $k$. In general, we simply write $\otimes$ and $\mathrm{Hom}(-,-)$ to denote $\otimes_k$ and $\mathrm{Hom}_k(-,-)$, respectively, unless there might be confusion.

\section{Comodules and Coalgebras}

\begin{Def}
    An \textit{associative unital $k$--algebra $A$} is an associative and unital monoid object $(A,m,\eta)$ in $\textbf{Mod}_k$. Explicitly, $(A,m,\eta)$ consist of  a $k$--module $A$ together with a \textit{multiplication} $m\colon A\otimes A\to A$ and a \textit{unit map} $\eta\colon k\to A$, both $k$--linear,  satisfying the \textit{associative axiom} $m\circ (m\otimes\mathrm{id}_A)=m\circ(\mathrm{id}_A\otimes m)$ and \textit{right and left unital axioms} $m\circ (\eta \otimes \mathrm{id}_A)=\mathrm{id}_A=m\circ(\mathrm{id}_A\otimes \eta)$ which amounts to the commutativity of the following diagrams 
     \begin{center}
        \begin{tikzcd}
            &A\otimes A\otimes A \arrow[r, "m\otimes \mathrm{id}_A"]\arrow[d, "\mathrm{id}_A\otimes m"']&A\otimes A\arrow[d, "m"] &A\otimes k \arrow[d, "\mathrm{id}_A\otimes \eta"']\arrow[r, "\eta\otimes\mathrm{id}_A"] \arrow[dr, equals]&A\otimes A \arrow[d, "m"] \\
            &A\otimes A \arrow[r, "m"']&A &A\otimes A \arrow[r, "m"'] &A. 
        \end{tikzcd}
    \end{center}
Moreover, we say that $A$ is \textit{commutative} if the following diagram is commutative
\begin{center}
        \begin{tikzcd}
            A\otimes A \arrow[rr,"(12)"] \arrow[rd,"m"'] & & A\otimes A \arrow[dl,"m"]\\ & A &
        \end{tikzcd}
    \end{center}
    where $(12)$ means swapping the factors, that is, $(12)(a\otimes a')=a'\otimes a$.  
\end{Def}

\begin{Ex}
    Let $k=\mathbb{Z}$. In this case, an associative unital $\mathbb{Z}$--algebra is simply and associative ring with unit.
\end{Ex}

\begin{Def}
    An \textit{associative unital $k$--coalgebra $C$} is an associative unital monoid $(C,\Delta, \epsilon)$ in the category $\textbf{Mod}_k^\textrm{op}$. Unpacking the definition, the monoid $(C,\Delta,\epsilon)$ consists of a $k$--module $C$ together with two linear maps $\Delta\colon C\to C\otimes C$ and $C\to k$ called the \textit{comultiplication} and \textit{counit}, respectively, making the following diagrams commutative 
    
        \begin{center}
        \begin{tikzcd}
            &C \arrow[d, "\Delta"']\arrow[r, "\Delta"]&C\otimes C\arrow[d, "\Delta\otimes \id_C"] &C \arrow[d, "\Delta"']\arrow[r, "\Delta"] \arrow[dr, equals]&C\otimes C \arrow[d, "\id_C\otimes \varepsilon"] \\
            &C\otimes C \arrow[r, "\id_C \otimes \Delta"']&C\otimes C \otimes C &C\otimes C \arrow[r, "\varepsilon \otimes \id_C"'] &C\otimes k. 
        \end{tikzcd}
    \end{center}
    In this case, we refer to the left square as  the  \textit{coassociative axiom}  and to the right diagrams as the  \textit{right and left counital axioms}. We say that an associative unital $k$--algebra $C$ is \textit{cocommutative} if the following diagram is commutative 
    \begin{center}
        \begin{tikzcd}
           & C \arrow[ld,"\Delta"'] \arrow[rd,"\Delta"] & \\ C\otimes C \arrow[rr,"(12)"']  & & C\otimes C.
        \end{tikzcd}
    \end{center}
\end{Def}

\begin{convention}
We will work mainly with associative and unital $k$--(co)algebras, hence from now on, we refer to them simply as $k$--(co)algebras.     
\end{convention}

\begin{Ex}\label{dual of a coalgebra}
Let $A$ be a $k$--algebra with multiplication $m:A\otimes A\to A$ and unit $\eta: k\to A$, and let $C$ be a $k$--coalgebra with comultiplication $\Delta:C\to C\otimes C$ and counit $\varepsilon:C\to k$. Then $\Hom_k(C, A)$ can be given the structure of a $k$--algebra as follows: let $f, g\in \Hom_k(C, A)$, and let $x\in C$ and write $\Delta(x)=\sum_i a_{i}\otimes b_{i}$. We define the \textit{convolution product} $f\ast g$ of $f$ and $g$ by the composite
\[
C\xrightarrow[]{\Delta} C\otimes C \xrightarrow[]{f\otimes g} A\otimes A\xrightarrow[]{m} A
\]
in elements, this is given by  
\[
(f \ast g)(x)=\sum_{i=i}^l m(f(a_i)\otimes g(b_i))
\]
    One checks that this product is associative using the associativity of the multiplication on $A$ and the coassociativity of the comultiplication on $C$. Moreover, it follows from the diagrams in the definition of an algebra and coalgebra that composite map $C\xrightarrow{\varepsilon}{} k \xrightarrow{\eta}{}A$ is a left and right identity element for this product operation. That is, 
    \[
    f\ast (\eta\circ\varepsilon) = (\eta\circ\varepsilon) \ast f = f.
    \]
    Moreover, if $C$ is cocommutative and $A$ is commutative, then $\Hom_k(C,A)$ is commutative. 
    
    In particular, the dual of a  $k$--coalgebra $C^*=\Hom_k(C, k)$ always can be given an algebra structure via the convolution product, moreover, the dual of the unit $\eta^*:C^*\to k^*$ composed with the natural isomorphism $k^*\simeq k$ gives an augmentation map $C^* \to k$, which is explicitly defined by $f\mapsto f(1)$, making $A^\ast$ into an augmented $k$--algebra. Recall that an \textit{augmented $k$--algebra} is simply a $k$--algebra $A$ together with $k$--linear map $A\to k$ which is called \textit{augmentation map}.
\end{Ex}

    \begin{Rem}
    It is not true in general that the dual of an associative unital $k$--algebra $A$ can be given the structure of a $k$--coalgebra. The key point is that the dual of the multiplication map $m:A\otimes A\to A$ may not induce a comultiplication $m^*: A^*\to A^*\otimes A^*$. However,  when $A$ is finitely generated and projective over $k$,  there is an  isomorphism of $k$--modules
    \[
    A^*\otimes A^* \xrightarrow[]{\varphi} (A\otimes A)^* \mbox{ given by } \varphi(f\otimes g)(\sum_i a_i\otimes b_i)=\sum_i f(a_i)g(b_i)
    \]
       which together with $m^*$ does indeed induce a comultiplication. In other words, $m^*(f)(a\otimes b)=f(ab)$ and the augmentation map defined above becomes the counit, making $A^*$ into a coalgebra. Moreover, $A^\ast$ is cocommutative provided that $m$ is commutative. 
\end{Rem}

\begin{Ex}\label{tensor coalgebra}
    Let $(C_1,\Delta_1,\varepsilon_1)$ and $(C_2,\Delta_2,\varepsilon_2)$ be $k$--coalgebras. We can equip the tensor product $C_1\otimes C_2$ with an structure of $k$--coalgebra as follows. Consider the map $\Delta_{1,2}$ given as the composition 
    \[
    C_1\otimes C_2 \xrightarrow[]{\Delta_1\otimes \Delta_2} C_1\otimes C_1\otimes C_2 \otimes C_2 \xrightarrow[]{\mathrm{id}_{C_1}\otimes (12)\otimes \mathrm{id}_{C_2}} C_1\otimes C_2 \otimes C_1\otimes C_2.
    \]
    It is an interesting exercise to verify that indeed $(C_1\otimes C_2,\Delta_{1,2},\varepsilon_{1,2})$ is a $k$--coalgebra with $\varepsilon_{1,2}=\varepsilon_1\otimes \varepsilon_2$. In this case, we refer to $C_1\otimes C_2$ as the tensor product of the coalgebras $C_1$ and $C_2$. 
\end{Ex}

\begin{Def}
    Let $(C_1,\Delta_1,\varepsilon_1)$ and $(C_2,\Delta_2,\varepsilon_2)$ be $k$--coalgebras. A \textit{a morphism of coalgebras from $C_1$ to $C_2$} is a $k$--linear map $f\colon C_1\to C_2$ such that the following diagrams are commutative 
    \begin{center}
        \begin{tikzcd}
           C_1 \arrow[r,"f"] \arrow [d,"\Delta_1"'] & C_2 \arrow[d,"\Delta_2"] & C_1\arrow[r,"f"] \arrow[rd,"\varepsilon_1"'] & C_2 \arrow[d,"\varepsilon_2"] \\
           C_1\otimes C_1 \arrow[r,"f\otimes f"'] & C_2\otimes C_2 &  & k.
        \end{tikzcd}
    \end{center}
    A morphism $f\colon A_1\to A_2$ of algebras is defined in a similar fashion. 
\end{Def}

It is easy to verify that the composition of  morphisms of coalgebras is itself a morphisms of coalgebras. In particular, one obtains that the coalgebras and morphisms of coalgebras  can be organized into a category which we will call the \textit{category of $k$--coalgebras} and denote it by $\textbf{CoAlg}_k$. In the same fashion, we can define the category of $k$--algebras $\mathbf{Alg}_k$. 

\begin{Prop}
    Let $A$ be a $k$--algebra. Then we obtain a contravariant functor 
    \[
    \Hom_k(-,A)\colon \mathbf{CoAlg}_k\to \mathbf{Alg}_k
    \]
    where we endow $\Hom_k(C,A)$ with the algebra structure given by the convolution product (see Example \ref{dual of a coalgebra}). 
\end{Prop}

\begin{proof}
    We only need to verify that for a morphism $f\colon C_1\to C_2$ of coalgebras, the map 
    \[
    f^\ast:=\Hom(f,A)\colon \Hom(C_2,A)\to \Hom(C_1,A)
    \]
    is a morphisms of algebras. Let $g,g'\in \Hom(C_2,A)$. We have the following commutative diagram 
     \begin{center}
        \begin{tikzcd}
           C_1 \arrow[r,"f"] \arrow [d,"\Delta_1"'] & C_2 \arrow[d,"\Delta_2"] \arrow[r,"\Delta_2"] & C_2\otimes C_2 \arrow[r,"g\otimes g'"] \arrow[d,equal] & A\otimes A \arrow[r] \arrow[d,equal] & A \arrow[d,equal] \\
           C_1\otimes C_1 \arrow[r,"f\otimes f"'] & C_2\otimes C_2 \arrow[r,equal] & C_2\otimes C_2 \arrow[r,"g\otimes g'"'] & A\otimes A \arrow[r] & A
        \end{tikzcd}
    \end{center}
    It follows that $f^\ast(g\ast g')=f^\ast(g)\ast f^\ast(g')$. The rest is easy. 
\end{proof}

\begin{Def}
    Let $(C, \Delta, \varepsilon)$ be a $k$--coalgebra. A \textit{(right) $C$--comodule} is a $k$--module $V$ equipped with a $k$--linear map $\Delta_V\colon  V\to V\otimes_k C$ called the \textit{coaction}, such that the following diagrams commute:
    \begin{center}
        \begin{tikzcd}
            &V \arrow[d, "\Delta_V"']\arrow[r, "\Delta_V"]&V\otimes C\arrow[d, "\Delta_V\otimes \id_C"] &V \arrow[r, "\Delta_V"] \arrow[dr, equals]&V\otimes C \arrow[d, "\id_V\otimes \varepsilon"] \\
            &V\otimes C \arrow[r, "\id_V \otimes \Delta"']&V\otimes C \otimes C & &V\otimes k. 
        \end{tikzcd}
    \end{center}

    Let $(W, \Delta_W)$ be another $C$--comodule, a $k$--linear map $f\colon V\to W$ is called a $C$-\textit{comodule homomorphism} if the following diagram commutes:
    \begin{center}
        \begin{tikzcd}
            &V \arrow[r, "\Delta_V"]\arrow[d, "f"']&V\otimes C \arrow[d, "f\otimes \id_C"] \\
            &W\arrow[r, "\Delta_W"'] &W\otimes C .
        \end{tikzcd}
    \end{center}

We will denote the set of $C$--comodule homomorphisms from $(V, \Delta_V)$ to $(W, \Delta_W)$ by $\Hom^C(V, W)$. Furthermore, we will denote the category of $C$--comodules and $C$--comodule homomorphisms by $\textbf{coMod}_C$. In a similar way, one defines the category of left $C$--comodules and homomorphisms, it will be denoted by  ${}_{C}\textbf{coMod}$. 

Furthermore, let $(V, \Delta_V)$ be any $C$--comodule. If $W\subseteq V$ is a $k$--submodule of $V$, we will say that it is a \textit{subcomodule} if the inclusion map is a $C$--comodule homomorphism. 
\end{Def}

\begin{convention}
   In practice, one often suppresses the coaction in the notation of a comodule $(V, \Delta_V)$ and simply says that $V$ is a $C$--comodule, with the coaction $\Delta_V$ understood.
\end{convention}

\begin{Rem}
   Of course, there are dual notions of right and left modules over an algebra $A$ and morphisms of $A$--modules; we leave the details to the reader. In particular, we write $\Hom_A(M, N)$ to denote the set of $A$--module homomorphisms from $M$ to $N$, and write $\mathbf{Mod}_A$  (resp. ${}_A\mathbf{Mod}$) to denote the category of right (resp. left)   $A$--modules   together with homomorphisms of $A$--modules.
\end{Rem}

There is the obvious forgetful functor from the category of $C$--comodules to the category of $k$--modules, the next proposition shows that this forgetful functor is part of an adjoint pair.

\begin{Th}\label{adjunction}
    Let $C$ be a $k$--coalgebra and $W$ be a $k$--module. Then the map $\id_W \otimes \Delta$ turns $W\otimes_k C$ into  a $C$--comodule. Moreover, let $(V, \Delta_V)$ be a $C$--comodule, then we have an isomorphism 
    \[
    \Hom_k(V, W)\simeq \Hom^C(V, W\otimes_k C)
    \]
    which is natural in $V$ and $W$. In other words, the forgetful functor is left adjoint to 
    \[
    -\otimes_k C\colon \mathbf{Mod}_k\to  \mathbf{coMod}_C.
    \]
\end{Th}

\begin{proof}
    The first claim is an easy exercise. For the second claim, let $f\in \Hom^C(V, W\otimes_k C)$, then the composite 
    \[
    f' \colon\quad V\xrightarrow[]{f} W\otimes_k C \xrightarrow[]{\id_W \otimes \varepsilon} W\otimes_k k=W
    \]
    defines a map of $k$--modules $f'\in \Hom_{k}(V, W)$. Moreover, given a $k$--linear map  $g'\colon V\to W$ the composition  
    \[
    g\colon \quad V\xrightarrow[]{\Delta_V} V\otimes_k C\xrightarrow[]{g'\otimes \id_C} W\otimes_k C
    \]
    defines a morphisms of $C$--comodules. One checks easily that these assignments are mutually inverse to each other, and that they are natural in $V$ and $W$. Therefore, we obtain a natural isomorphism between  $\Hom^C(V, W\otimes_k C)$ and  $\Hom_k(V, W)$ as we wanted. 
\end{proof}

\begin{Cor}
    Let $C$ be a $k$--coalgebra. Then $\End^C(C)\simeq C^*$.
\end{Cor}

\begin{proof}
    Using the above adjunction and taking $V=C$ and $W=k$ we get $\Hom^C(C, C)\simeq \Hom_k(C, k)$ as desired.
\end{proof}

\begin{Prop}\label{kernels}
    Let ${C}$ be a $k$--coalgebra. Then the category \emph{\textbf{coMod}}$_C$ has coproducts and cokernels. Moreover, if $C$ is flat as a $k$--module, then \emph{\textbf{coMod}}$_C$ has kernels.
\end{Prop}

\begin{proof}
    Let $\{(V_i, \Delta_{V_i}) \}_{i\in I}$ be a family of $C$--comodules. Then $\bigoplus_i V_i$ has the structure of a $C$--comodule via the map 
    \[
    \bigoplus_i V_i  \xrightarrow{\bigoplus_i \Delta_{V_i}}\bigoplus_i\left(V_i\otimes C\right) \xrightarrow{\sim} \left(\bigoplus_i V_i\right) \otimes C.
    \]
    One also easily checks that it satisfies the universal property of a coproduct. 
    
    Now let $f\colon V\to W$ be a map of $C$--comodules, it is routine to see that the cokernel of this map in $\textbf{Mod}_k$ inherits the structure of a $C$--comodule and the required universal property of a cokernel. 
    
    For kernels, we need to check that the coaction $\Delta_V$ on $V$ restricts to the kernel of $f$ in ${\textbf{{Mod}}}_k$, that is, that $\Delta_V|_{\ker f}$ lands in $\ker f\otimes \id_C$. Let $x\in \ker f$, and let $\Delta_V(x)=\sum_i x_i \otimes c_i$, we want to check that $x_i\in \ker f$ whenever $C$ is flat. Pushing $x$ around the commutative square
       \begin{center}
        \begin{tikzcd}
            &V \arrow[r, "\Delta_V"]\arrow[d, "f"']&V\otimes C \arrow[d, "f\otimes \id_C"] \\
            &W\arrow[r, "\Delta_W"'] &W\otimes C .
        \end{tikzcd}
    \end{center}
    we get that  
    \[
    \sum f(x_i)\otimes c_i=0 \mbox{, that is, } \sum_i x_i\otimes c_i\in \ker(f\otimes \id_C). 
    \]
     In particular, whenever $C$ is flat, we have that $\ker(f\otimes \id_C)= \ker f\otimes C$ by tensoring the left exact sequence $0\to\ker f\to M\to N$ everywhere with $C$. Thus we have that $x_i\in \ker f$ as desired. The remaining properties of the kernel are easily checked.
\end{proof}

\begin{Th}\label{abelian}
    Let $C$ be a $k$--coalgebra which is flat as a $k$--module. Then \emph{\textbf{coMod}}$_C$ is an abelian category. 
\end{Th}

\begin{proof}
    We have seen that whenever $C$ is flat, the category $\textbf{coMod}_C$ has coproducts, cokernels, and kernels. Moreover, it is easy to see that this category has finite biproducts and is enriched over abelian groups.

    To see that this category is abelian it remains to check that the category is normal, that is, that every monomorphism is a kernel and every epimorphism is a cokernel. Whenever $C$ is flat we can see that the forgetful functor respects monomorphisms. Thus we see that monomorphisms in $\textbf{coMod}_C$ are kernels, since if $f:V\to W$ is a monomorphism in $\textbf{coMod}_C$, then it is also a monomorphism of $k$--modules, and in particular it is the kernel of the map $g:W\to W/\im f$. Moreover, $W/\im f$ inherits the comodule structure from $W$ in a way such that $g$ is a comodule homomorphism. Thus $f$ is the kernel of $g$ in the category $\textbf{coMod}_C$. Similarly, the forgetful functor always respects epimorphisms since it is a left adjoint, and thus we see that epimorphisms in $\textbf{coMod}_C$ are cokernels.
\end{proof}

\begin{Th}\label{injectives}
    Let $C$ be a $k$--coalgebra which is flat as a $k$--module. Then the following properties hold. 
    \begin{enumerate}[(i)]
        \item If $J$ is an injective $k$--module, then $(J\otimes_k C, \id_J\otimes \Delta)$ is an injective $C$--comodule.
        \item The category \emph{\textbf{coMod}}$_C$ has enough injectives.
    \end{enumerate} 
\end{Th}

\begin{proof}
    To prove that $J\otimes_k C$ is injective we can show that $\Hom^C(-, J\otimes_k C)$ is an exact functor. Note that the forgetful functor $\mathbf{coMod}_C\to \mathbf{Mod}_k$ is exact, and Theorem \ref{adjunction} gives us a natural isomorphism $\Hom^C(-, J\otimes_k C)\simeq \Hom_k(-, J)$, and the latter functor is exact by assumption, so we are done. 
    
    Now will will show that $\textbf{coMod}_C$ has enough injectives. Let $(W, \Delta_W)$ be any $C$--comodule, then viewed as a $k$--module it embeds into an injective $k$--module $J_W$ via some $k$--linear monomorphism $i\colon W\to J_W$. Now, we claim that  the composite map 
    \[
    W\xrightarrow[]{\Delta_W} W\otimes_k C \xrightarrow[]{i\otimes \id_C} J_W\otimes_k C
    \]
    is injective. Indeed, it is a composite of injective maps. Note that the second map is injective because it is an injective map tensored with a flat $k$--module, and flat modules preserve injective maps when tensoring. Moreover, it is easy to see that the first map has trivial kernel.
\end{proof}

\begin{Rem}
    An alternative, and perhaps more enlightened, proof is possible here. In particular, whenever $C$ is flat as a $k$--module, the category $\textbf{coMod}_C$ is actually a Grothendieck abelian category, so that the existence of enough injectives follows immediately. The advantage of the proof given here is that we are actually showing how to construct these injectives explicitly.
\end{Rem}

\begin{Rem}\label{comodC embedding into modC*}
    Moving on, one sees that $\textbf{coMod}_C$ is always a subcategory of ${}_{C^*}\textbf{Mod}$. Given a $C$--comodule $(V, \Delta_V)$, it inherits the structure of a $C^*$--module via:
\[
C^* \otimes_k V \to V, \quad f\otimes x \mapsto (f \otimes \id_V)(v(x)).
\]
 In particular, one sees that a morphisms of 
$C$--comodules is in fact a morphism of $C$--modules with respect to the action defined above. 
Moreover, whenever $C$ is finitely generated and projective, the next theorem shows that these categories even agree.
\end{Rem}

\begin{Th}\label{equivalence}
    Let $C$ be a $k$--coalgebra which is finitely generated and projective as a $k$--module. Then there is an equivalence of categories 
    \[
    \mathbf{coMod}_C \simeq {}_{C^*}\mathbf{Mod}.
    \]
\end{Th}

\begin{proof}
    In the category $\textbf{Mod}_k$, the dualisable objects are precisely the finitely generated projective $k$--modules, thus we see that $C$ is a dualisable $k$--module. In other words, we have an adjunction:
    \[
    \Hom_k(C^*\otimes_k V, W) \simeq \Hom_k(V, W\otimes_k C ),
    \]
     where $V$ and $W$ are arbitrary $k$--modules. In particular, if $V=W$ we get
    \[
    \Hom_k(C^* \otimes_k V, V) \simeq \Hom_k(V, V\otimes_k C ).
    \]
     Now let $V$ be a $k$--module with the additional structure of a $C^*$--module, that is, we have an action map $\lambda\colon C^*\otimes_k V \to V$ turning $(V,\lambda)$ into a $C^\ast$--module. Then the image of $\lambda$ under this adjunction is some map $\lambda^\ast\colon  V\to V\otimes_k C$, which can be seen to satisfy the properties of a $C$--coaction. So what we have done is shown that a $C^*$--module can be given the structure of a $C$--comodule via $\lambda^\ast$. Conversely given a $C$--comodule, the image of the coaction map under this adjunction is also easily seen to give us the $C^*$--action map defined above. One can verify that these two associations are functorial and mutually inverse, giving us an equivalence of categories, so we are done.
\end{proof}

\begin{Rem}\label{projectives are injectives}
In many cases of interest, the category $\mathbf{coMod}_C$ also has enough projectives. Generally, projective and injective objects are quite distinct; however, when the coalgebra $C$ is finitely generated and projective over $k$, a strong relationship between projective and injective objects emerges provided that $C$ possesses additional structure as a \textit{commutative Hopf algebra}. In this case, it has been shown in \cite{Par71} that $C^\ast$ forms a \textit{quasi-Frobenius algebra} over $k$, as defined in \cite{Mul64}. Consequently, the regular representation $C^\ast$ is injective. Therefore, any $C^\ast$--module that is projective as a $k$--module is also injective as a $C^\ast$--module. For a detailed treatment when $k$ is a field, see \cite[Chapter 9]{DV94}.

A pertinent example is the group algebra $kG$, where $G$ is a finite group and $k$ is a field. In this case, $kG$ is indeed a Frobenius algebra, implying that $kG\cong (kG)^\ast$, and hence projective modules agree with injective modules.
    \end{Rem}

The following lemma is very easy, and is better left as an exercise:

\begin{Lemma}
    Let $(C, \Delta, \varepsilon)$ be a $k$-coalgebra and $(V, \Delta_V)$ be any $C$-comodule. Then the intersection of $C$-subcomodules of $V$ is again a $C$-subcomodule.
\end{Lemma}

\begin{Def}
    Let $(C, \Delta, \varepsilon)$ be a $k$-coalgebra, and $(V, \Delta_V)$ be any $C$--comodule. Let $v\in V$, then we denote by $\langle v \rangle$ the smallest $C$--subcomodule of $V$ containing $v$. Note that a smallest such subcomodule exists by the previous lemma, it is just the intersection of all subcomodules of $V$ containing $v$.
\end{Def}

The next proposition is usually referred to as ``local finiteness", and it will be crucial for us later on.

\begin{Prop}[\textbf{Local finiteness}]
    Let $(C, \Delta, \varepsilon)$ be a $k$--coalgebra which is flat as a $k$-module, and $(V, \Delta_V)$ be any $C$--comodule. Then for any $v\in V$, the $C$--comodule  $\langle v \rangle$ is finitely generated as a $k$-module. 
\end{Prop}

\begin{proof}
    Let $v\in V$ be arbitrary, and write $\Delta_V(v)=\sum_{i=1}^n v_i\otimes c_i$, where $v_i\in V$ and $c_i\in C$. Set $V'=\sum_i kv_i$. The counital axiom shows that $v\in V'$, since $v=\sum_i \varepsilon(c_i)v_i$. 
    
    We will now show that $\langle v \rangle \subseteq V'$. Set $W=\Delta_V^{-1}(V'\otimes C)$, in other words, $W$ consists of the elements $w\in V$ such that $\Delta_V(w)\in V'\otimes C$. Clearly $v\in W$, and moreover, as above, the counital axiom also proves that $W\subseteq V'$.

    Now we will prove that $W$ is a $C$-subcomodule of $V$, which will therefore imply that $\langle v \rangle \subseteq W$, proving that $\langle v \rangle \subseteq V'$ as desired. Now, using the flatness of $C$, we see that 
    \[
    W\otimes C=(\Delta_V\otimes \id_C)^{-1}(V'\otimes C \otimes C).
    \]
    To show that $W$ is a subcomodule of $V$ we must show that $\Delta_V(W)\subseteq W\otimes C$. As above, we see that the right hand side is equal to $ (\Delta_V\otimes \id_C)^{-1}(V'\otimes C \otimes C)$  therefore, equivalently we must show that 
    \[
    (\Delta_V\otimes \id_C)(\Delta_V(W))\subseteq V'\otimes C \otimes C.
    \]
      From the comodule axioms, we have that $(\Delta_V\otimes \id_C)\circ \Delta_V=(\id_V\otimes \Delta)\circ \Delta_V$, thus the left hand side is equal to $(\id_V \otimes \Delta)(\Delta_V(W))$, which, from the definition of $W$, is clearly contained in $V'\otimes C \otimes C$, so we are done.

    Thus $\langle v \rangle \subseteq V'$. Furthermore, $v_i\in \langle v \rangle$ for all $i$, and thus $V'=\sum_i kv_i \subseteq \langle v \rangle$. Thus we have shown that $\langle v \rangle = V'$, which is clearly finitely generated.
\end{proof}

\begin{Cor}\label{locallyfinite}
    Let $(C, \Delta, \varepsilon)$ be a $k$--coalgebra which is flat as a $k$-module, and $(V, \Delta_V)$ be any $C$--comodule. Then any finitely generated $k$--submodule of $V$ is contained in a finitely generated ${C}$--subcomodule of $V$.
\end{Cor}

\subsection*{Sweedler Notation}

 We introduce the \textit{Sweedler notation}, which serves to simplify our notation when dealing with coalgebras and comodules, and makes computations involving the comultiplication map much easier.
 
Whenever we have a coalgebra $C$, the comultiplication map $\Delta$ goes from $C$ to $C\otimes C$, meaning that the image of any $c\in C$ under $\Delta$ is of the form $\sum_i a_i\otimes b_i$ for some $a_i, b_i\in C$. Instead of picking new symbols $a$ and $b$ for the image of $c$, we should really try to reuse the symbol $c$ in some capacity. Something like $\Delta(c)=\sum_i c_{(1)i} \otimes c_{(2)i}$ is an improvement, but for practical purposes, computations only involve linear operations, so we can forget about the summation symbol and the index $i$ entirely and write: 
\[
\Delta(c)=c_{(1)}\otimes c_{(2)},
\]
remembering that the summation over $i$ is implicit in this notation. This is the \textit{Sweedler notation}.

It may not be apparent immediately the utility of this notation, but it will end up greatly simplifying the mess of symbols that appear during more involved computations. The coassociativity axiom becomes much easier to state in this formulation. The coassociativity axiom is that 
\[
(\Delta\otimes \id_C)\circ \Delta=(\id_C\otimes \Delta)\circ \Delta,
\]
 in the standard notation, this means: 
 \[
 \sum_i \Delta(c_{(1)i}) \otimes c_{(2)i}=\sum_i c_{(1)i} \otimes \Delta(c_{(2)i})
 \]
Expanded out, we would write:
 \[
  \sum_i \left(\sum_j (c_{(1)i})_{(1)j} \otimes (c_{(1)i})_{(2)j}\right) \otimes  c_{(2)i}=\sum_i c_{(1)i}\otimes \left(\sum_{j'} (c_{(2)i})_{(1)j'} \otimes (c_{(2)i})_{(2)j'}\right).
 \]
 It is clear that this notation is becoming almost too cumbersome to use. If we now denote $(c_{(i)})_{(j)}$ by $c_{(i)(j)}$, then in Sweedler's notation this equality would read: 
 \[
 c_{(1)(1)} \otimes c_{(1)(2)} \otimes c_{(2)} = c_{(1)} \otimes c_{(2)(1)} \otimes c_{(2)(2)}.
 \]

 Since the left and right side are equal above, we extend the Sweedler notation and simply write both sides of the equality as: 
 \[
 c_{(1)}\otimes c_{(2)} \otimes c_{(3)}.
 \]
Keep in mind that we are not literally saying that $c_{(1)(2)} = c_{(2)(1)}=c_{(2)}$ for example, since everything here is over completely different indexing sets. What we are saying is that we can make appropriate linear replacements. More precisely, if you see an expression in terms of $c_{(1)}\otimes c_{(2)}\otimes c_{(3)}$, you can replace it to be in terms of either of the two desired expressions above, depending on what best suits your current situation.

Similarly to the coassociativity axiom, the axiom for the counit map $\varepsilon: C\to k$ can be written as: 
\[
c=\varepsilon(c_{(1)})c_{(2)}=c_{(1)}\varepsilon(c_{(2)}). 
\]
As another example, let $C'$ be another coalgebra over $k$, then a $k$--linear map $f:C\to C'$ is a map of coalgebras if $(f\otimes f)\circ \Delta =\Delta' \circ f$. In Sweedler notation this can be expressed as the equality: 
\[
f(c_{(1)})\otimes f(c_{(2)}) = f(c)_{(1)}\otimes f(c)_{(2)}.
\]
Note that we are not literally asserting that $f(c_{(1)})=f(c)_{(1)}$ here, the left hand side and the right hand side may be summations over completely different indexing sets. 

The Sweedler notation for comodules is slightly different. Let $V$ be a $C$--comodule and let $v\in V$. Then the Sweedler notation for the coaction is written as: 
\[
\Delta_V(v)=v_{(0)} \otimes v_{(1)}.
\]
 We have used $(0)$ instead of $(1)$ for the first element to signify that it is an element of $V$ and not of $C$, that is, that the elements $v_{(0)i}$ are in $V$ for all $i$. Moreover, note that when using the Sweedler notation the $v_{(j)}$ for $j> 0$ are always elements of $C$.

\pagebreak

\section{Hopf Algebras}\label{Sec: Hopf algebras}

We first recall the definition of a bialgebra:
\begin{Def}
    Let $k$ be a ring. A \textit{bialgebra} $B$ over $k$ is a $k$--module which is both a unital associative algebra and a counital coassociative coalgebra over $k$ in a compatible way. 
    
    To spell this out more precisely: 
    \begin{enumerate}[(i)]
        \item $B$ is equipped with a $k$--linear map $\nabla\colon B\otimes_k B\to B $ called the \textit{multiplication}, and a $k$--linear map $\eta\colon k\to B$ called the \textit{unit}, such that $(B, \nabla, \eta)$ is a unital associative $k$--algebra.
        \item $B$ is equipped with a $k$--linear map $\Delta\colon B\to B \otimes_k B$ called the \textit{comultiplication}, and a $k$--linear map $\varepsilon\colon B\to k$ called the \textit{counit}, such that $(B, \Delta, \varepsilon)$ is a counital coassociative $k$--coalgebra.
        \item These structures are compatible, that is, $\eta$ and $\nabla$ are morphisms of counital coalgebras, and $\varepsilon$ and $\Delta$ are a morphisms of unital algebras. In diagrammatic form this means that the following diagrams commute:
\begin{center}
    \begin{tikzcd}[column sep=small]
        &B\otimes B \arrow[dr, "\varepsilon\otimes \varepsilon"']\arrow[rr, "\nabla"]& &B\arrow[dl, "\varepsilon"] & &k \arrow[rr, "\id"]\arrow[dr, "\eta"']& &k &  &k\otimes k\simeq k \arrow[dr, "\eta"] \arrow[dl, "\eta\otimes\eta"'] &\\
        
        & &k\otimes k\simeq k & & & &B \arrow[ur, "\varepsilon"'] &   &B\otimes B & &B \arrow[ll, "\Delta"]
    \end{tikzcd}
\end{center}
\begin{center}
\begin{tikzcd}
    &B\otimes B \arrow[r, "\nabla"]\arrow[d, "\Delta\otimes\Delta"']&B\arrow[r, "\Delta"] &B\otimes B \\
    &B\otimes B\otimes B\otimes B \arrow[rr, "\id\otimes (12)\otimes \id"']& &B\otimes B\otimes B\otimes B \arrow[u, "\nabla\otimes\nabla"']
\end{tikzcd}
\end{center}
where $(12)(x\otimes y)=y\otimes x$ is the swapping morphism. A $k$--linear morphism between bialgebras is said to be a \textit{bialgebra homomorphism} if it is both a map of algebras and coalgebras.
    \end{enumerate}

From Remark \ref{dual of a coalgebra}, we can see that for any bialgebra $B$, we can always give $\Hom_k(B, B)$ the structure of a $k$--algebra via the convolution product. A \textit{Hopf algebra} is a bialgebra $H$ over $k$, equipped with a $k$--linear map $S\colon H\to H$ called the \textit{antipode} which is both a left and right convolution inverse to the identity map $\id_H\colon H\to H$ in $\Hom_k(H, H)$. In other words, $S$ makes the following diagram commute:
        \begin{center}
            \begin{tikzcd}[column sep=tiny]
                & &H\otimes H \arrow[rr, "S\otimes \id"]& &H\otimes H \arrow[ddr, "\nabla"] & \\
                & & & & & \\
                &H \arrow[ddr, "\Delta"']\arrow[uur, "\Delta"]\arrow[rr, "\varepsilon"]& &k \arrow[rr, "\eta"]& &H \\
                & & & & & \\
                & &H\otimes H \arrow[rr, "\id \otimes S"'] & &H\otimes H \arrow[uur, "\nabla"']& 
            \end{tikzcd}
        \end{center}
       A \textit{morphism of Hopf algebras} is just a morphism of bialgebras which commutes with the antipodes. In fact, any morphism of bialgebras will automatically respect the antipode (see Lemma \ref{bialgebra maps}). We will denote the category of Hopf algebras over $k$ by $\textbf{Hopf}_k$.
 \end{Def}

\begin{Rem}
 Note that in an associative algebra, two sided inverses are necessarily unique, since if $S$ and $S'$ are both two sided inverses to $\id_H$ then: 
    \[
    S = S \ast (\id_H \ast 
 S') = (S \ast \id_H) \ast S' = S'.
    \]
    That is to say that being a Hopf algebra is a property of a bialgebra, rather than a structure.    
\end{Rem}

\begin{Lemma}\label{Properties of the antipode}
        Let $H$ be a Hopf algebra. Then the antipode $S$ is an antimorphism of the algebra and coalgebra structures on $H$, in other words:
        \begin{enumerate}[(i)]
            \item $S$ is an antialgebra morphism: $S(gh)=S(h)S(g)$, and $S(1)=1$.
            \item $S$ is an anticoalgebra morphism: $\Delta(S(h))= S(h_{(2)})\otimes S(h_{(1)})$ and $\varepsilon(S(h))=\varepsilon(h)$.
        \end{enumerate}
\end{Lemma}

\begin{proof}
    We will prove only the first point, the second one is similar. 
    
    Recall there is an associative algebra structure on $\Hom_k(H\otimes H, H)$ as given in Example \ref{dual of a coalgebra} by convolution, where $H\otimes H$ has the tensor product coalgebra structure from Example \ref{tensor coalgebra}. The identity element for convolution is the composite map $\eta\circ\varepsilon\circ \nabla$. To show that $S(gh)=S(h)S(g)$ we will show that the maps 
    \[
    \varphi\colon  g\otimes h \mapsto S(gh) \mbox{ and } \psi\colon g\otimes h \mapsto S(h)S(g)
    \]
      are both two sided convolution inverses to the multiplication map $\nabla\colon g\otimes h \mapsto gh$ in $\Hom_k(H\otimes H, H)$. Firstly we multiply $\varphi$ on the right by $\nabla$:
    \begin{align*}
        (\varphi\ast \nabla)(g\otimes h)&=\psi(g_{(1)} \otimes h_{(1)})\nabla(g_{(2)}\otimes h_{(1)}) \\
        &=S(g_{(1)}h_{(1)})g_{(2)}h_{(2)} \\
        &=S((gh)_{(1)})(gh)_{(2)} \\
        &=\eta\varepsilon(gh).
    \end{align*}
Now we calculate $\nabla \ast \psi$:
\begin{align*}
    (\nabla \ast \psi)(g\otimes h)&= g_{(1)}h_{(1)} S(h_{(2)})S(g_{(2)}) \\
    &= g_{(1)}(\eta\varepsilon(h))S(g_{(2)}) \\
    &=g_{(1)}S(g_{(2)})\varepsilon(h) \\
    &=\eta(\varepsilon(g))\varepsilon(h)\\
    &=\eta\varepsilon(gh).
\end{align*}
Then we have $\varphi=\varphi \ast (\nabla \ast\psi)=(\varphi \ast \nabla) \ast\psi =\psi$ as desired. To show that $S(1)=1$, note that $S(1)=(\id_H \ast S)(1)=\eta\varepsilon(1)=1$.
\end{proof}

\begin{Lemma}\label{bialgebra maps}
    Let $H$ and $H'$ be two Hopf algebras with antipodes $S_H$ and $S_{H'}$ respectively. Then for any map of bialgebras $f:H\to H'$,  we have $f \circ S_H = S_{H'}\circ f$.
\end{Lemma}

\begin{proof}
    As in the previous lemma, we will prove this equality by showing that $f \circ S_H$ and $ S_{H'}\circ f$ are both left and right convolution inverse to $f$ in $\Hom_k(H, H')$, so that they must be equal. Recall that the identity object of $\Hom_k(H, H')$ under convolution is the map $H\xrightarrow[]{\varepsilon} k \xrightarrow[]{\eta'} H'$. Since $f$ is a map of coalgebras we have: 
    \[
    f \ast (S_{H'} \circ f)=(\id_{H'} \ast S_{H'}) \circ f =  \eta' \circ \varepsilon' \circ f = \eta' \circ \varepsilon.
    \]
    Similarly, since $f$ is a map of algebras we have: 
    \[
    (f \circ S_H) \ast f = f \circ (S_{H} \ast \id_H) = f \circ \eta \circ \varepsilon = \eta' \circ \varepsilon.
    \]
    Now repeating the standard argument using the associativity of $\ast$: 
    \[
    (f \circ S_H)=(f \circ S_H) \ast (f \ast (S_{H'} \circ f))=((f \circ S_H) \ast f) \ast (S_{H'} \circ f) = (S_{H'} \circ f).
    \]
    In particular, $f$ is morphism of Hopf algebras.   
\end{proof}

\begin{Rem}\label{dual of a Hopf algebra}
   Note that if $H$ is a Hopf algebra, then the counit $\varepsilon\colon  H \to k$ is an augmentation. As observed in Remark \ref{dual of a coalgebra}, if $H$ is a Hopf algebra, then $H^\ast$ is an algebra. Furthermore, even when $H$ is not finitely generated and projective, the dual $H^\ast$ is still an augmented algebra - though it may lack a comultiplication. In this case, the augmentation is induced from the unit by dualising $\eta^\ast\colon  H^\ast \to k^\ast$ and composing with the natural isomorphism $k^\ast \simeq k$. More explicitly, the augmentation
    \[
    H^*\to k \mbox{ is defined by } f\mapsto f(1). 
    \]
     In the case where $H$ is finitely generated and projective over $k$, we saw in Remark \ref{dual of a coalgebra} that $\nabla^*$ indeed induces a comultiplication on  $H^\ast$, where the above augmentation is the counit, making $H^\ast$ into a $k$--bialgebra. Moreover, it is easy to see that $S^*$ is actually an antipode for $H^\ast$, turning $H^\ast$ into a Hopf algebra.
\end{Rem}

\begin{Ex}
    Let $G$ be a group. Recall that the \textit{group algebra} $kG$ is the free $k$--module with basis $\{e_g\}_{g\in G}$ indexed by the elements of $G$, and $k$--algebra structure induced by the group structure of $G$, that is, $e_ge_h=e_{gh}$. One easily sees that the element $e_1$ is the identity element of $kG$. This is actually a Hopf algebra, with comultiplication 
    \[
    \Delta\colon kG\to kG\otimes kG \mbox{ defined by } \Delta(e_g)=e_g\otimes e_g,
    \]
    antipode $S(e_g)=e_{g^{-1}}$, and counit $\varepsilon(e_g)=1$. As we have seen in Remark \ref{dual of a coalgebra}, its dual $kG^*$, which we will denote by $k[G]$, will naturally have the structure of an augmented algebra with multiplication given by convolution. We know that $k[G]$ contains the distinguished set of linear functionals $\{f_g\}_{g\in G}$ with $f_g(e_h)=0$ for $h\neq g$ and $f_g(e_g)=1$, which will not be a basis for $k[G]$ in general. To see how the augmented algebra structure acts on this distinguished set of functionals, we compute 
    \[
    (f_g \ast f_h)(e_\ell) = (f_g\otimes f_h)(\Delta(e_\ell))=f_g(e_\ell)f_h(e_\ell),
    \]
    therefore, for  $g\neq h$ we have $f_g\ast f_h =0$, and $f_g\ast f_g = f_g$. The augmentation is just the dualisation of the unit map in $kG$, so we see that
\[
\eta^*(f_g)=f_g(e_1)=\begin{cases} 1 \text{ if $g=1$}, \\
    0 \text{ otherwise}.
    \end{cases}
\]
    In the case where $G$ is a finite group, then $kG$ is a finite free module of rank $|G|$, and this collection $\{f_g\}_{g\in G}$ actually form a basis for $k[G]$. As in Remark \ref{dual of a Hopf algebra} we see that $k[G]$ inherits the full structure of a Hopf algebra, and the augmentation defined above becomes the counit. Similarly to the above calculation, we see that the comultiplication is defined by 
    \[
    \nabla^*(f_g)=\sum_{g', g'' | g'g''=g}f_{g'}\otimes f_{g''}.
    \]
    And lastly, the antipode for $k[G]$ is given by the dual of the antipode of $kG$, that is $S^*(f_g)=f_{g^{-1}}$. In this case, $k[G]$ is also often known as the \textit{coordinate algebra}, for reasons we will see in the next section.
\end{Ex}

\begin{Def}
    A right (resp. left) \textit{Hopf module} $M$ over a Hopf algebra $H$ is a $k$--module which is both a right (resp. left) $H$--module and a right (resp. left) $H$--comodule such that the coaction $\Delta_M$ is a morphism of right (resp. left) $H$--modules. More precisely, this means that for $m\in M$ with we have 
    \[
    \Delta_M(mh)=m_{(0)}h_{(1)}\otimes m_{(1)}h_{(2)}.
    \]
\end{Def}

\begin{Def}\label{def:coinvariants}
We will now define invariants and coinvariants:
\begin{enumerate}[(i)]
    \item Let $M$ be a left module over an augmented $k$--algebra $A$ with augmentation map $\varepsilon\colon A\to k$. We define its $k$--module of \textit{invariants} by 
    \[
    M^A\coloneqq\{m\in M\mid am=\varepsilon(a)m \text{ for all }a\in A\}.
    \]
    In the case $M=A$, elements of $A^A$ are called \textit{left integrals} of $A$. It is easy to see that $A^A$ is stable under right multiplication from $A$ using the associativity of multiplication, actually even more is true, see Lemma \ref{right mult}. Invariants for a right $A$--module are defined in a similar way.

    \item    Let $M$ be a right comodule over a $k$--coalgebra $C$. We define its $k$--module of \textit{coinvariants} by 
    \[
    M^{\text{co}C}\coloneqq \{m\in M \mid \Delta_M(m)=m\otimes 1_C\}.
    \]
    Moreover, for a map $f\colon M\to N$ of right $C$--comodules, we define $f^{\text{co}C}$ as the restriction of $f$ to $M^{\text{co}C}$. It is easy to see that the image of $f^{\text{co}C}$ is contained in $N^{\text{co}C}$, and thus that $(-)^{\text{co}C}$ is a functor from the category of right $C$--comodules to the category of $k$--modules. Furthermore, it is easy to see that the coinvariants form a subcomodule of $M$.
\end{enumerate}
\end{Def}

\begin{Lemma}\label{right mult}
    Let $M$ be a left module over an augmented algebra $A$ , then $A^A \cdot M \subseteq M^{A}$.
\end{Lemma}

\begin{proof}
Let $x\in A^A$ and $m\in M$ be arbitrary. We want to show that $x \cdot (a \cdot m)=\varepsilon(x)(a\cdot m)$ for all $x\in A$. Since the multiplication is associative, we have 
\[
x \cdot (a \cdot m)=(x \cdot a) \cdot m = (\varepsilon(x)a)\cdot m = \varepsilon(x)(a\cdot m)
\]
where we have used the fact that $a\in A^A$ in the second equality.
\end{proof}

\begin{Lemma}\label{invariants and coinvariants} Let $M$ be a right comodule over $H$. Then, viewing $M$ as a left module over $H^*$ via Remark \ref{comodC embedding into modC*}, we have $M^{\text{co}H}\subseteq M^{H^*}$. If in addition $H$ is finitely generated and projective over $k$, then $M^{\text{co}H}= M^{H^*}$.
\end{Lemma}

\begin{proof}
For the first claim, fix some $m\in M^{\text{co}H}$ and let $f\in H^*$ be arbitrary. Recall that the augmentation map on $H^*$ is defined by $f\mapsto f(1_H)$. Let $\mu$ denote the multiplication map $\mu\colon M\otimes k \to M$. We want to show that $m\cdot f =m f(1_H)$. By definition of the induced module action we have 
\[
m\cdot f=\mu((\id_M\otimes f)(\Delta_M(m)))=\mu((\id_M\otimes f)(m\otimes 1_H))=\mu(m\otimes f(1_H))=mf(1_H),
\]
which proves the first claim.

For the second claim: Since $H$ is projective it admits a basis, that is, a pair $((h_i)_{i\in I}, (\tau_i)_{i\in I})$ with $h_i\in H$ and $\tau_i\in H^*$ such that any $h\in H$ can be written as $h=\sum_i \tau_i(h)h_i$, with $\tau_i(h)=0$ for almost all $i$. Moreover, since $H$ is finitely generated and projective, so is $H^*$, and in particular it admits a dual basis of the form $((\tau_i)_{i\in I}, (c_{h_i})_{i\in I})$, where the maps $c_{h_i} \in (H^*)^*$ are the evaluation maps at the various $h_i$, that is $c_{h_i}(f)=f(h_i)$. Moreover, the map 
\[
\Psi\colon  (H^*)^*\to H, \quad f \mapsto \sum_i f(\tau_i)h_i
\]
 is an isomorphism, and sends $c_{h_j}$ to $h_j$, for all $j$. There is now a coevaluation map
 \[
 \text{coev}:k\to H^*\otimes (H^*)^* \simeq H^* \otimes H, \quad r \mapsto r\sum_i \tau_i\otimes h_i
 \]
 Theorem \ref{equivalence} gives an alternative characterisation of the coaction $\Delta_M\colon M\to M\otimes H$ in terms of the action $\lambda_M\colon H^*\otimes M\to M$. To get the coaction from the action, we consider the composite: 
 \[
\Delta_M\colon  M\simeq M\otimes k\xrightarrow[]{\id_M\otimes \text{coev}} M \otimes H^*\otimes H \xrightarrow[]{\lambda_M \otimes \id_H} M\otimes H
 \]
Now for $m\in M$ we may write $\Delta_M(m)$ as its image under this composite: 
\[\Delta_M(m)=\sum_i (\tau_i \cdot m) \otimes h_i\]
Now let $m\in M^{H^*}$, that is $m \cdot f=mf(1_H)$ for all $f\in H^*$. Thus we see that 
\[
\Delta(m)=\sum_i (\tau_i \cdot m) \otimes h_i = \sum_i (\tau_i(1_H)m) \otimes h_i =\sum_i m \otimes \tau_i(1_H) h_i =m\otimes 1
\]
where we have used the fact that $1_H=\sum_i \tau_i(1_H)h_i$.
\end{proof}

\begin{Lemma}\label{H* is a right Hopf Module}
    If $H$ is a finitely generated projective Hopf algebra, then $H^*$ is a right Hopf module over $H$.
\end{Lemma}

\begin{proof}
As we have seen in Remark \ref{dual of a coalgebra}, the dual of a coalgebra can be given the structure of an algebra. In particular, this means that $H^*$ has a product structure $H^*\otimes H^* \to H^*$ given by convolution, making $H^*$ into a left $H^*$--module via left multiplication. Since $H$ is finitely generated and projective over $k$ we may apply Theorem \ref{equivalence}, which gives $H^*$ the structure of a right $H$--comodule. 

Moreover, we can make $H^*$ into a right $H$--module in the following way: for $h\in H$ and $f\in H^*$ we define a map $f \cdot h\in H^*$ by $(f\cdot h)(a)=f(aS(h))$ for $a\in H$. One may check that this indeed gives $H^*$ a right $H$--module structure.

Now one checks that this right $H$--comodule structure and right $H$--module structure on $H^*$ are compatible in such a way that makes $H^*$ a right $H$-Hopf module as desired.
\end{proof}

\begin{Lemma}\label{summand}
    Let $M$ be a right Hopf module over $H$. Then as $k$--modules $M^{\text{co}H}$ is a direct summand of $M$.
\end{Lemma}
 \begin{proof}
There is a map $\varphi: M\to M^{\text{co}H}$ defined by $m\mapsto  m_{(0)}S(m_{(1)})$. This map is easily seen to be a $k$--linear retraction of the inclusion $M^{\text{co}H}\hookrightarrow M$.  To see that the image of this map lands in $M^{\text{co}H}$, we calculate the coaction on an element in the image of this map and see that it is coinvariant:
\begin{align*}
    (\Delta_M \circ \varphi)(m)&= \Delta_M(m_{(0)}S(m_{(1)})) \\
    &=m_{(0)(0)}S(m_{(1)})_{(1)} \otimes m_{(0)(1)}S(m_{(1)})_{(2)} &\text{($M$ is a Hopf Module)}\\ 
    &=m_{(0)(0)}S(m_{(1)(2)}) \otimes m_{(0)(1)}S(m_{(1)(1)}) &\text{(Lemma \ref{Properties of the antipode}(ii))}\\
    &=m_{(0)(0)(0)}S(m_{(1)})\otimes m_{(0)(0)(1)}S(m_{(0)(1)}) &\text{(coassociativity of $\Delta$)} \\
    &=m_{(0)(0)}S(m_{(1)}) \otimes m_{(0)(1)(0)}S(m_{(0)(1)(1)}) &\text{(coassociativity of $\Delta$)}\\
    &=m_{(0)(0)}S(m_{(1)}) \otimes \eta \varepsilon(m_{(0)(1)}) &\text{(definition of $S$)}\\
    &=m_{(0)(0)}\varepsilon(m_{(0)(1)}) S(m_{(1)}) \otimes 1 \\
    &=m_{(0)}S(m_{(1)}) \otimes 1. &\text{(counit axiom for comodules)}
\end{align*}

Where we have repeatedly made use of the coassocaitivity axiom for the coaction. To explain the agonizing steps, keep in mind that the coaction is happening twice, so in each line we are making one of the following replacements 
\begin{equation*}
    m_{(0)(0)}\otimes m_{(0)(1)} \otimes m_{(1)(1)} \otimes m_{(1)(2)} \mapsto m_{(0)(0)(0)}\otimes m_{(0)(0)(1)} \otimes m_{(0)(1)} \otimes m_{(1)}
\end{equation*}
for line 3 to line 4, and the replacement 
\begin{equation*}    
m_{(0)(0)(0)}\otimes m_{(0)(0)(1)} \otimes m_{(0)(1)} \otimes m_{(1)} \mapsto m_{(0)(0)}\otimes m_{(0)(1)(0)} \otimes m_{(0)(1)(1)} \otimes m_{(1)}
\end{equation*}
for line 4 to line 5.
 \end{proof}

 \begin{Rem}
     In the previous proof, we could have simplified the notation by using extended Sweedler notation, 
     but we chose to spell out the steps in detail at least once for clarity.
 \end{Rem}

\begin{Lemma}\label{tensor identity}
    Let $M$ be a right Hopf module over $H$. Then $M\simeq M^{\text{co}H}\otimes H$ as right Hopf modules, where $ M^{\text{co}H}\otimes H$ has the trivial Hopf module structure.
\end{Lemma} 

\begin{proof}
    We view $M^{\text{co}H}\otimes H$ as the Hopf module with action $(m\otimes h)\cdot h'=m\otimes hh'$ and coaction given by $\Delta'\coloneqq\id_M\otimes \Delta$, one checks that this makes $M^{\text{co}H}\otimes H$ into a Hopf module over $H$. We will show that the map
    \[
    \rho\colon M^{\text{co}H}\otimes H\to M, \quad m\otimes h\mapsto mh
    \]
    is an isomorphism of Hopf modules. In particular, that is it an isomorphism of right $H$--modules and right $H$--comodules. The map $\rho$ is a morphism of right $H$--modules since 
    \[
    \rho(m\otimes h) \cdot h'=mhh'= \rho((m\otimes h)\cdot h')).
    \]
    To check that the map $\rho$ is a morphism of right $H$--comodules we must check that $\Delta_M\circ \rho=(\rho\otimes \id_H)\circ\Delta'$. We have:
    \begin{align*}
        (\Delta_M\circ \rho) (m\otimes h) &= \Delta_M(mh) \\
        &= mh_{(1)} \otimes h_{(2)} \\
        &=(\rho \otimes 
        \id_H)(  m 
        \otimes h_{(1)} \otimes h_{(2)}) \\
        &=((\rho \otimes 
        id_H) \circ \Delta')(m\otimes h).
    \end{align*}
    Where in the second equality we have used that fact that $\Delta_M(m)=m\otimes 1$ and that $\Delta_M$ is a morphism of right $H$--comodules. We have seen that the map 
    \[
    \varphi\colon M\to M, \quad m\mapsto m_{(0)}S(m_{(1)})
    \]
    from  Lemma \ref{summand} has image contained in $M^{\text{co}H}$. Then we now define a map $\theta\colon M\to M^{\text{co}H}\otimes H$ by the composite: 
    \[
    M\xrightarrow{\Delta_M}{} M\otimes H \xrightarrow{\varphi \otimes \id_H}{} M^{\text{co}H}\otimes H.
    \]
    That is $\theta(m)=m_{(0)}S(m_{(1)})\otimes m_{(2)}$. Similarly, we can easily check that $\theta$ is a map of both right $H$--modules and right $H$--comodules. Moreover, these two maps $\theta$ and $\rho$ are $H$--linear inverses to each other. Indeed,  we calculate: 
    \begin{align*}
       (\theta\circ\rho)(m\otimes h) &= \theta(mh) \\
       &=\theta(m)\cdot h \\
       &=(m\otimes S(1)\otimes 1) \cdot h \\
       &=m\otimes h.
    \end{align*}
    Here we have used the fact that $\theta$ is $H$--linear, and that $S(1)=1$ from Lemma \ref{Properties of the antipode} (i), and the fact that $m\in M^{\text{co}H}$. For the other composite we have:
    \begin{align*}
        (\rho \circ \theta)(m) &= m_{(0)}S(m_{(1)})m_{(2)} \\
        &=m_{(0)}S(m_{(1)(0)})m_{(1)(1)} \\
        &=m_{(0)} \eta\varepsilon (m_{(1)}) \\
        &=m.
    \end{align*}
    Here we have used the coassociativity of $\Delta_M$, as well as the axiom defining the antipode $S$. We see that $\rho$ and $\theta$ are mutually inverse maps of Hopf modules over $H$, and so define an isomorphism as desired.
\end{proof}
 
\begin{Lemma}\label{coinvariants are rank 1}
    If $H$ is finitely generated and projective over $k$, then $(H^*)^{\text{co}H}$ is finitely generated and projective of rank $1$ over $k$.
\end{Lemma}
\begin{proof}
    Firstly, $(H^*)^{\text{co}H}$ is a $k$--summand of $H^*$ which is finitely generated and projective over $k$, and therefore $(H^*)^{\text{co}H}$ is itself finitely generated and projective over $k$.

    Now let $\mf{p}\in \Spec(k)$. It is a standard fact that localization and $\Hom$'s commute for finitely generated $k$--modules $M$. Therefore we see that $(M^*)_\mf{p}=(M_\mf{p})^*$. Moreover, coinvariants commute with this localization since for a right $H$--comodule $M$ we have 
    \[
    M^{\text{co}H}=\ker(\Delta_M-(1_M\otimes 1_H)),
    \]
    and $k_\mf{p}$ is flat over $k$ so that localization is an exact functor and preserves kernels. Therefore we may write $M^*_\mf{p}$ and $M^{\text{co}H}_\mf{p}$ without possibility of confusion. Since localization and tensor products commute, by Lemma \ref{tensor identity}, we can see that 
    \[
    H^*_\mf{p} \simeq (H^*)^{\text{co}H}_\mf{p} \otimes H_\mf{p}.
    \]
   Now $H_\mf{p}$ is a finitely generated projective module over the local ring $k_\mf{p}$, and thus is it free, and therefore isomorphic over $k_\mf{p}$ to its dual $H^*_\mf{p}$. Since rank is multiplicative over tensor products we must have that $\rank((H^*)^{\text{co}H}_\mf{p})=1$ for all primes $\mf{p}$. Moreover, the rank is a local property on $\Spec(k)$ we are done.
\end{proof}

\begin{Rem}
If in addition to $H$ being finitely generated and projective we have that $(H^*)^{\text{co}H}$ is actually free over $k$, then $H$ is a so-called Frobenius algebra. We will prove this in Theorem \ref{some Hopf algebras are Frobenius}.    
\end{Rem}

\begin{Def}\label{Frobenius Algebra}
    Let $A$ be a $k$--algebra which is finitely generated and projective. We say that $A$ is a \textit{Frobenius algebra} if there is an isomorphism of left $A$--modules $\Phi\colon A\to A^*$ where we consider $A^*$ as a left $A$--module via $(x\cdot f)(a)=f(ax)$ for all $a, x\in A$ and $f\in A^*$. We will call the morphism $\Phi$ a \textit{Frobenius isomorphism}.
\end{Def}

\begin{Rem}
    If $A$ is a Frobenius algebra, note that since $1$ freely generates $A$ as a left $A$--module, then the Frobenius isomorphism then $\Phi$ tells us that $\Phi(1)\coloneqq\psi$ freely generates $A^*$ as a left $A$--module.
\end{Rem}

\begin{Lemma}\label{antipode is bijective}
    If $H$ is finitely generated and projective over $k$, then the antipode $S$ is bijective.
\end{Lemma}

\begin{proof}
    To show that $S$ is injective, consider the isomorphism of right $H$--modules 
\[
\varphi\colon H^* \to (H^*)^{\text{co}H}\otimes H,
\]
 where the right $H$--module structure is given by $(f\cdot h)(a)=f(aS(h))$ for all $a, h\in H$. Now suppose $S(h)=0$, then for any $a\otimes b\in  (H^*)^{\text{co}H}\otimes H$ we have 
\[
\varphi^{-1}((a\otimes b)h)=\varphi^{-1}(a\otimes b) \cdot h = 0
\]
since it is multiplication inside the argument by $S(h)=0$. Now, since $\varphi^{-1}$ is a bijection, we have $(a\otimes b)h=0$, since $a$ and $b$ were arbitrary that is to say that multiplication by $h$ on  $(H^*)^{\text{co}H}\otimes H$ is just zero. Let $\mf{p}\in \Spec(k)$, since  $(H^*)^{\text{co}H}$ is finitely generated and projective of rank 1 over $k$ by Lemma \ref{coinvariants are rank 1}, over the local ring $k_\mf{p}$ is it free of rank $1$, therefore $(H^*)^{\text{co}H}_\mf{p}\simeq k_\mf{p}$. Hence the tensor identity gives us that 
\[
(H^*)^{\text{co}H}_\mf{p}\otimes H_\mf{p} \simeq H_\mf{p}. 
\]
Thus multiplication on the right by $h$ on $H_\mf{p}$ is zero for all $\mf{p}$, and hence multiplication on the right by $h$ must be zero on $H$, which implies $h=0$ as desired.

To show that $S$ is surjective, consider $Q=\coker S$, then $Q_\mf{p}=\coker(S_\mf{p})$ since localization is exact. It follows that $S$ induces a map 
\[
S'\colon H_\mf{p}/\mf{p}H_{\mf{p}}\to H_\mf{p}/\mf{p}H_{\mf{p}}
\]
with cokernel $Q_\mf{p}/\mf{p}Q_{\mf{p}}$. This holds since the tensor is right exact and $M/IM\simeq k/I \otimes M$.  We get that $S'$ is an antipode for the Hopf algebra $H_\mf{p}/\mf{p}H_{\mf{p}}$ over the residue field $F=k_\mf{p}/\mf{p}k_{\mf{p}}$. The above proof shows that this map is an injective map of $F$--modules. It is a standard fact that injective endomorphism of vector spaces are surjective. Therefore we see that $S'$ is surjective so that $Q_\mf{p}/\mf{p}Q_{\mf{p}}=0$ . The finite generation of $Q$ implies the finite generation of $Q_\mf{p}$. Nakayama's Lemma then implies that $Q_\mf{p}=0$ for all $\mf{p}$. Since being zero is a local property on $\Spec(k)$, we see that $Q=0$ as desired.
\end{proof}

\begin{Th}\label{some Hopf algebras are Frobenius}
    Let $H$ be a finitely generated projective Hopf algebra over $k$, and $(H^*)^{\text{co}H} \simeq k$. Then $H$ is a Frobenius algebra.
\end{Th}

\begin{proof}
    From Lemma \ref{tensor identity} we have an isomorphism $\varphi \colon H^* \to (H^*)^{\text{co}H} \otimes H$ of right $H$--modules, where $H^*$ has the right $H$--module structure of Lemma \ref{H* is a right Hopf Module} defined by $(f\cdot h)(a)=f(aS(h))$ for all $a, h\in H$. From our assumption that $(H^*)^{\text{co}H}\simeq k$, this is actually an isomorphism $\varphi \colon H^*\to H$ of right $H$--modules. Now the antipode $S$ is bijective via Lemma \ref{antipode is bijective}, so we have a bijection 
    \[
    \Phi\coloneqq \varphi^{-1}\circ S^{-1}\colon  H\to H^*.
    \]
     To show that $H$ is a Frobenius algebra, it remains to show that $\Phi$ is a map of left $H$--modules, that is, that $h\cdot \Phi(a)=\Phi(ha)$ for all $h, a\in H$, where the left $H$--module structure on $H^*$ is given as in Definition \ref{Frobenius Algebra}, that is $(h\cdot f)(a)=f(ah)$. Let $b\in H$ be arbitrary, then we have:\begin{align*}(\Phi(ha))(b)&=(\varphi^{-1}S^{-1}(ha))(b)\\
    &=(\varphi^{-1}(S^{-1}(a)S^{-1}(h)))(b) &\text{(Lemma \ref{Properties of the antipode})}\\
    &=(\varphi^{-1}(S^{-1}(a)) \cdot S^{-1}(h))(b) &(\varphi^{-1} \text{ is a map of right $H$--modules})\\
    &=(\varphi^{-1}(S^{-1}(a))(bh) &\text{(the definition of the right module action on $H^*$)}\\
     &=(\Phi(a))(bh) \\
     &=(h\cdot \Phi(a))(b)&\text{(the definition of the left module action on $H^*$)}
    \end{align*} 
    It follows that $\Phi$ is an isomorphism of left $H$--modules. 
\end{proof}

\begin{Lemma}\label{left norm}
    Let $A$ be an augmented Frobenius algebra with augmentation $\varepsilon$, and Frobenius isomorphism $\Phi$.  Then there exists a unique element $N\in A^A$ with $N\cdot \psi =\varepsilon$, where $\psi$ denotes $\Phi(1)$.
\end{Lemma}

\begin{proof}
    The existence of $N$ is simply the fact the $\psi$ freely generates $A^*$ as a left $A$--module, so it must hit $\varepsilon \in A^*$. More precisely, $\Phi$ is surjective and $A$--linear on the left, so that $\varepsilon$ has a preimage in $x\in A$. That is, $x\cdot \Phi(1) =\Phi(x)=\varepsilon$, so $x=N$ works. This also shows unicity since $\Phi$ is injective by assumption; so if $N, N'$ are both inverse images under $\Phi$ of $\varepsilon$, then they must be the same.

    It remains to show that $N$ is a left integral in $A$. For all $a, b\in A$ we have 
    \[
    (aN\cdot \psi)(b)=\psi(baN)=(N\cdot \psi)(ba)=\varepsilon(ba)=\varepsilon(b)\varepsilon(a)=(N\cdot \psi)(b)\varepsilon(a)=(\varepsilon(a)N\cdot \psi)(b).
    \]
    Thus we see that for all $a$ we have $\Phi(aN)=\Phi(\varepsilon(a)N)$, and the injectivity of $\Phi$ tells us that $aN=\varepsilon(a)N$. In other words, we get that $N\in A^A$.
\end{proof}

\begin{Th}
    If $H$ is a Hopf algebra which is also a Frobenius algebra, then, as $k$--modules, the ideal of left integrals $H^H$ is a direct summand of $H$ which is free of rank 1. 
\end{Th}

\begin{proof}
    A Hopf algebra is in particular augmented, so we will show that the element $N$ of Lemma \ref{left norm} is a free generator for $H^H$. Let $h\in H^H$ be any left integral. Then for any $a\in H$, we have 
    \[
    (h\cdot \psi)(a)=\psi(ah)=\psi(\varepsilon(a)h)=\psi(h)\varepsilon(a)=\psi(h)\psi(aN)=\psi(a\psi(h)N)=(\psi(h)N\cdot \psi)(a).
    \]
     Once again the injectivity of $\Phi$ tells us that $h=\psi(h)N$ for all $h$. Thus $N$ does indeed generate $H^H$ as a $k$--module. The map $H\to H^H$ defined by $h\mapsto \psi(h)N$ is then easily seen to be a retraction of the inclusion $H^H\hookrightarrow H$.

    It remains to show that $N$ freely generates $H^H$. That is, we want to show that the map
    \[
    k\to kN=H^H, \quad r\mapsto rN
    \]
      is injective. The map $\varepsilon$ is surjective, and $\Hom_k(-, k)$ is left exact, thus $\varepsilon^*$ is injective. Therefore, we see that the composite map $\rho: k\simeq k^* \to H^* \simeq H$ given by the natural isomorphism $k\simeq k^*$ composed with $\varepsilon^*$ and then finally composed with $\Phi^{-1}$, is injective, since all the constituents are injective. Let $r\in k$, then its image in $k^*$ under the first map is the $k$--linear function $f_r$ which sends $r$ to $1$. Now for all $h\in H$ we have 
      \[
      \varepsilon^*(f_r)(h)=r\varepsilon^*(f_1)(h)=r\varepsilon(h)=r(N\cdot \psi)(h)=(rN\cdot \psi)(h).
      \]
      Since this holds for all $h$, we get that the functions $\varepsilon^*(f_r)$ and $rN\cdot \psi$ agree. Now, we have 
      \[
      \rho(r)=\Phi^{-1}(\varepsilon^*(f_r))=\Phi^{-1}((rN\cdot \psi) = rN\Phi^{-1}(\psi) =rN
      \]
      so we are done.
\end{proof}

\begin{Cor}\label{cor:left-invariants-freesummand}
    Let $H$ be a Hopf algebra which is finitely generated and projective over $k$ with $(H^*)^{\text{co}H}\simeq k$. Then, as $k$--modules $H^H$, is a rank one free summand of $H$.
\end{Cor}

\begin{Rem}
We can apply Corollary \ref{cor:left-invariants-freesummand}, for example, whenever $H$ is finitely generated and projective over a local ring $k$, so that $(H^*)^{\text{co}H}\simeq k$ since projective modules are free over local rings. We will make use of this fact later on.    
\end{Rem}

\pagebreak

\pagebreak
\section{Affine Group Schemes}

We denote the category of affine schemes over $k$ by $\textbf{AffSch}_k$, which is equivalent to the opposite category of the category of commutative $k$--algebras $\textbf{CAlg}_k$, that is, $(\textbf{AffSch}_k)\simeq (\textbf{CAlg}_k)^{\text{op}}$. Moreover, we will denote the category of groups by $\textbf{Grps}$.

We recall the definition of a group object internal to a category $\mathcal{C}$.
\begin{Def}\label{def:group object}
    Let $\mathcal{C}$ be a category with binary products and a terminal object $\ast$. A \textit{group object} in $\mathcal{C}$ is an object $G$ of 
    $\mathcal{C}$, together with three morphisms: a \textit{multiplication map}  $m\colon G\times G\to G$, a \textit{unit map}  $e\colon\ast \to G$, and an \textit{inverse map} $i\colon G\to G$, such that the following diagrams commute:
    \begin{center}
        \begin{tikzcd}
            &G \times G\times G \arrow[d, "m\times \id"'] \arrow[r, "\id \times m"]&G\times G \arrow[d, "m"] \\
            &G\times G \arrow[r, "m"]&G 
        \end{tikzcd}
        \quad 
        \begin{tikzcd}
            &G \arrow[r, "{(\id, e)}"] \arrow[d, "{(e, \id)}"'] \arrow[dr, "\id"]&G\times G \arrow[d, "m"] \\
            &G\times G \arrow[r, "m"] &G 
        \end{tikzcd}

           \begin{tikzcd}
            &G \arrow[r, "{(i, \id)}"] \arrow[d, "{(\id, i)}"'] \arrow[dr, "e'"]&G\times G \arrow[d, "m"] \\
            &G\times G \arrow[r, "m"] &G 
        \end{tikzcd}
\end{center}
Where the pairing $(f, g)$ means the map $(f\times g)\circ d$, where $d\colon G\to G\times G$ is the diagonal morphism. Moreover, the map $e'$ is the unique map $G\to \ast\xrightarrow[]{e}G$. These diagrams respectively encode the associativity of the multiplication map, the identity element, and the existence of inverses. 

Let $G$ and $G'$ be group objects in to $\mathcal{C}$. A morphism $f\colon  G\to G'$ is a \textit{morphism of group objects} if the following diagrams commute:
\begin{center}
    \begin{tikzcd}
        &G\times G \arrow[d, "f\times f"']\arrow[r, "m"]&G \arrow[d, "f"] &\ast \arrow[dr, "e'"']\arrow[r, "e"]&G \arrow[d, "f"]&G \arrow[d, "f"]\arrow[r, "i"]&G \arrow[d, "f"]\\
        &G'\times G' \arrow[r, "{m'}"] &G' & &G' &G' \arrow[r, "i'"]&G' 
    \end{tikzcd}
\end{center}
The group objects in $\mathcal{C}$ together with their morphisms form a category which we will denote by $\textbf{Grp}(\mathcal{C})$.
\end{Def}

\begin{Def}\label{def:gp scheme as a group object}
    Let $k$ be a ring. Recall that $\Spec(k)$ is a terminal object in the category of affine schemes over $k$. We define an \textit{affine group scheme over} $k$ as a group object in the category of affine schemes over $k$. In other words, affine group schemes are objects of $\textbf{Grp}(\textbf{AffSch}_k) \simeq \textbf{Grp}((\textbf{CAlg}_k)^{\text{op}})$.
\end{Def}

As we will see, there is an alternative definition of an affine group scheme over $k$. Fortunately, both definitions are equivalent (see Remark \ref{equivalence of definitions}).

\begin{Def}[\textbf{Alternative definition}]
A \textit{$k$-group functor} is a covariant  functor $G\colon \textbf{CAlg}_k \to \textbf{Grps}$. An \textit{affine group scheme over $k$} is a $k$-group functor that is representable. 
If $G$ is an affine group scheme, then we will denote its representing object by $k[G]$, called its \textit{coordinate algebra}. In other words, for all commutative $k$--algebras $A$ we have 
\[
G(A)\simeq \Hom_{\textbf{CAlg}_k}(k[G], A).
\]
\end{Def}

\begin{Rem}[\textbf{Equivalence of the two definitions}]\label{equivalence of definitions} It is not so hard to see that these two definitions of affine group schemes are equivalent, via the Yoneda embedding. One easily sees that the category of affine schemes over $k$ is equivalent to the full subcategory of $\text{Fun}(\textbf{CAlg}_k, \textbf{Set})$ consisting of the representable functors with values in $\mathbf{Grp}$. In other words, for some commutative $k$--algebra $A$, the affine scheme $\Spec(A)$ corresponds to the representable functor $h^A=\Hom(A, -)$. In particular, if $G=\Spec(A)$ is an affine group scheme, then one applies the Yoneda embedding the diagrams defining the group structure in Definition \ref{def:group object}. Indeed, recall that products of affine schemes over $k$ are still affine over $k$, since they correspond to taking the tensor product over $k$ of their coordinate rings, and that the Yoneda embedding preserves these finite products. Hence the Yoneda embedding gives us the same diagrams in $\text{Fun}(\textbf{CAlg}_k, \textbf{Set})$ where $G$ has been replaced with $h^A$. For example:
\begin{center}
\begin{tikzcd}
      &h^A\times h^A \times h^A \arrow[r]\arrow[d]&h^A \times h^A \arrow[d]\\
  &h^A\times h^A \arrow[r]&h^A 
\end{tikzcd}
\end{center}
Since the arrows in these new diagrams are natural transformations, then we can feed these diagrams any other commutative $k$--algebra $B$ to get a commutative diagram in the category $\textbf{Set}$:
\begin{center}
\begin{tikzcd}
      &h^A(B)\times h^A(B) \times h^A(B) \arrow[r]\arrow[d]&h^A(B) \times h^A(B) \arrow[d]\\
  &h^A(B)\times h^A(B) \arrow[r]&h^A(B) 
\end{tikzcd}
\end{center}
But these new diagrams will precisely define $h^A(B)=\Hom(A, B)$ as a group object in \textbf{Set}. In other words, $\Hom(A, B)$ is just a group. This is to say that the representable functor $h^A\colon \textbf{CAlg}_k\to \textbf{Set}$ actually factors through the category of groups whenever $G=\Spec(A)$ is an affine group scheme. The converse equivalence is much the same. Let $G=\Spec(A)$ be any affine scheme over $k$. Then whenever $h^A$ factors through the category of groups we can easily see that indeed $G$ is an affine group scheme.
\end{Rem}

\begin{Rem}
  Let $G=\Spec (k[G])$ be an affine group scheme, where $k[G]$ is its coordinate $k$--algebra. From the perspective of $k[G]$, by duality, the multiplication map $m\colon G\times G\to G$ corresponds to a map $\Delta\colon k[G]\to k[G]\otimes_k k[G]$, the identity map $e\colon\Spec k\to G$ corresponds to a map $\varepsilon\colon k[G]\to k$, and the inverse map $i\colon G\to G$ corresponds to a map $S\colon k[G]\to k[G]$. One easily sees that the diagrams which are dual to those in Definition \ref{def:group object} are the same as those in the definition of a Hopf algebra. These maps therefore endow $k[G]$ with the structure of a (commutative) Hopf algebra. This extends to a duality:
  \[
  \textbf{Grp}(\textbf{AffSch}_k) \simeq (\textbf{CHopf}_k)^{\text{op}}.
  \]
 Let $A$ be some commutative $k$--algebra with multiplication map $\alpha$. We can now easily describe the group structure on $G(A)=\Hom_{\textbf{CAlg}_k}(k[G], A)$ in terms of the Hopf algebra structure on $k[G]$. Note that the isomorphism 
 \[
 h^{k[G]}(A) \times h^{k[G]}(A) \simeq h^{k[G]\otimes k[G]}(A)
 \]
  is given via the composition $(f, g) \mapsto \alpha \circ (f\otimes g)$. In fact, this isomorphism is natural in $A$.  Now, we compose with the map 
  \[
  \Delta^*\colon h^{k[G]\otimes k[G]}(A) \to h^{k[G]}(A)
  \]
  given by precomposition with $\Delta$. This is exactly the group structure on $G(A)$, to be precise, let $g, h\in G(A)$, then the product $gh$ is defined as the composite 
  \[
  k[G]\xrightarrow{\Delta} k[G] \otimes_k k[G] \xrightarrow{g\otimes h} A\otimes_k A \xrightarrow{\alpha} A.
  \]

\end{Rem}
\begin{Ex}\label{main examples} We give a few elementary examples:
    \begin{enumerate}[(i)]
    \item The \textit{additive group scheme} $\mathbb{G}_a=\Spec k[t]$. When viewed as the functor $\Hom(k[t], -)$, it sends a commutative $k$--algebra $A$ to the group $(A, +)$. One can see this easily since any $k$--linear map $k[t]\to A$ is defined by where it sends $t$. For some $a\in A$ we denote the map defined by $t\mapsto a$ by $f_a$, then $\Hom(k[t], A)$ has a natural group operation defined by $f_a + f_b=f_{a+b}$. Thus as groups we have an isomorphism $\Hom(k[t], A)\simeq (A, +)$. Let us now highlight the Hopf algebra structure on $k[t]$: The comultiplication is given by $\Delta(t)= t\otimes 1 + 1\otimes t$, the counit is given by $\varepsilon(t)=0$, and the antipode is given by $S(t)=-t$.
    \item The \textit{multiplicative group scheme} $\mathbb{G}_m=\Spec k[t, t^{-1}]$. Similar to the above example, when viewed as a functor $\Hom(k[t, t^{-1}], -)$ it is clear that this functor sends any commutative $k$--algebra $A$ to the group $(A^\times, \cdot)$, since any map $k[t, t^{-1}]\to A$ is determined by the image of $t$. The Hopf algebra structure on $k[t, t^{-1}]$ is given by the maps $\Delta(t)=t\otimes t$, $\varepsilon(t)=1$, $S(t)=t^{-1}$.
    
    \item Let $V$ be a $k$--module. The $k$-group functor $V_a$ assigns to a $k$--algebra $A$ the group underlying the module $V\otimes_k A$, and to a homomorphism of $k$--algebras $f\colon A\to B$ the group homomorphism 
    \[
    \id_V \otimes f \colon V\otimes A\to V\otimes B. 
    \]
    If $W$ is another $k$--module, and $f\colon V\to W$ is a $k$--linear map, then we obtain a morphism 
    $\Tilde{f}\colon V_a\to W_a $  by   $\Tilde{f}_A=f\otimes \id_A$.
    If $V$ is projective over $k$, then $V_a$ is an affine group scheme. Indeed,  we can use the universal property of the symmetric algebra to find the coordinate algebra for this functor. Note that 
    \[
    V\otimes_k A \simeq\Hom_{\textbf{Mod}_k}(V^*, A) \simeq \Hom_{\textbf{CAlg}_k}(\Sym(V^*), A).
    \]
    Therefore, we can see that $\Sym(V^*)$ is the coordinate algebra for $V_a$. In particular, one easily sees that when $V$ is the free module of rank $n$, $V=k^n$, we have that the coordinate algebra for $V_a$ is $k[t_1, \dots, t_n]$, and thus $V_a\simeq \mathbb{G}_a^n$. 
    
    \item Let $V$ be a projective $k$--module. The \textit{general linear group} $\GL_V$ assigns to a $k$--algebra $A$ the group of $A$--linear automorphisms $\Aut_A(V\otimes_k A)$. Note that we have isomorphisms: 
    \[
    \Hom_{\textbf{Mod}_A}(V\otimes A, V\otimes A)\simeq \Hom_{\textbf{Mod}_k}(V, \Hom_{\textbf{Mod}_A}(A, V\otimes A)) \simeq \Hom_{\textbf{Mod}_k}(V, V\otimes A)
    \]

    Given a morphism of $k$--algebras $f\colon A\to B$, we get an induced morphism 
    \[
    \Tilde{f}\colon \Aut_A(V\otimes A)\to \Aut_B(V\otimes B)
    \]
    defined by the map 
    \[
    \Hom_{\textbf{Mod}_A}(V\otimes A, V\otimes A)\simeq\Hom_{\textbf{Mod}_k}(V, V\otimes A) \to \Hom_{\textbf{Mod}_k}(V, V\otimes B)\simeq\Hom_{\textbf{Mod}_B}(V\otimes B, V\otimes B).
    \]
    This map indeed sends automorphisms to automorphisms. In the case where $V=k^n$, then one easily sees that $\Aut_A(V\otimes A)$ is the group of $n\times n$ matrices with entries in $A$. In this case, the map on morphisms just corresponds to the pushforward under $f$ of the constants which appear in the matrix i.e. the invertible matrix $(a_{ij})$ maps to the invertible matrix $(f(a_{ij}))$, and  moreover we have that the coordinate algebra of $\GL_{k^n}$ is simply given by $k[t_{11}, t_{12}, \dots, t_{nn}, s]/(\det(t_{ij})s-1)$.
    
    \item Let $\Gamma$ be any finite group. We will define the \textit{constant group scheme} associated to $\Gamma$. We have a contravariant functor $G_\Gamma\colon \textbf{AffSch}_k \to \textbf{Grps}$ defined by sending an affine scheme $X$ to the set of locally constant functions $|X|\to \Gamma$. The group structure on $\Gamma$ induces the group structure on $G_\Gamma(X)$ by the multiplication of functions. Indeed, this functor isn't quite constant, but rather we have $G_\Gamma(X)\simeq \Gamma^{\pi_0(X)}$. Moreover, we can see that this functor is representable with coordinate algebra given by $k^{|\Gamma|}$. Let $X=\Spec(A)$ and let $f\in \Hom(k^{|\Gamma|}, A)$, then any such $f$ can be identified with the data of a product decomposition $A=\prod_{|\Gamma|}A_i$ of $A$ by looking at the images of the idempotents $(0, \dots, 1, \dots, 0)$, or equivalently, the data of a decomposition of $X$ into $|\Gamma |$ disjoint components (the decomposition can possibly repeat components). Such a decomposition of $X$ is easily seen to be equivalent to the data of a locally constant function $|X|\to \Gamma$.
    \item Let $G$ be an affine group scheme. A \textit{subgroup scheme} $H$ of $G$ is a closed subscheme of $G$ which is also a subgroup of $G$. This means that the coordinate algebra of $H$ is of the form $k[H]=k[G]/I$ where $I$ is a \textit{Hopf ideal}. A Hopf ideal $I\subset k[G]$ is an ideal such that the following three conditions hold: 
    \[
    \Delta(I)\subset I\otimes k[G] + k[G]\otimes I, \quad S(I)\subseteq I, \quad \mbox{and} \quad \varepsilon(I)=0.
    \]
    Let us denote the quotient map $k[G]\to k[H]$ by $\phi$, one sees that the quotient of a Hopf algebra by a Hopf ideal is once again a Hopf algebra with the Hopf algebra structure induced by $\phi$ (and therefore the group structures on $H$ and $G$ are compatible). The three conditions in the definition of a Hopf ideal correspond respectively to the fact that $H(A)$ is closed under multiplication, inversion, and contains the unit element, in other words, $H$ is a subgroup functor of $G$, i.e. that $H(A)=\Hom_{\textbf{CAlg}_k}(k[H], A)$ is a subgroup of $G(A)=\Hom_{\textbf{CAlg}_k}(k[G], A)$ for all $A$. 
    
    Let us show this. Firstly, let $g, h \in H(A)$, we need to show that their product (in $G(A)$) lives in $H(A)$. We can regard $g, h$ as elements of $G(A)$ by identifying them with the composites $k[G]\xrightarrow{\phi} k[H] \xrightarrow{g} A$ which vanish on $I$. Conversely, any map $k[G]\to A$ which vanishes on $I$ identifies bijectively with a map $k[H]\to A$. The product $gh$ (in $G(A)$) is the map: 
    \[
    k[G]\xrightarrow{\Delta} k[G] \otimes_k k[G] \xrightarrow{g\otimes h} A\otimes_k A \xrightarrow{\alpha} A.
    \]
    Our goal is to show that this composite vanishes on $I$, so that it corresponds to a map $k[H]\to A$, i.e. an element of $H(A)$. Note that $\Delta(I)\subset I\otimes k[G] + k[G]\otimes I$ by assumption that $I$ is a Hopf ideal, and since both $g$ and $h$ vanish on $I$ therefore $g\otimes h$ vanishes on $\im \Delta|_I$. Therefore the composite $(g\otimes h)\circ \Delta$ vanishes on $I$, and so does the product $gh$, as desired. We will now show that $H(A)$ is closed under inversion. Suppose $g\in H(A)$, we want to show $g^{-1}\in H(A)$. As before, if $g$ vanishes on $I$, and $S(I)\subseteq I$, then $g$ certainly vanishes on $\im S|_I$. But $g^{-1}$ is given by the composite $k[G]\xrightarrow{S} k[G]\xrightarrow{g} A,$
    which therefore vanishes on $I$, showing that $g^{-1}\in H(A)$. Finally, we will show that $H(A)$ contains the unit element. The unit element in $G(A)$ is given by the morphism $k[G]\xrightarrow{\varepsilon} k \xrightarrow{\eta_A} A$, where $\eta_A$ is the unit morphism for $A$ as a $k$--algebra. By our assumption that $\varepsilon(I)=0$ we see that the unit element in $G(A)$ vanishes in $I$, and thus identifies with an element of $H(A)$. 
    
    What we have shown that subgroup schemes correspond precisely to closed subschemes which are also subgroup functors.
    
    \item Suppose $k$ has characteristic $n\geq 1$, and $A$ be any $k$--algebra. There is a functor sending $A$ to the group $\alpha_n(A)=\{a\in A \mid a^n=0\}$ under addition. It is easy to see that this group scheme is represented by $k[t]/t^n$. Moreover, clearly $\alpha_n$ is a closed subscheme of $\mathbb{G}_a$ defined by a Hopf ideal, and is therefore a subgroup scheme of $\mathbb{G}_a$.
    
    \item Let $n\geq 1$. The $n$\textit{th roots of unity} $\mu_n$ sends any commutative $k$--algebra $A$ to the group of $n$th roots of unity in $A$, that is, the group $\mu_n(A)=\{a\in A \mid a^n=1\}$. Similar to the above examples, it is easy to see that this group scheme is represented by $k[t]/(t^n-1)$. Moreover, clearly $\mu_n$ is a subgroup scheme of $\mathbb{G}_m$.
    
      \item Let $k$ be a field of characteristic $p$, and consider the Frobenius morphism $F_k: k\to k$ defined by $x\mapsto x^p$. Let $H$ be a Hopf algebra over $k$, we can define a new Hopf algebra $H^{(p)}$ by looking at the pushout diagram:
    \begin{center}
    \begin{tikzcd}
         &k \arrow[d, "F_k"']\arrow[r, "\eta"]&H \arrow[d, dashed]\\
         &k \arrow[r, dashed]& H^{(p)}
    \end{tikzcd} 
    \end{center}
    We also have a Frobenius map $F_H:H\to H$ given by $h\mapsto h^p$ which clearly satisfies $\eta \circ F_k= F_H \circ \eta$. Now using the universal property of this diagram we get an induced morphism of Hopf algebras $F_{H/k}: H^{(p)}\to H$ called the {relative Frobenius morphism}, that is:
        \begin{center}
    \begin{tikzcd}
         &k \arrow[d, "F_k"']\arrow[r, "\eta"]&H \arrow[ddr, bend left=40, "F_H"]\arrow[d] &\\
         &k \arrow[drr, bend right=40, "\eta"]\arrow[r]& H^{(p)} \arrow[dr, dashed, "F_{H/k}"]& \\
         & & &H 
    \end{tikzcd} 
    \end{center}
    We let $F_{H/k}^{(r)}:H^{(p^r)}\to H$ denote the $r$th iterative of the relative Frobenius morphism, and we call the affine group scheme corresponding to the Hopf algebra $\coker F_{H/k}^{(r)}$ the $r$\textit{th Frobenius kernel}, and we denote it by $G_{(r)}$. Note that the image of a map of Hopf algebras is clearly a Hopf ideal, therefore $G_{(r)}$ will always be a subgroup scheme of $\Spec H$

    For example, if $H=k[t]$, then $H^{(p)}=k[t]$, and $F^{(r)}_{H/k}$ is defined by $t\mapsto t^{p^r}$. Thus we see that $G_{(r)}$ is the affine group scheme $\Spec k[t]/(t^{p^r})=\alpha_{p^r}$. Similarly, if we take $H=k[t, t^{-1}]$ we get that $G_{(r)}=\Spec k[t]/(t^{p^r}-1)=\mu_{p^r}$.
    \end{enumerate}
\end{Ex}

 \begin{Def}
     Let $G$ be an affine group scheme over $k$. 
     \begin{enumerate}
         \item We will say $G$ is \textit{flat} if $k[G]$ is flat as a $k$--module.
         \item We will say $G$ is \textit{finite type} if $k[G]$ is finitely generated as a $k$--algebra.
         \item We will say $G$ is \textit{finite} if $k[G]$ is finitely generated as a $k$--module.
         \item We will say $G$ is \textit{\'etale} if $k[G]$ is a separable algebra.
     \end{enumerate}
 \end{Def}

 \begin{Rem}
    We are mainly interested in \'etale group schemes when the ground ring $k$ is a field of characteristic 0. In this case, a $k$--algebra $A$ is \textit{separable} if 
    \[
    A \otimes \Bar{k} = \Bar{k} \times \ldots \times \Bar{k}
    \]
    where $\Bar{k}$ denotes the algebraic closure of $k$. In fact, this is equivalent to $A$ being \textit{reduced}, i.e., it has no non-zero nilpotent elements. We refer to \cite[Section 6.2]{Wat79} for further details.
 \end{Rem}

\subsection*{Cartier's Theorem and \'Etale Group Schemes}

We will need a few results in characteristic zero, namely \textit{Cartier's Theorem}, of which we give an outline of the proof following \cite{Wat79}. In particular, Cartier's theorem shows in characteristic zero that all finite group schemes are \'etale. The upshot of this result is that we can show finite group schemes over an algebraically closed field of characteristic zero must be constant!

\begin{Th}[Cartier]\label{Thm:Cartier}
    Let $H$ be a finite dimensional Hopf algebra over a field $k$ of characteristic zero. Then $H$ is reduced.
\end{Th}
Let $I=\ker \varepsilon$ denote the augmentation ideal of $H$. Since $H$ is finite dimensional, so is $I/I^2$. Let $x_1,\dots, x_n \in I$ be elements whose residue classes give a basis for $I/I^2$

\begin{Lemma}\label{lemma-monomials}
    The monomials $x_1^{m_1}\cdots x_n^{m_n}$ with $\sum_i m_i=s$ are a basis for $I^s/I^{s+1}$
\end{Lemma}

\begin{proof}
Let $\Omega_s$ denote the subset of $\mathbb{N}^n$ consisting of the tuples $m=(m_1, \dots, m_n)\in \mathbb{N}^n$ with $\sum_i m_i=s$, and for brevity we will denote $x_1^{m_1}\cdots x_n^{m_n}$ by $x^m$. Furthermore, let $1_i$ denote the element of $\mathbb{N}^n$ with $1$ in the $i$th position and $0$ everywhere else. We have a map 
\[
g\colon  H\to I/I^2, \quad h\mapsto h-\varepsilon(h)1_H
\]
which can be seen to be a derivation of $H$--modules. For each $i\in\{1, ..., n\}$  may also define a $k$--linear map $f_i\colon  I/I^2\to k$ by $x_i \mapsto \delta_{ij}$ which can be extended to a $H$--linear map $\hat{f}_i: H\otimes I/I^2 \to H$ by $a\otimes b \mapsto f_i(b)a$. We may now define a family of derivations $d_i:H\to H$ by the composites 
\[
H\xrightarrow[]{\Delta}H\otimes H \xrightarrow[]{\id_H\otimes g} H\otimes I/I^2 \xrightarrow[]{\hat{f}_i} H,
\]
with the property that $d_i(x_j)=\delta_{ij}\mod I$. Let $\sum_i m_i=s$, then since these $d_i$ are derivations we have that $d_i(x_1^{m_1}\cdots x_n^{m_n})=m_ix_1^{m_1}\cdots x_i^{m_i-1}\cdots x_n^{m_n} \mod I^s$. In particular, $d_i(I^s) \subseteq I^{s-1}$. To show that these monomials $x^m$ for $m\in \Omega_s$ form a basis for $I^s/I^{s+1}$ we must show linear independence. Let $\sum_{m\in \Omega_s} \alpha_m x^m =0 \mod I^{s+1}$, then applying $d_i$ to both sides we see that 
\[
\sum_{m\in \Omega_s} \alpha_n m_i x^{m-1_i}=0 \mod I^s
\]
 and by induction we may conclude that $\alpha_mm_i=0$. Since we are in characteristic zero we must therefore have that $\alpha_m=0$ as desired.
\end{proof}

\begin{Lemma}
    Any square zero element $h\in H$ is contained in $\bigcap_j I^j$.
\end{Lemma}

\begin{proof}
    Suppose that $h^2=0$ and that there exists some $j+1$ for which $h\notin I^{j+1}$, so that $h$ is non-zero in $I^{j}/I^{j+1}$. We can write 
    \[
    h=\sum_{m\in \Omega_{j}}\alpha_m x^m \mod I^{j+1}.
    \]
    Squaring $h$ and using the Lemma \ref{lemma-monomials} we see that $(\sum_{m\in \Omega_{j}}\alpha_m x^m)^2$ is non-trivial modulo $I^{2j+1}$, which is a contradiction since it must be congruent to $h^2=0$.
\end{proof}

\begin{proof}[Proof of Theorem \ref{Thm:Cartier}] We may embed $k$ into its algebraic closure, and assume without loss of generality that $k$ is algebraically closed. Since $H$ is a finite dimensional algebra it is a quotient of the polynomial ring $k[x_1, \cdots, x_\ell]$ by some ideal $J$. Moreover, since we are over an algebraically closed field of characteristic zero, Hilbert's Nullstellensatz tells us that the maximal ideals of $H$ are of the form 
\[
\mf{m}_a=(x-a_0, \dots, x-a_\ell)
\]
for $a_i \in k$ with $f(a_1, \dots, a_\ell)=0$, for all $f\in J$. In particular, the ideal $\mf{m}_a$ is the kernel of the evaluation map at $a=(a_1, \dots, a_\ell)$, that is, it is the kernel of the map $g_a:H\to k$ defined by $f \mapsto f(a_1, \dots, a_\ell)$. We have a map of $k$--algebras: 
\[
\psi_{g_a}\colon H\xrightarrow[]{\Delta} H\otimes H\xrightarrow{g_a\otimes \id_H}k\otimes H \simeq H,
\]
which is an isomorphism with inverse $\psi_{g_a \circ S}$. Now one calculates $\psi_{g_a}(\mf{m}_a)=I$.

Let $h\in H$ be square zero, then we have that $\psi_{g_a}(h)$ is square zero, and thus contained in $\bigcap_j I^j$. Since $\psi_{g_a}$ is a bijection we see that $h\in \bigcap_j \mf{m}_a^j$. Now Krull's intersection theorem tells us that $h=0$.
\end{proof}

\begin{Cor}
    Finite group schemes defined over a field of characteristic zero are \'etale.
\end{Cor}

\begin{proof}
    For a finite dimensional algebra $A$ over a field, separable is equivalent to reduced. See \cite[Section 6.2]{Wat79}. 
\end{proof}

\begin{Prop}\label{Prop:etale in charc 0 are constant}
Any \'etale group scheme $G$ over an algebraically closed field $k$ of characteristic 0 is constant.  
\end{Prop}

First, we need to make some preparations. In what follows, let $k_s$ denote the separable closure of $k$, and let $\mathcal{G}$ denote the group of automorphisms of $k_s$ over $k$. Let  $\mathbf{CAlg_k^\textrm{sep}}$ denote the category whose objects are separable $k$--algebras and whose morphisms are given by homomorphisms of $k$--algebras, and let  $\mathcal{G}$-$\mathbf{FinSet}$ denote the category whose objects are finite sets with a continuous action of $\mathcal{G}$, with morphisms given by $\mathcal{G}$-equivariant maps.

\begin{Prop} 
The category $\mathbf{CAlg_k^\textrm{sep}}$ is equivalent to the opposite of the category $\mathcal{G}$-$\mathbf{FinSet}$. In symbols, 
\[
  \mathbf{CAlg_k^\textrm{sep}} \simeq \mathcal{G}\textrm{-}\mathbf{FinSet}^\textrm{op}.
\]
\end{Prop}

\begin{proof} For a separable $k$--algebra $A$, let $X_A$ be $\Hom_k(A,k_s)$ equipped with the $\mathcal{G}$--action induced by the action of $\mathcal{G}$ on $k_s$. This clearly defines a functor 
\[
X_{-}\colon\mathbf{CAlg_k^\textrm{sep}} \to \mathcal{G} \textrm{-}\mathbf{FinSet}^\textrm{op}.
\]
We claim that this is an equivalence. First, we will show that $X_{-}$ is fully faithful. Let $A$ be a separable $k$--algebra. Then we can recover $A$ as the $\mathcal{G}$-fixed points of the ring of functions $X_A\to k_s$. Indeed, the latter is isomorphic to $A\otimes_k k_s$ via 
\[
a\otimes \sigma \mapsto (f\mapsto f(x)=x(a)\sigma)
\]
and this isomorphism is $\mathcal{G}$-equivariant if we let $\mathcal{G}$ act on $A\otimes_k k_s$ on the right by automorphisms and trivially on the left. Hence $(A\otimes_k k_s)^\mathcal{G}\cong A$. In particular, for an equivariant map $X_B\to X_A$, we get a map between the rings of functions $X_B\to X_A\to k_s$ and taking $\mathcal{G}$--fixed points give us a morphism $A\to B$. We deduce that $X_{-}$ is fully faithful. 
It remains to show that $X_{-}$ is essentially surjective. Let $X$ be a transitive $\mathcal{G}$--set, so assume that $X$ has the form $\mathcal{G}_{x_0}$. Then one can choose a finite Galois extension $L$ of $k$ such that the action of $\mathcal{G}$ factors through the Galois group $\mathrm{Gal}(L/k)$. Consider the isotropy group $H$ of $x_0$, and let $A$ be the subfield of $L$ fixed by $H$. It follows that $X_A\to A$ is a $\mathcal{G}$-equivariant isomorphism, which sends $x_0$ to the inclusion $A\to L$. The general case follows by noticing that the disjoint union of $X_A$ and $X_B$ is isomorphic to $X_{A\times B}$. 
\end{proof}

\begin{Rem}
    One can use the Hopf algebra structure on $k[G]$ to show that the $\mathcal{G}$--sets from the previous proposition have a group structure. In fact, one obtains an equivalence  $\mathbf{CAlg_k^\textrm{sep}}$ with the opposite of the category of finite groups with a continuous action of $\mathcal{G}$ by group homomorphisms.
\end{Rem}

\begin{proof}[Proof of Proposition \ref{Prop:etale in charc 0 are constant}] 
Since algebraically closed is equivalent to separably closed in characteristic zero, we deduce that $\mathcal{G}$ is trivial. Hence, there is a finite group $\Gamma$ such that $G$ is the constant group scheme mapping a $k$--algebra $A$ to $\Gamma^{\pi_0(\Spec (A) )}$ as defined in Example \ref{main examples}(v). Unpacking the above equivalences, one checks that $\Gamma=\Hom_k(k[G],k)$.
\end{proof}

\pagebreak
\section{Representations of Affine Group Schemes}

In this section, we introduce the notion of a module over an affine group scheme and compare it with the concept of a comodule over its coordinate algebra. Recall that $k$  denotes an arbitrary commutative ring throughout.

\begin{Def}
    Let $G$ be an affine group scheme over $k$, and $V$ be a $k$--module. A \textit{representation} of $G$ on $V$ is a natural transformation $\rho\colon  G\times V_a\to V_a$ such that, for all $k$--algebras $A$, the map 
    \[
    \rho_A\colon G(A)\times V\otimes A\to V\otimes A
    \]
    is a group action and $G(A)$ acts on $V\otimes A$ by $A$--linear automorphisms.

    Let $(W,\tau)$ be another representation of $G$. A $k$--linear map $f\colon V\to W$ is a \textit{morphism of representations} if the following diagram commutes
    \begin{center}
        \begin{tikzcd}
            &G\times V_a \arrow[r, "\rho"] \arrow[d, "\id\times \Tilde{f}"'] &V_a \arrow[d, "\Tilde{f}"] \\
            &G\times W_a \arrow[r, "\tau"']&W_a. 
        \end{tikzcd}
    \end{center}
    We will refer to a representation of $G$ as a \textit{$G$--module}, and a morphism of representations as a \textit{$G$--module homomorphism}. We will denote the category of $G$--modules and $G$--module maps by $\mathbf{Rep}(G)$.
\end{Def}

\begin{Rem}
     If $V$ is free, then one sees that the data of a representation of $G$ on $V$ is equivalent to the data of a morphism of affine group schemes $\rho\colon  G\to GL_V$.
\end{Rem}

\begin{Ex}\label{examples 2} We give some simple examples of representations of affine group schemes.
\begin{enumerate}[(i)]
    \item Picking any $k$--module $V$ and the identity group action yields the \textit{trivial representation}; that is, for $g\in G(A)$ and $v\otimes a \in  V\otimes A$ we define $\rho(g)(v\otimes a)=v\otimes a$.
    \item One dimensional representations of $G$, i.e., representations of $G$ on $k$. Those correspond precisely to morphisms $G\to \mathbb{G}_m$, and are called \textit{characters}. We call 
    \[
    \chi(G)\coloneqq\Hom(G, \mathbb{G}_m)
    \]
    the \textit{character group}. It is easy to see that this forms a group under pointwise multiplication. Note that different characters give different representations of $G$ on $k$.
    \item If we pick $V=k[G]$, then there is the so-called \textit{right  regular representation $\rho_r$}. We can view $k[G]\otimes_k A$ as the collection of scheme homomorphisms $G_A\to \mathbb{A}^1_A$, where $\mathbb{A}^1_A=\Spec( k[t]\otimes_k A) $ denotes the affine line over $A$ and $G_A=\Spec k[G]\otimes_k A$ denotes the base change of $G$. Moreover, from the functor of points perspective, we can think of $G(A)$ as the points over $A$ of the scheme $G$, or equivalently, the points of $G_A$. Now let $g\in G(A)$ and 
    \[
    f\in k[G]\otimes_k A \simeq \Nat(G_A, \mathbb{A}^1_A),
    \]
    we define $\rho_r(g)(f)(x)=f(xg)$ for all $x\in G(A)$. Analogously, we have a \textit{left regular representation $\rho_l$} by $\rho_l(g)(f)(x)=f(g^{-1}x)$ for all $x\in G(A)$. 

    Alternatively from the algebraic perspective, given $x\in k[G]$ with $\Delta(x)=x_i\otimes f_i$, the right regular representation is more explicitly given by 
    \[
    \rho(g)(x\otimes a)=\sum_i x_i\otimes g(f_i)a, \quad \mbox{for all }  g\in G(A), \mbox{ and } a\in A.
    \]
\end{enumerate}
\end{Ex}

\begin{Th}\label{comodules}
    Let $G$ be an affine group scheme over $k$. Then there is an equivalence of categories 
    \[
    \mathbf{Rep}(G) \simeq \mathbf{coMod}_{k[G]}.
    \]
    Furthermore, if $G$ is finite and flat, and $k$ is Noetherian, then there is an equivalence of categories 
    \[
    \mathbf{Rep}(G) \simeq \mathbf{Mod}_{k[G]^*}. 
    \]
\end{Th}

\begin{proof}
    To prove the first statement; let $(V, \Delta_V)$ be a $k[G]$--comodule. Then we may define a representation as follows. Pick any commutative $k$--algebra $A$, and some $g\in G(A)=\Hom(k[G], A)$. Then we get  $k$--linear map 
    \[
    V\xrightarrow[]{\Delta_V} V\otimes_k k[G] \xrightarrow[]{\id_V\otimes g} V\otimes_k A,
    \]
    which corresponds to an $A$--linear map $V\otimes_k A\to V\otimes_k A$ by tensoring everywhere with $A$ and then composing with the multiplication map on $A$. More precisely, let $x\otimes a \in V\otimes A$ and $v(x)=\sum_i x_i\otimes f_i$. Then for $g\in G(A)$ our representation is defined by 
    \[
    \rho(g)(x\otimes a)=\sum_i x_i \otimes g(f_i)a. 
    \]
     It is easy to see that this is a representation of $G$.

    Conversely, given a representation $\rho\colon G\times V_a\to V_a$, where $V$ is some $k$--module, we may define a $k[G]$--coaction $\Delta_V$ on $V$ as follows. Consider the identity morphism $\id_{k[G]} \in G(k[G])=\Hom(k[G], k[G])$, then we can define a map
    \[
    \Delta_V\colon V \to V\otimes_k k[G], \quad v\mapsto \rho(\id_{k[G]})(v\otimes 1).
    \]
      It is easy to verify that  this is a coaction on $V$.

    One checks that these constructions are mutually inverse, and that this assignment behaves correctly with morphisms. 

    For the second statement we can just prove that $\textbf{Mod}_{k[G]^\ast}\simeq \mathbf{coMod}_{k[G]}$. From assumptions on $G$ and the fact $k$ is Noetherian we have that $k[G]$ is finitely presented and flat as a $k$--module. It is well known that such a $k$--module is finitely presented and flat if and only if it is finitely generated and projective. Thus the result follows from Theorem \ref{equivalence}. 
\end{proof}

\begin{Rem}
   By unpacking the assignments from the previous theorem, one observes that the right regular representation of  $G$ corresponds to $k[G]$ with coaction given by $\Delta_{k[G]}$. On the other hand, the left regular representation corresponds to $k[G]$ with coaction 
   \[
   (12)\circ(S_{k[G]} \circ \mathrm{id}_{k[G]})\circ \Delta_{k[G]}
   \]
    where $S_{k[G]}$ denotes the antipode and $(12)$ is the swapping map.  
\end{Rem}

\begin{Prop}\label{Thm: Group schemes over fields are embedding in GL_n}
     Let $G$ be an affine group scheme over $k$, such that the coordinate algebra $k[G]$ is a free $k$-module of rank $n$. Then there is a closed immersion $G\hookrightarrow \GL_n$.
 \end{Prop}

 \begin{proof}
Let $\{e_1, \dots , e_n\}$ be a $k$-basis for $k[G]\simeq k^n$. The right regular representation gives a $k[G]$-comodule structure on $k[G]\simeq k^n$, which, as we have seen, is equivalent to the data of a morphism of group schemes $G\to \GL_n$. We want to check that this morphism is a closed immersion. A morphism of group schemes gives a morphism of coordinate algebras in the opposite direction, so we get a map of coordinate algebras $k[\GL_n]\to k[G]$. From Example \ref{main examples} (iv), we see that the coordinate algebra of $\GL_n$ is given by $k[t_{11}, t_{12}, \dots, t_{nn}, s]/(\det(t_{ij})s-1)$. To prove that this morphism is a closed immersion we will equivalently show that this morphism on coordinate algebras is surjective.

Now, write $\Delta(e_i)=\sum_j e_j\otimes\alpha_{ij}$. The counit axiom gives us that $e_i=\sum_j \varepsilon(e_j)\alpha_{ij}$. From the description of the coordinate algebra of $\GL_n$ in Example \ref{main examples} (iv), it is not so hard to see that the morphism $k[\GL_n]\to k[G]$ is defined by $t_{ij}\mapsto \alpha_{ij}$. In particular, from the consequence of the counit axiom above, the $e_i$'s are in the image of this map. Thus this map hits a basis of the target, and is hence surjective.
 \end{proof}

\pagebreak
\section{Invariants and Cohomology}

In this section, we will only consider flat affine group schemes. In this case, the category of representations is abelian by Theorem \ref{abelian} and this category has enough injectives by Theorem \ref{injectives}, so talking about the exactness of functors and taking derived functors makes sense. 

\begin{Def}
    Let $(V, \rho)$ be a representation of $G$. We define the \textit{invariants} $V^G$ of $V$ as follows:
\[
V^G=\{x\in V \mid x\otimes 1\in (V\otimes_k A)^{G(A)}, \text{ for all $k$--algebras $A$}\}.
\]
Where $(V\otimes_k A)^{G(A)}$ denotes the fixed points of $V\otimes_k A$ under the group action $\rho_A$.
\end{Def}

\begin{Lemma}\label{lemma:invariants-coinvariants}
    Let $(V, \rho)$ be a representation of $G$. Then 
    \[
    V^G=\{x\in V \mid \Delta_V(x)=x\otimes 1\}
    \]
    where $\Delta_V$ is the coaction defined on $V$ in Theorem \ref{comodules}. Moreover,  we have a natural isomorphism $V^G\simeq\Hom^{k[G]}(k, V)$, where $k$ is given the structure of the trivial representation.
\end{Lemma}
\begin{proof}
    Clearly if $x\in V^G$, then picking $A=k[G]$ shows that 
    \[
    x\otimes 1 \in (V\otimes_k k[G])^{G(k[G])}.
    \]
    In other words,  $x\otimes 1$ is invariant under the action of $G(k[G])$. In particular,  acting via the identity morphism $\id_{k[G]} \in G(k[G])$ must also therefore fix $x\otimes 1$. This action is precisely the definition of the coaction $\Delta_V$. Therefore $\Delta_V(x)=x\otimes 1$.

    Conversely, suppose $\Delta_V(x)=x\otimes 1$, and let 
    \[
    f\in G(A)=\Hom_{\textbf{CAlg}_k}(k[G], A)
    \]
    be any morphism. We want to show that the action of $f$ on $x \otimes 1 \in V\otimes_k A$ via $\rho_A$ does nothing. Since $\rho$ is by definition a natural transformation of functors, the map $f\colon k[G]\to A$ gives us a commutative square: 
    \begin{center}
        \begin{tikzcd}
            &G(k[G]) \times V\otimes_k k[G] \arrow[d]\arrow[r, "\rho_{k[G]}"]&V\otimes_k k[G] \arrow[d] \\
            &G(A) \times V\otimes_k A \arrow[r, "\rho_{A}"']&V\otimes_k A.
        \end{tikzcd}
    \end{center}
    If we fix $\id_{k[G]} \in G(k[G])$ then top horizontal map is just the definition of the coaction $\Delta_V$ on elements of the form $x'\otimes 1$ for $x'\in V$. Therefore for our given $x$ the top path in the square gives us 
    \[
    (\id_{k[G]}, x\otimes 1) \mapsto \Delta_V(x)=x\otimes 1 \mapsto x\otimes 1
    \]
     since we are assuming $\Delta_V(x)=x\otimes 1$. Now if we follow the bottom path we see that the vertical map gives us $(\id_{k[G]}, x\otimes 1) \mapsto (f, x\otimes 1)$, and we must therefore have that the horizontal map sends $ (f, x\otimes 1)$ to $x\otimes 1$. In other words, the action of $f$ on $x\otimes 1$ must be trivial. Therefore $x\in V^G$ since $f\in G(A)$ was arbitrary.
    
    The second claim now follows easily using the first; given a map of representations $\varphi\colon k\to V$, the assignment $\varphi\mapsto \varphi(1)$ gives us the desired bijection. Naturality in $V$ is easily checked.
\end{proof}

\begin{Rem}
    Note that $\{x \in V \mid \Delta_V(x) = x \otimes 1\}$ is precisely the definition of the $k[G]$-coinvariants of the comodule $(V, \Delta_V)$, as given in Definition \ref{def:coinvariants}, and denoted by $V^{\mathrm{co}k[G]}$. When the context is clear, we use these notations interchangeably without further comment.
\end{Rem}

\begin{Cor}
    Let $G$ be a flat affine group scheme. Then the functor $(-)^G\colon \mathbf{Rep}(G)\to \mathbf{Ab}$ is left exact. 
\end{Cor}
\begin{proof}
   By Lemma \ref{lemma:invariants-coinvariants},  we get that the functor $(-)^G$ is naturally isomorphic to the representable functor $\Hom^{k[G]}(k, -)$. The latter  is left exact.
\end{proof}

We have seen that the category of $k[G]$--comodules has enough injectives whenever $k[G]$ is flat. Therefore we can form injective resolutions following the procedure in Theorem \ref{injectives}, and take the right derived functors of the functor $(-)^G$. 

\begin{Def}
    We denote the $i$th right derived functor of $(-)^G$ by $H^i(G, -)$, and we will call it the $i$\textit{th group cohomology functor}. For a $G$--module $N$, we denote the $i$th right derived functor of $\Hom^{k[G]}(N,-)$ by $\mathrm{Ext}_G^i(N,-)$, we will call it the $i$\textit{th extension group functor}. In particular, with this notation $\mathrm{Ext}^i_G(k,-)$ is simply $H^i(G,-)$.
\end{Def}

\begin{Def}
    Let $G$ be an affine group scheme over $k$. A $G$--\textit{algebra} is a $k$--algebra $A$ which also has the structure of a $k[G]$--comodule, and such that the multiplication map $m\colon A\otimes_k A\to A$ is a morphism of comodules. Note that $A\otimes_k A$ has the structure of a $k[G]$--comodule with coaction $\Delta_{A\otimes A}$ given by the composite 
    \[
     A\otimes A \xrightarrow{\Delta_A \otimes \Delta_A} A\otimes k[G] \otimes A\otimes k[G] \xrightarrow[]{\mathrm{id}_A\otimes(12) \otimes \mathrm{id}_A} A\otimes A \otimes k[G]\otimes k[G]\xrightarrow[]{\mathrm{id}_{A\otimes A}\otimes \lambda} A\otimes A\otimes k[G]
    \]
    where the map $\lambda$ denotes the multiplication of $k[G]$.  Furthermore, if $I \subseteq A$ is a subcomodule of $A$, we will say it is a $G$\textit{-subalgebra of $A$} if it contains the unit and is closed under multiplication, i.e. $m(I\otimes_k I) \subseteq I$. We will say $I$ is a $G$\textit{-ideal of $A$} if we have that $m(I\otimes_k A) \subseteq I$, in which case it is easy to see that the quotient has the structure of a $G$--algebra.
\end{Def}

\begin{Ex}
    The trivial $k[G]$--comodule is indeed a $G$--algebra with the structure of $k$--algebra given by multiplication. Recall that the  coaction  $k\to k\otimes k[G]$ is given by $a\mapsto a\otimes 1_{k[G]}$.  
\end{Ex}

\begin{Lemma}\label{AG is a k-algebra}
    Let $G$ be a flat affine group scheme over $k$, and $A$ be a $G$--algebra.  Then the multiplication on $A$ induces a $k$--algebra structure on invariants $A^G$.
\end{Lemma}

\begin{proof}
    Let $m \colon A\otimes_k A\to A$ denote the multiplication map on $A$, and $\Delta_A\colon A\to A\otimes_k k[G]$ denote the coaction. Clearly $A^G$ has the structure of a $k$--module and contains the unit element, so what we have to show is that it is closed under multiplication. Let $x, y \in A^G$. We have to show that $xy=m(x\otimes y)$ is also in $A^G$. Since the map $m$ is a map of comodules, we have a commutative diagram:
    \begin{center}
        \begin{tikzcd}
            &A\otimes_k A \arrow[d, "m"]\arrow[r, "\Delta_{A\otimes A}"]&A\otimes_k A\otimes_k k[G] \arrow[d, "m \otimes \id_{k[G]}"] \\
            &A \arrow[r, "\Delta_A"]&A\otimes_k k[G].
        \end{tikzcd}
    \end{center}
    To show that $xy=m(x\otimes y)$ is invariant we must show that $\Delta_A(m(x\otimes y))=xy\otimes 1$, this is just the bottom path in the commutative square, applying the top path to $x\otimes y$ gives us 
    \[
    \Delta_{A\otimes A}(x\otimes y)=x\otimes y\otimes 1 \mbox{  and  } m\otimes \id_{k[G]}(x\otimes y\otimes 1)=xy\otimes 1.
    \]
      Therefore we have shown that $\Delta_A(m(x\otimes y))=xy\otimes 1$ as desired.
\end{proof}

\begin{Def}
    Let $G$ be a flat affine group scheme over $k$. We say that $G(k)$ is \textit{dense in $G$} if there is no proper closed subfunctor $X\subset G$ such that $G(k)\subseteq X(k)$. 
\end{Def}

There are many cases of interest where $G(k)$ is dense in $G$. For instance, this happens when $k$ is a field and $G$ is a constant group scheme. Other examples include $\mathbb{G}_a$, $\GL_n$, and $\mathrm{SL}_n$ when $k$ is an infinite field. See \cite[Section 4.5]{Wat79}.

\begin{Prop}\label{Prop:naive invariants vs functorial invariants}
    Let $G$ be a flat affine group scheme over $k$ such that $G(k)$ is dense in $G$. Let $M$ be a $G$--module that is projective as a $k$--module. Then $M^G$ agrees with $M^{G(k)}$. 
\end{Prop} 

\begin{proof}
    For any subset $S$ of $M$, consider the $k$--functor $Z_G(S)$ given by 
    \[
    Z_G(S)(A)\coloneqq \{g\in G(A)\mid g(m\otimes 1)=m\otimes 1\text{ for all } m\in S\}
    \]
    where $A$ is a $k$--algebra. By \cite[Section I.2.12]{jantzen2003representations}, $Z_G(S)$ is a closed subfunctor of $G$, provided that $M$ is projective as a $k$--module. In particular, for 
    \[
    S'=\{m\in M\mid g(m\otimes 1)=m\otimes 1 \text{ for all } g\in G(k)\}
    \]
    we obtain that $G(k)=Z_G(S')(k)$. Therefore, $G=Z_G(S')$ as we wanted.  
\end{proof}

We will now define two important functors, called the induction and restriction functors. 

\begin{construction}
    Let $G$ be an affine group scheme, and let $H\subseteq G$ be a subgroup scheme. Equivalently, from the perspective of coordinate algebras, we have a surjective map of Hopf algebras $\phi\colon k[G]\to k[H]$, so that $k[H]\simeq k[G]/I$ for some Hopf ideal $I$. In this case we get that $H$ is a subgroup functor of $G$ (see Example \ref{main examples} (vi)).
Now if $(V, \rho)$ is any representation of $G$ on $V$, then by restriction of the group action we get a representation of $H$ on $V$, more precisely, for every commutative $k$--algebra $A$ we have a map 
\[
G(A)\times V\otimes_k A\to V\otimes_k A
\]
which restricts to a map 
\[
H(A)\times V\otimes_k A\to V\otimes_k A
\]
via the map 
\[
\phi^\ast\colon  H(A)=\Hom_{\textbf{CAlg}_k}(k[H], A)\to \Hom_{\textbf{CAlg}_k}(k[G], A)=G(A).
\]
Since this is natural in $A$, we get a functor $\mathbf{Rep}(G)\to \mathbf{Rep}(H)$. Viewing $V$ as a $k[G]$--comodule with coaction $\Delta_V$, we can see that this functor corresponds to the assignment $(V, \Delta_V) \mapsto (V, \Delta_{V'})$ where $\Delta_{V'}$ is the new coaction defined by composing with the quotient map 
\[
V\xrightarrow{\Delta_V} V\otimes_k k[G] \xrightarrow{\id_V\otimes \phi} V\otimes_k k[H].
\]
\end{construction}

\begin{Def}
    Let $H\subseteq G$ be a subgroup scheme. We call the functor $\mathbf{Rep}(G)\to \mathbf{Rep}(H)$ from the above construction the \textit{restriction functor}, and denote it by $\res^G_H$.
 \end{Def}

This restriction functor is part of an adjoint pair, we will now construct its right adjoint.

\begin{construction}
    Let $H$ now be a flat subgroup scheme of $G$, that is, suppose $k[H]$ is itself flat as a $k$--module. Consider some $k[H]$--comodule $(W, \Delta_W)$. Then $W\otimes_k k[G]$ has a $k[H]$--coaction given by the tensor product of  $\Delta_W$ and the canonical $k[H]$--coaction on $k[G]$ via the restriction of the regular representation. To be precise, the $k[H]$--coaction on $W\otimes_k k[G]$ is given by the composite map:
\[
W\otimes_k k[G] \xrightarrow{\Delta_W \otimes \Delta} W\otimes_k k[H] \otimes_k k[G] \otimes_k k[G] \xrightarrow{\id \otimes \id \otimes \id \otimes \phi} W\otimes_k k[H] \otimes_k k[G] \otimes_k k[H] \to W\otimes_k k[G] \otimes_k k[H]
\]
where $\phi\colon k[G]\to k[H]$ is the Hopf map defining $H$, and the last map multiplies the elements in $k[H]$ together. Since we have a $k[H]$--coaction we can take $H$--invariants and consider $(W\otimes_k k[G])^H$. Now we also have a $k[G]$--comultiplication on $W\otimes_k k[G]$ given by the identity action on $W$ tensored with the comultiplication on $k[G]$, that is: 
\[
W\otimes_k k[G] \xrightarrow{\id_W\otimes \Delta} W\otimes_k k[G]\otimes_k k[G].
\]
The claim is now that this map descends to a map 
\[
(W\otimes_k k[G])^H \xrightarrow{\id_W\otimes \Delta} (W\otimes_k k[G])^H\otimes_k k[G]
\]
which we leave as an easy but tedious exercise. This makes $(W\otimes_k k[G])^H$ into a $k[G]$--comodule, and one checks that this construction is functorial in $W$, giving us a functor $\mathbf{Rep}(H)\to \mathbf{Rep}(G)$.
\end{construction}

\begin{Def}
     If $H\subseteq G$ is a flat subgroup scheme, then we call the functor $\mathbf{Rep}(H)\to \mathbf{Rep}(G)$ from the above construction the \textit{induction functor}, and denote it by $\ind^G_H$.
\end{Def}

\begin{Prop}\label{adjuntion restriction induction}
    Let $G$ be a flat affine group scheme, and let $H\subseteq G$ be a flat subgroup scheme. Let $V$ be a $k[G]$--comodule and let $W$ be a $k[H]$--comodule. Then: 
    \[
    \Hom^{k[H]}(\res_H^GV, W) \simeq \Hom^{k[G]}(V, \ind^G_HW).
    \]
\end{Prop}

\begin{proof} 
   Consider the map $\epsilon_W$ given by the composition 
   \[
   (W\otimes k[G])^H\hookrightarrow W\otimes k[G] \xrightarrow{\mathrm{id}\otimes \epsilon} W\otimes k\cong W
   \]
   where $\epsilon$ denotes the counit of $k[G]$. It turns out that $\epsilon_W$ is a map of $k[H]$--modules. In fact,  composing with $\epsilon_W$ induces an isomorphism 
   \[
   \epsilon_W^\ast\colon  \Hom^{k[G]}(V, \ind^G_HW) \to \Hom^{k[H]}(\res_H^GV, W). 
   \]
   Indeed, an inverse is given by the map  that sends a $k[H]$--map $f\colon \res_H^G V \to W$  to the map $V\to \ind^G_H W$, $v\mapsto f(v)\otimes 1$. The details are left to the interested reader.   
\end{proof}

\begin{Rem}\label{flat subgroup scheme}
    Under some assumptions, the induction functor from $H$ to $G$ can be interpreted as a quasi-coherent sheaf on $G/H$, and in fact this allows us to derive some useful results we would otherwise not be able to. As before, let $H\subseteq G$ be a subgroup scheme of an affine scheme $G$, where $G$ and $H$ are both flat over a Noetherian ring $k$. Since $H$ is a subgroup functor of $G$, we have a functorial assignment of quotient groups 
    \[
    A\mapsto G(A)/H(A) \quad \textrm{for all }k\textrm{-algebras } A.
    \]
    This however is not the correct definition in general for the quotient $G/H$! Instead it is defined via a universal property: let the action of $H$ on $G$ be given on the left by the map $\alpha\colon  H\times G\to G$, and the projection to the second factor be given by $p\colon H\times G \to G$. Then the quotient scheme $G/H$ is given by the coequaliser of the pair $(\alpha$, $p)$ in the category of $k$--schemes. In other words, the quotient $G/H$ comes with a map $\pi\colon  G\to G/H$ and is the unique scheme with the property that any morphism of schemes $f\colon G\to X$ that is constant on $H$--orbits factors uniquely through $\pi$. The problem is whether $G/H$ even exists in the category of $k$--schemes, for a discussion of this see \cite[Chapter 5]{jantzen2003representations}. 
    
    In the situation where we will be applying this method, the quotient $G/H$ is actually an affine scheme: \textit{If $H$ is finite, then $G/H$ is an affine scheme} \cite[page 72 (6)]{jantzen2003representations}. Following \cite[5.8]{jantzen2003representations} we can associate to each representation $W$ of $H$ a quasi-coherent sheaf $\mathcal{L}(W)$ on $G/H$. Let $(W, \tau)$ be a representation of $H$ on $W$, and consider the quasi-coherent $\mathcal{O}_{G}$--module associated to the $k[G]$--module $k[G]\otimes_k W$. Now for any open set $U\subseteq G/H$ such that $\pi^{-1}U$ is affine, $H$ has a canonical diagonal action on the quasi-coherent $\mathcal{O}_G$--module $k[\pi^{-1}U]\otimes_k W$, where $H$ acts on $k[\pi^{-1}U]$ via the restricted regular representation, and $H$ acts on $W$ via $\tau$. The sheaf $\mathcal{L}(W)$ is then defined on $U$ by 
\[
\mathcal{L}(W)(U)=(k[\pi^{-1}U]\otimes_k W)^H.
\]
    Since sections are invariant under the $H$ action, this makes $\mathcal{L}(W)$ into a $\mathcal{O}_{G/H}$--module which can be seen to be quasi-coherent by \cite[5.9]{jantzen2003representations}. The global sections functor of this sheaf is precisely the induced representation functor. Moreover, since the induced representation functor is a right adjoint, it is left exact, and we can compute the derived induced representations equivalently as the sheaf cohomology of this quasi-coherent sheaf. In particular, sheaf cohomology of a quasi-coherent sheaf on an affine scheme vanishes by \cite[III 3.7]{hartshorne2013algebraic}, and since we are assuming that $G/H$ is affine, we can see that in this case the higher derived induced representations vanish. In other words, $\ind^G_H(-)$ is exact whenever $G/H$ is affine.
\end{Rem}

\begin{Def}
    A flat subgroup scheme $H$ of $G$ is \textit{exact} if the functor $\mathrm{ind}_H^G(-)$ is exact. 
\end{Def}

\begin{Prop}\label{frobenius reciprocity}
Let $H$ be a flat subgroup scheme of a flat affine group scheme $G$, $M$ be $G$--module and $N$ be an $H$--module. If $H$ is exact, then the following properties hold.
\begin{enumerate}
    \item \textit{Frobenius reciprocity:} For each non-negative integer $n$, there is an isomorphism 
    \[
    \Ext_G^n(M,\mathrm{ind}_H^G(N))\cong \mathrm{Ext}_H(\mathrm{res}^G_H(N),M).
    \]
    \item \textit{Shapiro's Lemma:} For each non-negative integer $n$, there is an isomorphism 
    \[
    H^n(G,\mathrm{ind}_H^G(N))\cong H^n(H,N).
    \]
\end{enumerate}
\end{Prop}

\begin{proof}
    This follows by Groethendieck spectral sequence which computes the derived functors of a composition of functors applied to $\Hom_G(M,-)\circ \mathrm{ind}_G^H$ since  this composition can by identified with $\Hom_H(N,-)$ by Proposition \ref{adjuntion restriction induction}. Explicitly, the spectral sequence is given by 
    \[
    E^{p,q}_2= \mathrm{Ext}^p(N, \mathrm{R}^q\mathrm{ind}_H^G(M))\Rightarrow \mathrm{Ext}^{p+q}_H(M,N)
    \]
    where $\mathrm{R}^q\mathrm{ind}_H^G(-)$ denotes the $q$th right derived functor of $\mathrm{ind}_H^G(-)$. Hence the result follows by the exactness of $H$. 
\end{proof}

\begin{Cor}\label{vanishing of cohomology for induced modules}
    Let $M$ be a $G$--module. Then $H^i(G,M\otimes k[G])=0$ for all $i>0$, and $H^0(G,M\otimes k[G])=M$.  
\end{Cor}

\begin{proof}
    The $G$--module $M\otimes k[G]$ agrees the module $M_{\textrm{triv}}\otimes k[G]$ where the action on the factor $k[G]$ is by the regular represenation, and $M_\textrm{triv}$ denotes $M$ with the trivial $G$--action. Moreover, by the \textit{tensor identity} (see \cite[Proposition I.3.6]{jantzen2003representations}) this also agrees with  $\mathrm{ind}_1^G(M)$. Now the result follows by Frobenius reciprocity since $\mathrm{ind}_1^G(-)=-\otimes k[G]$ is exact. 
\end{proof}

Let $G$ be a finite group scheme and $M$ be a $G$--module. We will describe the so called \textit{Hochschild complex} $C^\bullet(G,M)$  of $M$  which computes the cohomology groups $H^i(G,M)$.

\begin{construction}
    Fix $n\geq0$. Define $C^n(G,M)$ by $M\otimes_k k[G]^{\otimes n} $ and the differentials $\partial^n\colon C^n(G,M)\to C^{n+1}(G,M)$  by $\sum_{i=0}^{n+1} (-1)^i \partial^n_i $, where 
 \[
 \partial^n_i(m\otimes \lambda_1\otimes\ldots \otimes \lambda_n) = \left\{
        \begin{array}{ll}
            \varphi(m)\otimes \lambda_1\otimes\ldots\otimes \lambda_n & \textrm{if $i=0$,} \\
            m\otimes \lambda\ldots\otimes \Delta(\lambda_i)\otimes\ldots\otimes\lambda_n & \textrm{if $1\leq i<n$,}\\
            m\otimes\lambda_1\ldots\otimes\lambda_n\otimes 1 & \textrm{if $i=n+1.$}
        \end{array}
    \right.
 \]
Here, $\varphi$ denotes the comodule map of $M$, and $\Delta$ the comultiplication of the coordinate algebra $k[G]$. It is a routine exercise to verify that $\partial^{n+1}\circ\partial^{n}=0$, so that $(C^\bullet(G,M),\partial)$ is a chain complex. The key step to verify that $C^\bullet(G,M)$ computes $H^\ast(G,M)$ is the existence of an exact complex 
\[
0\to M \to M\otimes k[G]\to M\otimes k[G]^{\otimes 2} \to \ldots
\]
where the action of $G$ on $M\otimes k[G]^{\otimes n}$ is trivial on $M$, and by the left regular representation on the first factor of $k[G]^{\otimes n}$, and trivial on the rest. The point is that each $M\otimes k[G]^{\otimes n}$ is acyclic (see Corollary \ref{vanishing of cohomology for induced modules}), and hence the previous resolution computes the cohomology of $G$ with coefficients in $M$. Moreover, one verifies that $(M\otimes k[G]^{\otimes n+1})^G$ can be identified with $C^n(G,M)$. We refer the interested reader to \cite[Section I.4.15]{jantzen2003representations} for further details.
\end{construction}

\begin{Rec}
    We can use the Hochschild complex to define a $k$--algebra structure on the graded group $H^\ast(G,k)=\bigoplus_{i\geq0}H^i(G,k)$. First, note that we can identify $C^{n}(G,k)\otimes C^m(G,k)$ with $C^{n+m}(G,k)$. Via this identification we define the \textit{cup product} 
    \[
    C^\bullet(G,k)\otimes C^\bullet(G,k)\to C^\bullet(G,k)
    \]
    as the cochain map sending a homogeneous element $\lambda \otimes \lambda'$ in $C^n(G,k)\otimes C^m(G,k)$ to $\lambda \otimes \lambda'$ in $C^{n+m}(G,k)$. Moreover, using the differential in $C^{n}(G,k)\otimes C^m(G,k)$ we deduce that the cup product satisfies 
    \[
    \partial^{m+n}(\lambda \otimes \lambda')=\partial^n\lambda\otimes \lambda' +(-1)^m \lambda \otimes \partial^m\lambda'.
    \]
    From this is clear that an element $\lambda\otimes\lambda'$ is a cocycle if $\lambda$ and $\lambda'$ are cocycles. Now we are ready to define the cup product at the level of cohomology as: 
\begin{align*}
     H^\ast(G,k)\otimes H^\ast(G,k) & \to H^\ast(G,k)\\ 
    [\lambda]\otimes[\lambda'] & \mapsto [\lambda\otimes \lambda']
\end{align*}
One verifies that this is well defined, that is, that the cohomology class $[\lambda\otimes \lambda']$ just depends on the cohomology classes $[\lambda]$ and $[\lambda']$. Note that the cup product is graded commutative, meaning that for homogeneous elements $[\lambda]$ in $H^n(G,k)$ and $[\lambda']$ in $H^m(G,k)$ we have $[\lambda]\otimes[\lambda']=(-1)^{m+n}[\lambda'\otimes\lambda]$. Now, let $M$ be a $G$--module. We can identify $C^n(G,k)\otimes C^m(G,M)$ with $C^{m+n}(G,M)$ and follow the same construction of the cup product above to define a multiplication 
\[
H^\ast(G,k)\otimes H^\ast(G,M)\to H^\ast(G,M), \quad [\lambda][\lambda']\mapsto [\lambda \lambda'] 
\]
This action turns $H^\ast(G,M)$ into a left $H^\ast(G,k)$--module.
\end{Rec}

For convenience, let us record the above discussion in the following lemma.

\begin{Lemma} 
The cup product turns  $H^\ast(G,k)$ into a graded commutative $k$--algebra, and for any $G$--module $M$,  the graded cohomology groups $H^\ast(G,M)$ have a natural structure of a graded  $H^\ast(G,k)$--module.
\end{Lemma}

\begin{Rem}
Let us mention that there are different approaches to define the cup product on cohomology. For instance,  one can consider the extension groups $\Ext_G^i(k,k)$ as equivalence classes of exact sequences 
\[
0\to k\to M_{i-1} \to \ldots\to M_0 \to k\to 0
\]
in $\mathbf{Rep}(G)$  (e.g. see \cite[Section 2.6]{Ben}), and use  the Yoneda composition to define a product on the graded extension groups $\Ext^\ast_G(k,k)$. As expected, the product obtained by the Yoneda composition agrees with the cup product defined above. See \cite{suarez2004hilton}.     
\end{Rem}

\begin{Rec}
    When $M$ is a $G$--algebra, the graded cohomology $H^\ast(G,M)$ has more structure as we will explain. For this, let us briefly recall ourselves how to define the cup product using injective resolutions. Let $M$ and $N$ be $G$--modules. Let $J_M$, $J_N$ and $J_{M\otimes N}$ denote  injective resolutions of $M$, $N$ and $M\otimes N$ respectively. Recall that $J_M\otimes J_N$ denotes the tensor product over $k$ equipped with diagonal action of $G$. In particular, note that 
\[
0\to M\otimes N\to J_M^\bullet \otimes J_N^\bullet
\]
is an exact complex, where $J_M^\bullet$ denotes the complex $0\to J_M^{1}\to J^{2}_M\to \ldots$, and similarly for $J_N^\bullet$.  Indeed, it is enough to check this at the level of chain complexes of $k$--modules since forgetting the $G$--action detects exactness. For this, consider the K\"unneth spectral sequence applied to the complexes   $J_M^\bullet$ and $J_N^\bullet$:
\[
E_{p,q}^2 = \bigoplus_{k_1 + k_2 = q} \mathrm{Tor}_p^{k}( H_{k_1}(J_M^\bullet), H_{k_2}(J_N^\bullet)) \Rightarrow H_{p+q}(J_M^\bullet \otimes J_N^\bullet) 
\]
and hence the result follows. Alternatively one can pass to the derived category of $k$--modules and verify that 
\[
M\otimes^L N\to J_M^\bullet \otimes^L J_N^\bullet
\]
is a quasi-isomorphisms.

Now, since the complex $0\to M\otimes N\to J_M^\bullet \otimes J_N^\bullet$ is exact and $J_{M\otimes N}$ is an injective resolution of $M\otimes N$, we obtain a unique chain chain map up to homotopy  
\[
J_M\otimes J_N\to J_{M\otimes N}.
\]
This follows by the defining property of an injective object. Via this map, we obtain the cup product at the level of chain complexes by
  \begin{align*}
     (J_M)^G\otimes (J_N)^G & \to (J_M\otimes J_N)^G\\ 
    \lambda\otimes\lambda' & \mapsto \lambda\otimes \lambda'
\end{align*}
 just as before, we can check that this is a graded commutative product. 
\end{Rec}

\begin{Lemma}
   Let $G$ be a flat affine group scheme over a Noetherian ring $k$, and $A$ be a $G$--algebra. Then the graded cohomology module $H^\ast(G,A)\coloneqq\bigoplus_i H^i(G, A)$ forms a graded commutative $k$--algebra.
\end{Lemma}

\begin{proof}
 The multiplication is given by the composition 
   \[
   H^\ast(G,A)\otimes H^\ast(G,A)\to H^\ast(G,A\otimes A)\to H^\ast(G,A)
   \]
    where the first map is product the cup product as defined above, and the second map is induced by the multiplication $A\otimes A\to A$. It is a routine exercise to show that this multiplication turns $H^\ast(G,A)$ into a graded commutative $k$--algebra.  
\end{proof}

\begin{Rem}
    In fact, one can also give an interpretation of the cup product on $H^\bullet(G,A\otimes A)$ using the Yoneda product on extension groups, at least when $A$ is flat as $k$--module, see \cite[Lemma 2.5]{touze2010cohomology}.  
\end{Rem}

\begin{Rec}
    Since localization is an exact functor, we can see that the cohomology of the Hochschild complex is unaffected when localizing. More precisely, let $k$ be the ground ring and localize at any multiplicative subset $S \subset k$. Then $S^{-1}k$ is a flat $k$--module. Now, let $M$ be any $G$--module. Write $G_{S^{-1}k}$ for the base change of $G$ along the localization $k\to S^{-1}k$.  Then the Hochschild complex associated to the $G_{S^{-1}k}$--module $M[S^{-1}]$ is just the Hochschild complex for $M$ as $G$--module tensored with $S^{-1}k$, that is
    \[
    C^\bullet(G_{S^{-1}k}, S^{-1}M)=C^\bullet(G, M)\otimes_k S^{-1}k.
    \]
    Since $S^{-1}k$ is flat over $k$, we see that the cohomology is unaffected. In other words 
    \[
    H^i(G_{S^{-1}k}, S^{-1}M)=H^i(G, M)
    \]
    as $S^{-1}k$--modules. Actually, any flat base change doesn't affect cohomology, that is, for any morphism $k\to \ell$ of rings, which makes $\ell$ flat as $k$--module, then the above argument can be applied. We can see that for any $G_k$--module $M$ we have $H^i(G_k, M)=H^i(G_\ell, M)$ as $\ell$--modules.
\end{Rec}

From the above discussion we get the following result.

\begin{Cor}\label{Flat base change}
    Flat base change does not affect group cohomology.
\end{Cor}

We finish up this section with an important result regarding base change for cohomology. This requires some preparations.

\begin{construction}
    Let $l\to k$ be a map of rings, $G_l$ be a group scheme over $l$ and let $G_k$ be the base change of $G_l$ to $k$, that is, $k[G_k]=l[G_l]\otimes_l k$.  Any $G_k$--module $M$ with coaction $\Delta_M\colon M\to  M\otimes_k k[G_k]$ can be viewed as a $G_l$--module  by considering $M$ as $l$--module and $\Delta_M$ as an $l$--linear map by restriction along the map $l\to k$. This makes sense since 
    \[
    M\otimes_k k[G_k]=M \otimes_k k \otimes_l \otimes l[G_l]=M \otimes_l l[G_l].
    \]
    We can also consider any morphism of $G_k$--modules as a morphism of $G$--modules in the same fashion. This construction is functorial. 
\end{construction}

\begin{Def}
  The functor $\mathbf{Rep}(G_k)\to \mathbf{Rep}(G_l)$ from the above construction is called \textit{base change functor}, and it is denoted by $\mathrm{res}^k_l$. 
\end{Def}

\begin{Rem}
    The base change functor is not a standard tool in the literature on the representation theory of group schemes but rather a clever observation due to van der Kallen. Although it seems simple, it has important applications as we will see.  
\end{Rem}

\begin{Lemma}\label{base change}
    Let $l\to k$ be a morphism of rings, and $G_l$ be a flat group scheme over $l$. Let $G_k$ be the base change of $G_l$ to $k$. Then for every $G_k$--module $M$ there is an isomorphism of graded $l$--modules 
    \[
    H^\ast(G_k,M)\cong H^\ast(G_l,M).
    \]
\end{Lemma}

\begin{proof}
    Note that we can express the invariants as an equalizer of two maps. Indeed,  let $f\colon M\to M\otimes k[G]$ be the $k$--module map defined by $m\mapsto m\otimes 1$. Then we have that: 
    \[
    M^{G_k}=\mathrm{Ker}(M\xrightarrow[]{\Delta_M-f} M\otimes k[G] )
    \]
     and this agrees with $M^{G_l}$ since $M\otimes_k k[G_l]=M\otimes_l l[G_l]$ and hence our $k$--linear map $\Delta_M-f$ looks the same when considered as an $l$--linear map. 
     
    Let $J_M$ be an injective resolution of $M$ as $G_k$--module. We can assume that the terms of the resolution $J_M$ are typical injectives, that is, modules of the form $I\otimes_k k[G_k]$ for an injective $k$--module $I$.  Now, note that base change functor $\mathrm{res}^k_l$ is always exact since it is obtained from the restriction functor from $k$--modules to $l$--modules which is exact. Then the exactness of $\mathrm{res}^k_l$ together with our first observation give us  $(J_M)^{G_k}=(J_M)^{G_l}$ as complexes of $l$--modules. Moreover, the $G_k$--modules of the form $I\otimes_k k[G]$ have the form $I\otimes_l l[G]$ when viewed as $G_l$--modules. In particular, those modules are $G_l$--acyclic (see Corollary \ref{vanishing of cohomology for induced modules}). Thus $J_M$ is a resolution of $M$ by acyclic $G_l$--modules, hence $(J_M)^{G_l}$ computes  $H^\ast(G_l,M)$. But recall that the complex $(J_M)^{G_k}$ computes $H^\ast(G_k,M)$ as well, hence the result follows.  
\end{proof}

\pagebreak

\section{Costandard Modules and Good Filtrations}

Our aim in this section is to introduce some basic definitions and notation related to costandard modules and good filtrations, which will be relevant later on. Along the way, we also present some interesting results in the case where the base ring is a field.

 Let $k$ be a commutative ring, and let $G$ denote $\GL_n$. Recall that its coordinate algebra $k[G]$ is generated by  $t_{11}, t_{12}, \dots, t_{nn}$ and $det(t_{ij})^{-1}$ where the $t_{i,j}$ denote the dual basis of the space of $n\times n$--matrices with basis given by the matrices $e_{ij}$ having a one at entry $(i,j)$ and the rest zeros (see Example \ref{main examples}(iv)).

Consider the subgroup $T$ of $G$ of diagonal matrices, that is, $T(A)$ consists of diagonal matrices of $G(A)$, for $A$ a $k$--algebra. Similarly, consider the subgroup $B$ of lower triangular matrices. In particular, $T$ is a maximal torus of $G$ and can be identified with $(\mathbb{G}_m)^n$, and $B$ is a Borel subgroup of $G$. Moreover, $T\leq B\leq G$ and there is a retraction $B\to T$. 

\begin{Rec}
Recall that the character group $\chi(T)$ of $T$ is given by $\mathrm{Hom}(T,\mathbb{G}_m)\cong \mathbb{Z}^n$ (see Example \ref{examples 2}). In fact, the matrices $\epsilon_i\coloneq t_{ii}|_T$, for $1\leq i\leq n$, form a basis of $\chi(T)$. Let $\mathcal{Y}(T)$ denote the abelian group $\mathrm{Hom}(\mathbb{G}_m,T)$, and we consider it with a basis $\epsilon'_i$, for $1\leq i\leq n$, where $\epsilon'_i(a)= ae_{ii}+\sum_{j\not=i}e_{jj}$. There is a paring 
\[
\langle-,-\rangle\colon \chi(T)\times \mathcal{Y}(T)\to \mathbb{Z}
\]
such that $\lambda\circ \varphi$ is the map $a\to a^{\langle\alpha,\varphi\rangle}$, for all $a\in \mathbb{G}_a(A)$, $A$ a $k$--algebra. This paring is bilinear and induces an isomorphism $\mathcal{Y}(T)\cong\mathrm{Hom}_\mathbb{Z}(\chi(,\mathbb{Z})$. In elements of the basis the paring $\langle \epsilon_i,\epsilon'_j\rangle$ corresponds to the Kronecker symbol $\delta_{ij}$. The set 
\[
R\coloneq \{\epsilon_i-\epsilon_j\mid 1\leq i,j\leq n, \, i\not= j\}
\]
is called \textit{root system of} $G$. 
For a root $\epsilon_i-\epsilon_j\in R$ we let $(\epsilon_i-\epsilon_j)^\vee$ to denote $\epsilon'_i-\epsilon'_j$, these elements are usually refer as \textit{coroots}.   We let $R^+$ denote the set $\{\epsilon_i-\epsilon_j\mid 1\leq i<j\leq n\}$ and refer to its elements as \textit{positive roots}. 
\end{Rec}

\begin{Def}
 An element $\lambda \in \chi(T)$ is called \textit{dominant weight of $T$} if  $\langle \lambda, \varphi^\vee\rangle\geq 0$, for all positive roots $\varphi$. The set of all dominant weights of $T$ is denoted by $\chi_+(T)$. 
\end{Def}

\begin{Rem}
We warn the reader that the notion of dominant weight is attached to the choice of a set of positive roots, but in this document we will only work with $R^+$ as defined above unless we explicitly state otherwise. 
    
\end{Rem}

\begin{Rem}
   Recall that the character group $\chi(T)$ is isomorphic to $\mathbb{Z}^n$ and we have specified a basis $\epsilon_i$. Hence an element $\lambda$ in $\chi(T)$ is often denoted by $(\lambda_1,\ldots,\lambda_n)\in \mathbb{Z}^n$. In particular, $\lambda$ is a dominant weight if and only if  $\lambda_1\geq\ldots\geq \lambda_n$. 
\end{Rem}

\begin{Rec}
An element $\lambda$ in $\chi(T)$, determines a $T$--module structure on $k$, and we will write $k_\lambda$ to refer to $k$ with this $T$--action. We will also write $k_\lambda$ to denote the $B$--module obtained from the $T$--module $k_\lambda$ by restriction via the projection $ B\to T$. We let $\nabla_\lambda$ denote the induced module $\mathrm{ind}_B^G(k_\lambda)$.     
\end{Rec}

For the rest of this section, we let $k$ denote a field. The following result can be found in \cite[Section II.2.6, I.5.12]{jantzen2003representations}.

\begin{Lemma}\label{costandard have finite dimension over k}
    The induced module $\nabla_\lambda$ is non trivial if and only if $\lambda$ is a dominant weight. Moreover, $\nabla_\lambda$ is finitely generated over $k$. 
\end{Lemma}

\begin{Rec}
    Recall that the \textit{Weyl group $W$ of  the root system $R$} corresponds the group generated by the set of reflections $s_\alpha$ on the character group $\chi(T)$ corresponding to elements $\alpha$ in the root system, that is, $s_\alpha \lambda$ corresponds to $\lambda - \langle \lambda, \alpha\rangle\alpha$ and 
    \[
    W=\langle s_\alpha\mid \alpha \in R\rangle. 
    \]
    In fact, $W$ can be identify with the symmetric group on $n$ letters and it acts by permutation on $\{\epsilon_1\ldots,\epsilon_n\}$. We reserve $w_0$ to denote the element in $W$ corresponding to the permutation $i\mapsto n+1-i$.
\end{Rec}

\begin{Def}
 A \textit{costandard $G$--module} is a module of the form $\nabla_\lambda$ for $\lambda$ a dominant weight. A  module of the form $(\nabla_{-w_0\lambda})^*$, for $\lambda$ a dominant weight, is called \textit{standard $G$--module} and it is denoted by $\Delta_\lambda$. Note that the dual of a costandard module makes sense by Lemma \ref{costandard have finite dimension over k}.
\end{Def}

The following theorem will not be used in this notes, however, it largely reflects the relevance of the (co)standard modules in the representation theory of group schemes. For more details, see \cite[Section II.2]{jantzen2003representations}.

\begin{Th}
    The set $\{ L(\lambda) \mid \lambda \, \textrm{ is a dominant weight}\}$ is a complete list of simple $G$--modules up to isomorphism, where $L(\lambda)$ denotes the socle of $k_\lambda$. 
         For any dominant weight $\lambda\in \chi(T)$, there is an isomorphism $\nabla_\lambda\cong L(\lambda)$. In particular, the dual $(\nabla_\lambda)^\ast$ is isomorphic to $\nabla_{-w_0\lambda}$.
\end{Th}

Another remarkable feature of costandard modules is that they are acyclic, i.e., their cohomology vanishes. This is a consequence of Kempf's Vanishing Theorem and we will record it in the following theorem together with the vanishing of the extension groups of a costandard module. A proof can be found in \cite[Section II.4]{jantzen2003representations}.

\begin{Th}\label{Kempf's theorem}
    Let $\lambda$ be a dominant weight of $\chi(T)$. Then $R^i \mathrm{ind}_B^G k_\lambda$ vanish for $i>0$. Moreover, the group $\mathrm{Ext}_G^i(\nabla_\lambda,\Delta_\mu)$ is trivial, unless $i=0$ and $\lambda=\mu$. Since $\Delta_0=\nabla_0=k$, we obtain that $H^i(G,\nabla_\lambda)=0$ for  any non-trivial dominant weight $\lambda$. 
\end{Th}

We are ready to define a modules with a good filtration; the key feature to keep in mind is that those modules behave as modules in characteristic 0. 

\begin{Def}
    Let $M$ be a $G$--module. A chain of $G$--submodules of $M$
\[
\{0\} = M_0 \subset M_1 \subset M_2 \subset \dots
\]
is a good filtration of $M$ if every $ M_{i} / M_{i-1}$ is isomorphic to a direct sum of costandard modules. 
\end{Def}

\begin{Rem}
    The previous definition is due to van der Kallen \cite{vdK5-JB} and it is more general than the one given  in \cite{jantzen2003representations} where the quotient of consecutive terms in the filtration has to be a single costandard module. 
\end{Rem}

\begin{Th}\label{modules with a good filtration are acyclic}
   Let $M$ be a $G$--module. Then the following properties are equivalent.
   \begin{enumerate}
       \item $M $ has a good filtration. 
       \item $\mathrm{Ext}^i(\Delta_\lambda, M)=0$ for all dominant weights $\lambda$, and $i>0$.
       \item $\mathrm{Ext}^1(\Delta_\lambda, M)=0$ for all dominant weights $\lambda$.
   \end{enumerate}
\end{Th}

The following is a simple observation from the previous theorem in the particular case that $\lambda=0$ since $\Delta_0=k$.

\begin{Cor}
    If $M$ has a good filtration,  then  $H^i(G,M)=0$ for every $i>0$.
\end{Cor}

We finish this subsection with some properties of the class of modules with good filtrations. We refer to \cite[Chapter 3]{Don85} and \cite{Mathieu3}.

\begin{Lemma}\label{properties of good filtrations}
     Let $M$ and $N$ be a $G$--modules with good filtrations. Then the following properties hold. 
     \begin{enumerate}
         \item Any extension of $M$ by $N$ has a good filtration.
         \item $M\otimes N$ has a good filtration.
     \end{enumerate} 
\end{Lemma}

\begin{Rem}
We stress that the definitions in this subsection still make sense if we replace the field for a more general commutative Noetherian ring. Even some of the results hold more generally, for instance Kempf's Vanishing Theorem  is proved over the integers in \cite[Section B.4]{jantzen2003representations}. We will make this precise in Section \ref{section:GrosshansFiltrations}.
\end{Rem}
\pagebreak

\part*{Part II - Cohomological Finite Generation}\label{part II}
\addcontentsline{toc}{part}{Part II - Cohomological Finite Generation}

\section{Outline of the Second Part}

For the rest of this document, we will assume that all group schemes are flat and affine. Therefore, we will simply refer to them as group schemes without further qualification. The goal of this second part of the document is to give a panoramic overview of van der Kallen's theorem \cite{vdK23} on finite generation of cohomology for finite group schemes over an arbitrary Noetherian base: 

\begin{Th}[van der Kallen]
    Finite group schemes over a Noetherian ring $k$ have the \textit{cohomological finite generation property}.
\end{Th}

In order to make the above statement precise, let us introduce the relevant terminology. 

\begin{Def}
        Let $G$ be a group scheme over a ring $k$. We say that $G$ satisfies 
        \begin{enumerate}
            \item the \textit{finite generation (FG) property} if for every finitely generated $G$--algebra $A$, the algebra of invariants $A^G$ is also a finitely generated $k$--algebra, 

            \item the \textit{cohomological finite generation (CFG) property}  if for every finitely generated $G$--algebra $A$, the graded algebra of derived invariants $H^*(G, A)$ is also a finitely generated $k$--algebra.
        \end{enumerate}
\end{Def}

In particular, one can draw some interesting conclusions for group schemes with the (FG) or (CFG) property, for instance:

\begin{Prop}[c.f. \cite{TvdK10}]\label{Prop: CFG implies finite genarion of modules}
    Let $G$ be a group scheme over $k$. Then the following properties hold. 
    \begin{enumerate}
        \item Assume that $G$ satisfies (FG). Let $A$ be a finitely generated $G$--algebra, and $M$ be an $AG$--module finitely generated over $A$. Then $H^0(G,M)$ is finitely generated over $H^0(G,A)$. 
        \item Assume that $G$ satisfies (CFG). Let $A$ be a finitely generated $G$--algebra, and $M$ be an $AG$--module finitely generated over $A$. Then $H^\ast(G,M)$ is finitely generated over $H^\ast(G,A)$. 
    \end{enumerate}
\end{Prop}

\begin{proof}
    Let $B$ be  a $k$--algebra and $N$ be a $B$--module. We define a $k$--algebra $B\rtimes N$ as follows. It is  $B\oplus N$ as $k$--module, and the multiplication is given by  $(b.n)\cdot (b',n')=(bb',bn'+b'n)$. Note that $B\rtimes N$ is finite generated as $k$--algebra if and only if $B$ is finitely generated as $k$--algebra and $N$ is finitely generated as $B$--module. Now use 
    \[
    H^\ast(G,A\rtimes M)=H^\ast(G,A)\rtimes H^\ast(G,M).
    \]
    Hence the result follows.  \qedhere \end{proof}

Here we give an outline of the proof of van der Kallen's theorem, the details of which are the core of what remains in this document. The essence of van der Kallen's proof consists of a few steps, each of which comprises a section in this chapter:

Firstly, we show that $\GL_n$ satisfies (CFG) over a Noetherian base $k$ when $k$ contains a field (see Theorem \ref{CFG for GLn}).

From this result, we can obtain a form of \textit{provisional (CFG)} over an arbitrary Noetherian base $k$ (see Theorem \ref{torsion}). Reduction from $(CFG)$ in the case where the base ring contains a field to the weaker provisional (CFG) for arbitrary $k$ is via a reduction to the Noetherian ring $k \otimes \mathbb{Z}_p/p(k \otimes \mathbb{Z}_p)$, which contains the field $\mathbb{F}_p$. For a finitely generated $\GL_n$--algebra $A$, this provisional (CFG) allows you to deduce the finite generation $H^*(\GL_n, A)$ over any Noetherian base $k$ via the existence of a uniform bound on its torsion. 

Now take our finite group scheme $G$ over an arbitrary Noetherian base $k$, we may embed $G$ into $\GL_n$ for some $n>0$ (see Proposition \ref{embedding of G}). Now reduction from a finite group scheme $G$ to $\GL_n$ is known as the \textit{Reduction Lemma} (see Lemma \ref{vdK's Reduction Lemma}) from which we establish an isomorphism $H^*(G, A)\simeq H^*(\GL_n,\mathrm{ind}_G^{\GL_n} A )$.

For the last step, we establish the existence of a bound on the torsion of $H^*(G, A)$ (see Theorem \ref{bounded torsion}), which coupled with provisional (CFG) and the above isomorphism immediately gives us van der Kallen's theorem!

\pagebreak

\section{The Reduction Lemma}\label{section: reduction lemma}

In this section, we show how to reduce the (CFG) problem for finite group schemes to that of $\GL_n$. The following result will play a fundamental role in the rest of this document (c.f. \cite[Exp. VIB, Proposition 3.15]{SGA3I}). Recall that all the group schemes in the rest of this document are assumed to be flat and affine.

\begin{Prop}\label{embedding of G}
 Let $G$ be a finite flat affine group scheme over a Noetherian ring $k$. Then there is an closed embedding $G\hookrightarrow \GL_n$ for some $n>0$.
\end{Prop}

\begin{proof}
   Consider the right regular representation $\rho_r$ of $k[G]$. Note that $\rho_r$ is faithful. Indeed, we have that $\rho_r(g)(f)(1)=f(g)$ for $f\in k[G]$, $g\in G(k)$ and $1$ the unit of $G(k)$. Now, since  $k[G]$ is projective, there is a $k$--module $Q$ such that $k[G]\oplus Q\cong k^n$ for some $n>0$. Letting $G$ act trivially on $Q$, we obtain that $k[G]\oplus Q$ is a faithful $G$--module. Hence we obtain a morphism \[\phi\colon G\to \GL_n.\] We claim that this morphism is a closed embedding. Indeed, recall that a morphism of affine schemes $\mathrm{Spec}(A)
   \to \mathrm{Spec}(B) $ is a closed embedding if and only if the corresponding morphism of rings $B\to A$ is a surjection. Moreover, surjections are detected locally on $\mathrm{Spec}(k)$. Then we need to show that the corresponding map $\phi'\colon k[\GL_n]\to k[G]$ is a surjection after localization at each prime $\mathfrak{p}$ of $k$. But $k_\mathfrak{p}[G]$ is a free $k_\mathfrak{p}$--module since $k_\mathfrak{p}$ is local. Hence the same proof as in Theorem \ref{Thm: Group schemes over fields are embedding in GL_n}   
   shows that $\phi'\otimes k_\mathfrak{p}$ is surjective. This completes the proof. 
\end{proof}

\begin{Lemma}[Reduction Lemma]\label{vdK's Reduction Lemma} Let $k$ be a ring. Let $\mathbb{G}$ be a group scheme over $k$ of finite type, and $G$ be a flat subgroup scheme of $\mathbb{G}$. Let $A$ be a finitely generated $G$--algebra. Suppose that:
\begin{itemize}
    \item[$(i)$] $G$ satisfies (FG).
    \item[$(ii)$] $\mathbb{G}/G$ is an affine scheme. 
\end{itemize}
   \noindent Then the following properties hold. 
   \begin{itemize}
       \item[(a)] $\mathrm{ind}_G^\mathbb{G} A:= (k[\mathbb{G}]\otimes A)^G$ is a finitely generated $k$--algebra.  
       \item[(b)] $H^\ast(G,A)\cong H^\ast(\mathbb{G},\mathrm{ind}_G^\mathbb{G} A)$.
   \end{itemize}
\end{Lemma}

\begin{proof}
    Since $k[\mathbb{G}]$ and $A$ are both finitely generated over $k$, so is $k[\mathbb{G}]\otimes_k A$. The assumption that $G$ satisfies (FG) gives us that $(k[\mathbb{G}]\otimes_k A)^G$ is finitely generated over $k$, thus proving part (a).
    
    In the discussion of the induction functor in Remark \ref{flat subgroup scheme}, we saw that if $\mathbb{G}/G$ is an affine scheme, then the derived induction functors $R^i\mathrm{ind}^\mathbb{G}_G(M)$ are given by the cohomology of a quasi-coherent sheaf on $\Spec k[\mathbb{G}/G]^*$. A classical result of Serre \cite[III 3.7]{hartshorne2013algebraic} is that quasi-coherent sheaves on affine schemes have vanishing higher cohomology groups. Therefore we see that the higher derived induction functors vanish, that is,  $R^{>0}\mathrm{ind}^\mathbb{G}_G(M)=0$. In other words, the induction functor is exact. Part (b) is called generalized Frobenius reciprocity, and now just follows from Proposition \ref{frobenius reciprocity} (1).
\end{proof}

Interestingly, over a field, Touz\'e and van der Kallen proved in \cite{TvdK10} that (FG) and (CFG) are equivalent properties. The hope is that, over an arbitrary Noetherian base, the properties (FG) and (CFG) are not too different. By looking into the future, and using results which appear in the coming two sections, we can see the utility of the Reduction Lemma in this direction via the following particular instance of van der Kallen's theorem:

We consider the case where $k$ is a Noetherian ring containing a field $\mathbb{F}$. In this situation we have shown in Theorem \ref{CFG for GLn} that $\GL_n$ has (CFG). Now, take $G$ to be your favourite finite group scheme over $k$, and embed it into $\mathbb{G} = \GL_n$, as above. We see that $G$ satisfies (FG) by Theorem \ref{FG for Chevalley}. Point $(ii)$ holds by the discussion in Remark \ref{flat subgroup scheme}, which states that when $G$ is finite, $\GL_n/G$ is affine. Thus, the conclusions $(a)$ and $(b)$ of the Reduction Lemma together show that since $\GL_n$ satisfies (CFG), so does $G$. In particular, (CFG) for $\GL_n$ implies (CFG) for the finite group scheme $G$ that you started with.

\section{Finite Generation of Invariants}\label{section: FG}
A crucial point in establishing the cohomological finite generation theorem over a Noetherian base is first establishing the \textit{finite generation} (FG) property for the general linear group over a Noetherian ring $k$. Indeed, in \cite{FvdK10} it is proved that \textit{Chevalley group schemes} over a Noetherian ring satisfy (FG) (of which $\text{GL}_n$ is an example). Moreover, in \cite{van2021reductivity} it is proved that finite group schemes over a Noetherian ring also satisfy (FG). In this section, we are going to outline the key points in the proofs of these two important theorems.

We begin with the definition of the (FG) property:

\begin{Def}
    Let $G$ be an affine group scheme over a ring $k$. We say that $G$ satisfies the \textit{finite generation property} (FG) if for every finitely generated $G$--algebra $A$, the subring of invariants $A^G$ is also a finitely generated $k$--algebra.
\end{Def}

This is not a property we can expect to hold in general; indeed, the next example shows that finite generation does not even hold for the additive group scheme $\mathbb{G}_a$. 

\begin{Ex}\label{Ex: Nagata} We highlight two examples where the algebra of invariants is not finitely generated. A further discussion can be found in \cite[Section 2.1]{DK15}.

\begin{enumerate}
    \item Let $a_{i,j}$ be algebraically independent elements over  $\mathbb{Q}$ for $i=1,2,3$ and $j=1,\ldots, 16$. Let $G$ be the subgroup of $\GL_{32}$ of block diagonal matrices $[A_r]_{1\leq r\leq 16}$ where 
    \[A_r= \begin{pmatrix}
 c_j & b_j c_j \\
 0 & c_j
\end{pmatrix}\]  
with $c_j $ and $b_j$ in $\mathbb{C}$ satisfying that $c_1\cdot\ldots\cdot c_{16}=1$ and $\sum_{j=1}^{16}a_{i,j}b_j=0$ for $i=1,2,3$. Consider $x_1,\ldots,x_{16},y_1,\ldots,y_{16}$ complex numbers algebraically independent over $\mathbb{C}$. Then 
\[k[x_1,\ldots,x_{16},y_1,\ldots,y_{16}]^G\]
is not finitely generated over $\mathbb{C}$. This is Nagata's famous counterexample to \textit{Hilbert's fourteenth problem}. In fact, one can replace $\mathbb{C}$ with any algebraically closed field $k$ of characteristic 0 and $\mathbb{Q}$ with the prime field of $k$.  See \cite{Nag59}. 

\item Let $k$ be an algebraically closed field of characteristic 0. Consider the action of the additive group scheme $\mathbb{G}_a$ on $k[x_1,x_2,y_1,y_2,y_3]$ given by 
\begin{align*}
    \sigma\cdot(x_1,x_2,y_1,y_2,y_3) =  \\
 \left (x_1, x_2,y_1 + \sigma x_1^2 , y_2 +\sigma(x_1 y_1 + x_2) + \frac{1}{2}\sigma^2 x_1^3, y_3 + \sigma y_2 + \frac{1}{2}\sigma^2 ( x_1 x + x_2)+ \frac{1}{6}\sigma^3 x_1^3\right).
\end{align*}
Then $k[x_1,x_2,y_1,y_2,y_3]^{\mathbb{G}_a}$ is not finitely generated over $k$. This example is due to Daigle and Freudenburg. See \cite{DF99}.
\end{enumerate}
\end{Ex}

\begin{Rem}
    We warn the reader that the previous examples are far from trivial. In fact, there is an active research area focused on finding other counterexamples to Hilbert's fourteenth problem and variants of it. 
\end{Rem}

A key step in establishing (FG) for certain classes of group schemes is showing that they satisfy the following reductivity property:

\begin{Def}
    Let $G$ be an affine group scheme over $k$. Let $L$ be any cyclic $k$--module, and $L$ be the trivial representation of $G$ on $L$. Let $M$ be any representation of $G$, and let $\varphi$ be a surjective $G$--module homomorphism from $M$ to $L$. If there exists a positive integer $d$ such that the $d$th symmetric power induces a surjection 
    \[
    (S^dM)^G\to S^dL,
    \]
    we will say that $G$ is \textit{power reductive.}
\end{Def}

We  give a few examples of power reductivity:
\begin{Ex}
    \begin{enumerate}[(i)]
        \item Let $k$ be a field of characteristic $2$, let $M=k^2$ with basis $x, y$, and let $L=k$ with basis $z$. Let $\mathfrak{S}_2$ denote the symmetric group on two elements. This group can act on $M$ in the obvious way via $x\mapsto y$ and $y\mapsto x$. The invariants of $M$ is the one dimensional vector space spanned by the invariant $x+y$, i.e. $M^{\mathfrak{S}_2}=k\langle x+y\rangle$
        
        We define a $k$--linear map $M\to L$ via $x\mapsto z$ and $y\mapsto z$, this is clearly a map of representations. Under this map we have that $x+y\mapsto 2z=0$, since we are in characteristic $2$, so the induced map $M^{\mathfrak{S}_2} \to L$ is just $0$. Looking in degree $2$ we can see that $x^2+xy+y^2 \mapsto 3z^2=z^2$. Since $x^2+xy+y^2 \in (S^2M)^{S_2}$, and $S^2L=k[z^2]$, we can see that the map \[(S^2M)^{\mathfrak{S}_2} \to S^2L\] is surjective. This exhibits the power reductivity of $\mathfrak{S}_2$.
    \item Let $k=\mathbb{Z}$, and consider $\mathrm{SL}_2$ acting by conjugation on the group $\mathcal{M}$ of $2\times 2$ matrices with entries in $\mathbb{Z}$, this has a basis consisting of matrices with a $1$ in one entry and $0$ everywhere else. Consider also the group $\mathcal{L} \subset \mathcal{M}$ of matrices consisting of multiples of the identity matrix. Let us consider the linear functional on $\mathcal{M}$ defined by 
    \[
    \alpha\colon \mathcal{M}\to km, \quad \begin{pmatrix}
         1 &0 \\
         0 &0 
    \end{pmatrix}\mapsto 1
    \]
    and similarly we will denote by $\beta, \gamma, \delta$ the linear functionals dual to the other basis elements. Clearly $\{\alpha, \beta, \gamma, \delta\}$ is a dual basis for $M=\mathcal{M}^*$. The action of $\mathrm{SL}_2$ on $\mathcal{M}$ also gives us an action of $\mathrm{SL}_2$ on $M$. Consider  the linear functional on $\mathcal{L}$ given by 
    \[
    \lambda\colon \mathcal{L}\to k, \quad \begin{pmatrix}
         1 &0 \\
         0 &1 
    \end{pmatrix} \mapsto 1.
    \]
     Clearly $\{\lambda\}$ is a dual basis for $L=\mathcal{L}^*$. There is the obvious inclusion map $\mathcal{L}\to \mathcal{M}$, and taking duals we get a surjective map $M\to L$, we can see that this dual map is defined on the dual basis by $\alpha, \delta \mapsto \lambda$ and $\beta, \gamma \mapsto 0$. This extends to a surjective map on their symmetric algebras:
\[S^*M=\mathbb{Z}[\alpha, \beta, \gamma, \delta] \to \mathbb{Z}[\lambda]=S^*L\]
The action of $\mathrm{SL}_2$ on $M$ is inherited by $\mathbb{Z}[\alpha, \beta, \gamma, \delta]$, and we can see that the determinant polynomial $D=\alpha \delta - \beta \gamma$ and the trace polynomial $t=\alpha+\delta$ generate all the invariants of this $\mathrm{SL}_2$ action. So taking invariants we get a map: 
\[\mathbb{Z}[\alpha, \beta, \gamma, \delta]^{\mathrm{SL}_2}=\mathbb{Z}[t, D] \to \mathbb{Z}[\lambda].\]
We can see that this map is not surjective when restricting degree $1$, since $t\mapsto 2\lambda$, but indeed it is surjective in degree two since $D\mapsto \lambda^2$. This exhibits the power reductivity of $\mathrm{SL}_2$.
    \end{enumerate}
\end{Ex}

\begin{Rem}
    Let $G$ denote $\mathfrak{S}_2$ or $\mathrm{SL}_2$, and let $M$ denote the $G$--module defined in the previous example, respectively. Note that we considered above the \textit{naive fixed points} in the sense that we only looked at the $G(k)$-invariants of the corresponding representation $G(k) \times M \to M$. However, in both cases, $G(k)$ is dense in $G$, and hence there is no difference between $M^G$ and $M^{G(k)}$; see Proposition \ref{Prop:naive invariants vs functorial invariants}.
\end{Rem}

\begin{Def}
    Let $\varphi:B\to A$ be a morphism of $k$--algebras. We will say that $\varphi$ is \textit{power surjective} if every $a\in A$ has $a^k \in \im \varphi$ for some positive integer $k$.
\end{Def}

\begin{Prop}\label{integral2}
    let $k$ be a ring and let $G$ be a  power reductive flat affine group scheme over $k$. Let $f\colon  A\to B$ be a power surjective map of $G$--algebras. Then the induced map 
    \[
    f^G\colon  A^G\to B^G
    \]
    is power surjective.
\end{Prop}

\begin{proof}
 Pick some $b\in B^G$, since our map is power surjective $b^k$ has a preimage in $A$ for some $k>0$. Now take $L$ to be the cyclic module generated by $b^k$, i.e. $L=\langle b^k\rangle$. Since $b^k$ is an invariant, the induced subcomodule structure on $L$ is the trivial one. Let $M$ be the kernel of the $G$--equivariant map $A\to B\to B/L$ which is a $G$--submodule of $A$ by the flatness assumpion on $G$. Note that $f(M)$ is a submodule of $B$ contained in $L$, so it must be equal to $L$. 
The assumption of power reductivity gives us a $d>0$ such that there exists a commutative diagram:
 \begin{center}
\begin{tikzcd}
    &(S^dM)^G \arrow[r] \arrow[d, "(S^df)^G", twoheadrightarrow]&(S^dA)^G \arrow[r] \arrow[d, "(S^df)^G"]&A^G \arrow[d,"f^G"] \\
    &S^dL \arrow[r]& (S^dB)^G \arrow[r] &B^G 
\end{tikzcd}
\end{center}
where the rightmost horizontal maps are multiplication in $A$ and $B$. Since the leftmost vertical map is surjective, the element $b^k\otimes \dots \otimes b^k$ ($d$ times) in $S^dL$ has a preimage in $(S^dM)^G$, say $m=m_1\otimes \dots \otimes m_d$. Now mapping this element $m$ around the outside commutative square, the top path gives us $f(\prod_{i=1}^d m_i)$, and the bottom path gives us $b^{kd}$. In other words, $b^{kd}$ has a preimage in $A^G$, so that $A^G\to B^G$ is power surjective.
\end{proof}

\begin{Rem}
  In the context of  Proposition \ref{integral2}, a relevant case is as follows. Let $I$ be a $G$--stable ideal of $A$. Then the projection map $A\to A/I$ induces a power-surjective morphism on invariants 
  \[A^G\to (A/I)^G.\]
\end{Rem}

\begin{Th}[Hilbert's fourteenth problem]\label{Hilberts 14th}
    Let $G$ be a flat affine group scheme over a Noetherian ring $k$. Let $A$ be a finitely generated $G$--algebra. If $G$ is power reductive, then the algebra of invariants $A^G$ is a finitely generated $k$--algebra. In other words, flat power reductive group schemes satisfy (FG).
\end{Th}

\begin{Rem}
   The original proof given in \cite{FvdK10} of the theorem above claims that the same reasoning from pages 23-26 of \cite{springer2006invariant} is good enough to obtain the result and one only needs to replace the conclusion of \cite[Lemma 2.4.7]{springer2006invariant} for geometrically reductive algebraic groups over a field with the conclusion of Proposition \ref{integral2}. This is indeed the case, however, we decided to include some details for completeness. We need some preparations. 
\end{Rem}

\begin{Def}
  Let $B$ be a commutative ring, and let $A\subseteq B$ be a subring. An element $b\in B$ is \textit{integral over A} if there exists $n\geq 1$, and  $a_0,a_1\ldots, a_{n+1}\in A$ such that 
  \[
   0 = b^n +a_{n-1}b^{n-1} + \ldots +a_0. 
  \]
  We say that $B$ is \textit{integral over $A$} if all elements of $B$ are integral over $A$. We also say that a map of rings $\varphi\colon A\to B$ is integral if $B$ is integral over the image of $\varphi$. 
\end{Def}

\begin{Lemma}\label{integral1}
    Power surjective maps of $k$--algebras are integral extensions.
\end{Lemma}

\begin{proof}
    let $\varphi:A
    \to B$ be a power surjective map of $k$--algebras. We want to show that every $b\in B$ is the root of a monic polynomial with coefficients in the image of $A$. Let $b\in B$ be arbitrary, then since $\varphi$ is power surjective we have $b^k=\varphi(a)$ for some $a\in A$. In particular $b$ is a root of the monic polynomial $t^k-\varphi(a)$.
\end{proof}

Let us recall the following well-known result. 

\begin{Lemma}[Artin-Tate]
    Let $k$ be a commutative Noetherian ring and $B\subseteq C$ be algebras over $k$. If $C$ is finitely generated as $k$--algebra and $C$ is finitely generated as $B$--module, then $B$ is finitely generated as $k$--algebra. 
\end{Lemma}

\begin{proof}
    This is left as an exercise. 
\end{proof}

\begin{Lemma}\label{Lemma Artin-Tate}
   Let $k$ be a commutative Noetherian ring and $B\subseteq C$ be algebras over $k$. If $C$ is finitely generated as $k$--algebra and $C$ is integral over $B$, then  $B$ is also a finitely generated $k$--algebra. Moreover, $C$ is finitely generated as $B$--module 
\end{Lemma}

\begin{proof}
    Write $C=k[c_1,\ldots,b_m]$ for some elements $c_i\in C$. Since each  $c_i$ is integral over $B$, there are $n_i\geq 1$ and elements $b_{i,1},\ldots , b_{i,n_i}$ such that 
    \[
    0=c_i^{n_i} + b_{i,1}c_i^{n_i-1}+\ldots b_{i,n_i}.
    \]
    Let $X$ be the set containing all the  $b_{i,j}$ appearing in these equations. Consider the $k$--algebra $A=k[X]$. Note that $A$ is a Noetherian by Hilbert's basis theorem. Moreover, the ring $C$ is integral over $A$, and hence it  is finitely generated as $A$--module. Since $B$ is an $A$--submodule of $C$, we deduce that $B$ is a finitely generated $A$--module. It follows that $B$ is a finitely generated $k$--algebra.  For the second claim, a set of generators of $B$ as $A$--module together with $X$ is a set of generators for $C$ as $B$--module. 
\end{proof}

\begin{Rec}\label{Noetherianity for graded ring}
    Recall that for a $\mathbb Z$--graded commutative ring $A_\ast$ with underlying ring $A$, the following conditions are equivalent:
    \begin{enumerate}
        \item  $A$ is Noetherian.
        \item $A_0$ is Noetherian and  $A$ is a finitely generated  $A_0$--algebra.
        \item Any homogeneous ideal of $A_\ast$ is finitely generated. 
    \end{enumerate}
\end{Rec}

\begin{Rec}\label{Rec: A+ is finitely generated}
   Let $A$ be a non-negatively graded $k$--algebra. If the homogeneous ideal $A^+=\bigoplus_{n\geq 1}A_{n}$ is finitely generated, then $A$ is a finitely generated $A_0$--algebra. 
   \end{Rec}

\begin{Rec}\label{Rec: SV}
    Let $V$ be a finitely generated free $k$--module. Write $S(V)$ to denote \textit{the $k$--algebra of polynomial functions on $V$,} that is, the subalgebra of the algebra of functions on $V$ (with the pointwise structure) generated by a dual basis for $V$.  One can verify that $S(V)$ is naturally isomorphic to both $ k[t_1,\ldots, t_{\mathrm{rk}(V)}]$ and to the symmetric algebra $\mathrm{Sym}(V^\ast)$. We will consider $S(V)$ with the obvious grading coming from $k[t_1,\ldots, t_{\mathrm{rk}(V)}]$. 

  On the other hand, consider a flat group scheme $G$ and let  $V$ be a $G$--module that it finitely generated and free as $k$--module. Then  $S(V)$ has structure of $G$--algebra. Indeed, for a polynomial $f\in V(S)$, we let
  \[
  g\cdot f(v)=f(g^{-1}v), \textrm{ for } g\in G \textrm{ and } v\in V. 
  \]
   Note that this $G$--action respects the grading, and hence $S(V)^G$ is a graded $k$--algebra. 
\end{Rec}

\begin{Rec}\label{structure of G-algebras}
    Let $G$ be a flat group scheme over $k$. Then any $G$--algebra $A$ that is finitely generated as $k$--algebra is of the form 
    \[A= S(V)/ I\]
    for some $G$--invariant ideal $I$ of $S(V)$. Here we are considering the $G$--action on $S(V)$ as in Recollection \ref{Rec: SV}. Indeed, consider the $G$--submodule $V$ of $A$ spanned by a set of generators of $A$ as $k$--algebra. By the universal property of the symmetric algebra, we must have a $k$--algebra map 
    \[f\colon \mathrm{Sym}(V^\ast) \to A\]
    that restricts to the inclusion $V\to A$, and hence $f$ must be surjective.  Moreover, we obtain a surjective map $f'\colon S(V)\to A$ by identifying  $\mathrm{Sym}(V^\ast)$  with $S(V)$. The map $f'$  is actually $G$--equivariant and hence  the kernel of this map is a stable $G$-ideal of $S(V)$. It follows that $S(V)/\mathrm{ker}(f')\cong A$.   
\end{Rec}

\begin{Prop}\label{Prop: I is homogeneous}
    Let $G$ be a power-reductive group scheme over $k$. Let $V$ be  a  $G$--module that is finitely generated free as $k$--module, and $I$ be a homogeneous $G$-stable ideal of $S(V)$. Then $(S(V)/I)^G$ is a finitely generated $k$--algebra. 
\end{Prop}

\begin{proof}
    We will proceed by contradiction. Assume that $(S(V)/I)^G$ is not finitely generated as $k$--algebra. Since $k$ is Noetherian, so is $S(V)$. Now, consider the collection of $G$--invariant homogeneous ideals $I'$ of $S(V)$ such that $(S(V)/I')^G$ is not finitely generated. By our assumption, this collection is not empty. Hence there exists a maximal $G$--stable homogeneous ideal $I$ of $S(V)$ with this property. By maximality of $I$, we obtain that any homogeneous ideal of $A=S(V)/I$ must satisfy that $(A/J)^G$ is a finitely generated $k$--algebra. Moreover, we claim that for any $G$--stable homogeneous ideal $J$ of $A$, the algebra $A^G/(A^G\cap J)$ is finitely generated.  Indeed, consider the  projection map $\pi\colon A\to A/J$ which induces  map 
    \[ \pi\colon A^G\to (A/J)^G. \]
    By Proposition \ref{integral2}, we get that $(A/J)^G$ is integral over $\pi (A^G)$. By Lemma \ref{Lemma Artin-Tate} we have that $\pi(A^G)$ must be finitely generated as $k$--algebra. But note that $\pi(A^G)\cong A^G/(A^G\cap J)$, this completes the claim.

    Now, let $a\in A^G$ be homogeneous of positive degree. We claim that $a$ must be a zero divisor. Assume otherwise. Let $ax\in A^G$. Then $a(gx-x)=0$ for all $g\in G$. This implies that $x\in A^G$. It follows that $a A\cap A^G =a A^G$. By the previous paragraph,  we conclude that $A^G/aA^G$ is finitely generated as $k$--algebra. It follows that $(A^G)^+/aA^G$ is finitely generated, and so is $(A^G)^+$. By Recollection \ref{Rec: A+ is finitely generated}, we deduce that $A^G$ must be a finitely generate $k$--algebra which is a contradiction.  Then $a$ must be a zero divisor. 

    Fix $a\in A^G$ homogeneous of positive degree. Note that 
    \[
    \mathrm{Ann_A(a)\coloneqq \{x\in A\mid xa=0\}}
    \]
    is a non-zero, $G$--invariant homogeneous ideal of $A$. It follows that $(A/\mathrm{Ann}_A(a))^G$ is finitely generated $k$--algebra. Using again Proposition \ref{integral2} applied to the projection map $A\to A/\mathrm{Ann}_A(a)$, we deduce that $(A/\mathrm{Ann}_A(a))^G$ is integral over $A^G/\mathrm{Ann}_A(a)\cap A^G$. Another layer of Lemma \ref{Lemma Artin-Tate} gives us that  $(A/\mathrm{Ann}_A(a))^G$ is a finitely generated $(A^G/\mathrm{Ann}_A(a)\cap A^G)$--module.  In fact, the map 
    \[
    (aA)^G \to  (A/\mathrm{Ann}_A(a))^G, \, \, \,  ax\mapsto [ax] 
    \]
    is an isomorphism of  $(A^G/\mathrm{Ann}_A(a)\cap A^G)$--modules, which implies that $(aA)^G$ is a finitely generated $A^G$--module. Since $A^G/(aA)^G\cong A^G/aA\cap A^G$ is finitely generated as $k$--algebra,  we conclude that the ideal $(A^G)^+$ must be finitely generated. By Recollection \ref{Rec: A+ is finitely generated} we obtain that $A^G$ must be Noetherian. We obtain the desired contradiction. 
\end{proof}

\begin{proof}[Proof of Theorem \ref{Hilberts 14th}]
  We proceed as in \cite[Lemma 11]{van2021reductivity}. Let $M$ be a finitely generated $G$--submodule of $G$ containing a finite set of generators of $A$ as $k$--algebra and also the unit $1$. Let $B$ the graded $k$--subalgebra of $A[x]$ generated by $M[x]$. Note that $B$ is also a $G$--algebra and evaluation at $1$ gives us a $G$--algebra map  $\mathrm{ev}_1\colon B\to A$. Let $B_d$ denote the $d$th homogeneous part of $B$.  By the definition  of $B$, we get that  $\mathrm{ev}_1|_{B_d}$ is an inclusion, for any $d\geq0$. In particular, this tell us that any $G$--invariant of $A$ lies in the image of some $G$-invariant of $b_d$ under $\mathrm{ev}_1$, for some $d\geq 0$. In other words, 
  \[
  \mathrm{ev}_1^G\colon B^G\to A^G
  \]
  is a surjective $k$--algebra map. Now, by Proposition \ref{Prop: I is homogeneous}, we deduce that $B^G$ is finitely generated over $k$. Since $\mathrm{ev}_1^G$ is surjective, we deduce that $A^G$ is also finitely generated over $k$. 
\end{proof}

\begin{Rem}
  One might attempt to adapt Nagata’s method for proving finite generation of invariants over fields to the case of $G$--algebras of the form $S(V)/I$, where $I$ is not necessarily homogeneous, and then conclude using Recollection \ref{structure of G-algebras}. For the  readers interested in this approach, we refer to the argument  presented in \cite[Section 3.4]{Dolgachev}, restricted to the case of fields. See also \cite[Remark 9]{van2021reductivity} for a further discussion. 
\end{Rem}

In principal, to prove that a certain class of group schemes satisfy (FG), one could alternatively prove that the desired class is power reductive. As it turns out, power reductivity is a local property on $\Spec(k)$, as the next proposition shows. First, let us recall  that for a group scheme $G$ over a commutative ring $k$, and a map of rings $k\to l$, we write $G_l$ to denote the group scheme obtained by base change along $k\to l$.

\begin{Prop}\label{Power Reductivity is Local}
  Let $G$ be a group scheme over a commutative Noetherian ring $k$.  If $G_{k_{\mathfrak{m}}}$ is power reductive for every maximal ideal $\mathfrak{m}$ of $k$, then $G$ is power reductive. 
\end{Prop}

\begin{proof}
    We define a collection of objects:
    \[\mathcal{J}_d(k)=\{x\in k \mid \exists m>0 \text{ such that } x^mS^dL\subseteq S^d\varphi((S^dM)^G)\}\]
     where $L,\,  M,$ and $\varphi$ are as in the definition of power reductivity. It is easy to see that the condition $\mathcal{J}_d(k)=k$ is equivalent to the power reductivity of $G$. Moreover, this construction commutes with localization, since taking invariants commutes with localization. Therefore, we can see that $\mathcal{J}_d(k_\mathfrak{m})$ is equal to $\mathcal{J}_d(k)_{\mathfrak{m}}$ as ideals of $k_\mathfrak{m}.$ Now if we assume that $G_{k_\mathfrak{m}}$ is power reductive for all maximal ideals $\mathfrak{m}$, we get that 
     \[
     k_\mathfrak{m}=\mathcal{J}_d(k_\mathfrak{m})=\mathcal{J}_d(k)_{\mathfrak{m}}
     \]
     for some $d$ and all $\mathfrak{m}$. Since the equality of ideals is a local property, we can conclude that this property holds globally, in other words $\mathcal{J}_d(k)=k$.
\end{proof}

There is an alternative characterisation of power reductivity which is useful, called Integrality or (Int):

\begin{Def}
    We will say that an affine group scheme $G$ satisfies property (Int) if for every surjective map $A\to B$ of $G$--algebras we have that $B^G$ is integral over the image of $A^G$.
\end{Def}

\begin{Prop}\label{Int iff PR}
    Let $G$ be a flat affine group scheme over $k$. Then $G$ satisfies property (Int) if and only if $G$ is power reductive.
\end{Prop}

\begin{proof}
  It follows immediately from \ref{integral1} and \ref{integral2} that power reductivity implies (Int). We will prove the converse.
  
  Let $\varphi\colon M\to L$ be a surjective map of comodules, where $M$ is arbitrary, and $L$ is a cyclic $k$--module endowed with the trivial action. After taking symmetric powers we will still have a surjective map of comodules 
  \[
  S^*\varphi\colon  S^*M\to S^*L
  \]
   where the actions on the symmetric algebras are the diagonal actions, this is even a surjective map of $G$--algebras. Now pick some generator $b\in L$, viewed as a homogeneous element of degree $1$ in $S^*L$. Using our assumption of the property (Int), that is, that $S^*L$ is integral over $(S^*M)^G$ via the map $(S^*\varphi)^G$, there exists a monic polynomial
   \[
   t^n + a_1t^{n-1} + \dots +a_n \textrm{ with } b \textrm{ as a root and } a_i\in \im((S^*\varphi)^G).
   \]
   Since $b$ is homogeneous of degree $1$, we may assume that the $a_i$ are homogeneous of degree $i$, and we write $a_i=r_ib^i$ for some $r_i\in k$. Now since the $a_i$ are in the image of $(S^i\varphi)^G$, we have that $a_i^{n!/i}$ are in the image of $(S^{n!}\varphi)^G$ and thus that the cokernel of $(S^{n!}\varphi)^G$ is annihilated by $r_i^{n!/i}$ since $r_i^{n!/i}b^{n!}=a_i^{n!/i}$. Putting $r=1+r_1+\dots+r_n$, the equation $rb^n=0$ is equivalent to fact that $b$ is a root of the monic polynomial above, and from this we see that $r^{j}b^{n!}=0$ for any $1\leq j$. Now the unit ideal in $k$ is generated together by $r^{j!}$ and $r_i^{n!/i}$ for all $1 \leq i, j \leq n$, but all these elements of $k$ annihilate $b^{n!}$ in the cokernel of $(S^{n!}\varphi)^G$, and thus the cokernel vanishes.
\end{proof}

This alternative characterisation of power reductivity has a few useful consequences, the first being an easy proof that constant group schemes associated to finite groups satisfy power reductivity:

\begin{Cor}
    Finite groups are power reductive.
\end{Cor}

\begin{proof}
    For a finite group $G$ acting on a $k$--algebra $A$, we see that $A$ is integral over $A^G$ since $a$ is a root of the polynomial $\prod_{g\in G}(x- g(a))$. Now property (Int) follows since a composite of an integral extension $A^G\to A$ followed by a surjective map $A\to B$ is still an integral extension. Then apply the equivalence of power reductivity and (Int).
\end{proof}

\begin{Cor}
   Let $G_l$ be a flat group scheme over $l$ and  $l\to k$ be a morphism of rings. If $G_l$ is power reductive, then so is  
   its  base change $G_k$ along $l\to k$. 
\end{Cor}

\begin{proof}
    Base change doesn't affect invariants, see Lemma \ref{base change}, so it is easy to see that (Int) is preserved under base change. Now apply the equivalence of power reductivity and (Int).
\end{proof}

An interesting class of group schemes are called \textit{Chevalley group schemes}.

\begin{Def}
    A \textit{Chevalley group scheme over $\mathbb{Z}$} is a connected flat affine group scheme $G$ over $\mathbb Z$ such that the structural map $G\to \mathrm{Spec}(\mathbb Z)$ is smooth with connected reductive fibers and that admits a fiberwise maximal $\mathbb Z$--torus. In other words, $G$  is a connected split reductive flat group scheme over $\mathbb{Z}$. A \textit{Chevalley group scheme over $k$}, where $k$ is a Noetherian ring, is just the base change of a Chevalley group scheme over $\mathbb{Z}$.
\end{Def}

\begin{Rem}
    While  not all the terms in the above definition are treated with detail in this document, the upshot is that the algebraic groups that we know and love, $\text{GL}_n, \text{SL}_n, \text{PGL}_n, \text{Sp}_{2n}, $ and $\text{SO}_n$, are all Chevalley groups over $\mathbb{Z}$. We refer to \cite[Section 1]{ConradNonsplit} for a further discussion on Chevalley group schemes. 
\end{Rem}

 We have now arrived at the main technical theorems of this section: 

\begin{Th}[{Mumford's Conjecture}]\label{Mumfords Conjecture}
A Chevalley group scheme is power reductive for every Noetherian base.
\end{Th}

A full proof of this statement lies beyond the scope of this document, as it would lead us too far afield. Instead, we outline the key ideas and treat the following lemma as a black box. For a complete treatment, we refer the interested reader to \cite[Section 3]{FvdK10}.

 \begin{Lemma}\label{Key lemma for power reductivity}
     Let $G$ be a simply connected semi-simple group scheme defined over a local commutative Noetherian ring $k$. Let $M$ be a finitely generated $G$--module. Then there is a $G$--module $\mathrm{St}_r$ satisfying the following properties. 
     \begin{enumerate}
         \item $\mathrm{St}_r$ is $k$--free and of finite rank. 
         \item $\mathrm{Ext}^n_G(\Hom_k(\mathrm{St}_r,\mathrm{St}_r)^\ast,M)=0$ for all $n>0$. 
     \end{enumerate}
     Here $\Hom_k$ denotes the internal space of homomorphisms, which in particular is again a $G$--module, and $(-)^\ast$ is short for the $k$--dual, i.e., $\Hom_k(-,k)$. 
 \end{Lemma}

\begin{Rem}
The construction of the $G$--module $\mathrm{St}_r$ is analogous to the so-called \textit{Steinberg module} in the case of a field of positive characteristic. Unfortunately, some terminology is required to make this construction explicit. 
\end{Rem}

\begin{proof}[Sketch proof of Theorem \ref{Mumfords Conjecture}]
    By Proposition \ref{Power Reductivity is Local}, we know that power reductivity is a local property on $\mathrm{Spec}(k)$. Then we can replace $G$ by $G_{k_\mathrm{m}}$ for a maximal ideal $\mathrm{m}$ of $k$. Indeed, Chevalley group schemes are defined via base change from $\mathbb Z$, so our $G_{k_\mathfrak{m}}$ is also a Chevalley group scheme. In order words, we can assume that $k$ is local. 

    Let $f\colon M\to L$ be a surjective map of $G$--modules with $L$ cyclic and with trivial $G$--action. We want to show that there is some $d>0$ such that 
    \[
    S^d(f)^G\colon S^d(M)^G \to S^d L
    \]
    is surjective. We want to use Lemma \ref{Key lemma for power reductivity} which applies to simply connected semi-simply groups, so we need to reduce our problem to such setting.  To this end, consider the Chevalley group scheme $G_\mathbb{Z}$ from which $G$ is obtained via base change. Note that we can replace $G$ by $G/Z$ where $Z$ is the base change of the center $Z_\mathbb{Z}$ of $G_\mathbb{Z}$. Indeed, one verifies that $M^Z \to L$ remains surjective. Moreover, there is a natural isomorphism $(N^Z)^{G/N}\cong N^G$ for any $G$--module. The upshot is that $G/Z$ is a semi-simple group scheme, but it's not necessary simply connected. Fortunately, there is a \textit{simply connected cover} $\widetilde{G}_\mathbb{Z}$ of $G$ with center $\widetilde{Z}_\mathbb{Z}$ such that $\widetilde{G}_\mathbb{Z}/\widetilde{Z}_\mathbb{Z}= G_\mathbb{Z}/Z_\mathbb{Z}$ (see for instance \cite[Ex, 6.5.2]{ConradReductive}). It follows that we can further replace $G/Z$ by the base change of $\widetilde{G}_\mathbb{Z}$ to $k$ which is a simply connected semi-simple group scheme as we wanted. 

    Hence assume that $G$ is a simply connected semi-simple group scheme. We can also replace $M$ with a finitely generated $G$--submodule in such a way that the restriction of $f$ to such submodule remains surjective. Thus let us assume that $M$ is finitely generated $G$--module.  
    
    Let $\mathrm{St}_r$ be the $G$--module associated to $\mathrm{Ker}(f)$ from Lemma \ref{Key lemma for power reductivity}. Let $d$ denote the rank of $\mathrm{St}_r$. 
    Then $S^d(\Hom_k(\mathrm{St}_r,\mathrm{St}_r)^\ast)$ contains the determinant $\mathrm{det}$ since it is polynomial of degree $d$. Moreover, the determinant is $G$--invariant and satisfies that $\det(1_{\mathrm{St}_r})=1$. In particular, we obtain a commutative diagram 
    \begin{center}
        \begin{tikzcd}
           &  & \Hom_k(\mathrm{St}_r,\mathrm{St}_r)^\ast \arrow[ddll,"\varphi"'] \arrow[dd,bend left=50,"\psi"] \arrow[d,"\mathrm{ev}"] \\ 
            & & k \arrow[d,"\simeq"] \\
            M \arrow[rr,"f"]& & L=k\langle b\rangle
        \end{tikzcd}
    \end{center}
    where the $\varphi$ exists due to the vanishing of $\mathrm{Ext}^1_G(\Hom_k(\mathrm{St}_r,\mathrm{St}_r)^\ast,\mathrm{Ker}(f))$. By the previous observation, $S^d(\psi)$ sends the determinant to $b^{d}$.  Since the latter is $G$--invariant, we deduce that  $ S^d(\varphi)(\det)$ is a $G$--invariant in $M$ which is sent to $b^d$ under $f$ as we wanted. 
\end{proof}

\begin{Th}\label{finite over local rings}
    Finite group schemes over local rings satisfy (Int). In particular, finite group schemes are power reductive over an arbitrary Noetherian base. 
\end{Th}

For a discussion of the proof of this see \cite{van2021reductivity}. It is a special case of a more general theorem regarding groupoid schemes over local rings, see also Remark \ref{rem: discussion on fg for finite groups}.

As a corollary of these two theorems we get the following:

\begin{Cor} \label{FG for Chevalley}
    If $G$ is a Chevalley group scheme or a finite flat group scheme then it satisfies (FG) for every Noetherian base.
\end{Cor}

\begin{proof}
    The claim for Chevalley group schemes follows from the combination of Theorem \ref{Mumfords Conjecture} and Theorem \ref{Hilberts 14th}. The claim for finite flat group schemes follows from Theorem \ref{finite over local rings}, the equivalence of (Int) and power reductivity \ref{Int iff PR}, and the locality of power reductivity \ref{Power Reductivity is Local}, combined with Theorem \ref{Hilberts 14th}.
\end{proof}

\begin{Rem}\label{rem: discussion on fg for finite groups}
    Finite generation of invariants for finite group schemes dates back to 1970. In \cite[Sec. 12, Theorem 1]{Mum70}, Mumford proved (FG) when the base ring is a field. For a Noetherian base, the earliest references (as far as we can tell) are \cite[Ch. 3, Sec. 2, no 6, Corollary]{DG70} and \cite[Exp. V, Theorem 4.1]{SGA3I}. Note that \cite[Exp. V, Theorem 4.1]{SGA3I} shows that certain groupoid schemes satisfy (Int), but as we have seen, this implies (FG). In particular, finite group schemes fit into this context. We refer to \cite[Theorem 20]{van2021reductivity} and the references therein for further discussion. We thank Peter Symonds for these references.
\end{Rem}

\pagebreak

\section{Good Grosshans Filtrations}\label{section:GrosshansFiltrations}

The notion of a \textit{good filtration} plays an important role on the (CFG) property for linear algebraic group schemes over fields. For instance, let $G$ denote  $\GL_n$ or $\mathrm{SL}_n$ over an algebraically closed field $k$ of positive characteristic. Let $A$ be finitely generated $G$--algebra and $M$ be a Noetherian $AG$--module. It has been shown in \cite[Theorem 1.1]{svdK09} that if $A$ has a good filtration, then $M$ has finite \textit{good filtration dimension} and $H^\ast(G,M)$ is Noetherian as $A^G$--module. This is one of the relevant ingredients in the proof of the (CFG)-conjecture over fields given in \cite{TvdK10}. However, good filtrations are not longer  adequate if one wants to extend the (CFG) property for finite  group schemes over arbitrary Noetherian commutative rings. Here is where the notion of good Grosshans filtration becomes relevant and, in fact, it agrees with good filtrations when restricted to fields. 

In particular, good Grosshans filtrations are used in \cite{vdK15} to generalize the preceding result, thereby enabling van der Kallen to extend the (CFG) property to Noetherian commutative rings containing a field. The aim of this section is to present these results.

For the rest of this section, we let $k$ denote a Noetherian ring and $G$ denote $\GL_n$, unless stated otherwise. 

\begin{Rec}
 We let $T$ denote the standard torus of  of $G$, that is, the subfunctor of $G$ such that $T(A)$ are diagonal matrices of $G(A)$ for all $k$--algebra $A$. Let  $B$ (resp. $B^+$)  denote the subgroup of $G$ such that $B(A)$ (resp. $B^+(A)$) consists of lower (resp. upper) triangular matrices of $G(A)$ for all $A$. In particular, we have $B\cap B^+=T$. We will also consider the unipotent radical subgroup $U$ and $U^+$ of $B$ and $B^+$, respectively. Explicitly, $U(A)$ (resp. $U^+(A)$)  consists of all matrices in $B(A)$ (resp $B^+(A)$) with diagonal entries equal to 1, for all $A$. Note that the roots of $U^+$ are positive. 

Let $M$ be a $G$--module. We say that  $m\in M$ is a \textit{weight vector}  of weight $\lambda=(\lambda_1,\dots,\lambda_n)$ if
    \begin{equation*} 
\begin{pmatrix}
x_1 &  & 0 \\
 & \ddots &  \\
0 &  & x_n
\end{pmatrix} \cdot m=x_1^{\lambda_1}\dots x_n^{\lambda_n}m.
\end{equation*}    
\end{Rec}

\begin{Def}
 \textit{The Grosshans height} $\mathrm{ht} \lambda$ of a weight  $\lambda$ is defined by  
 \[\mathrm{ht}(\lambda)=\sum_{i=1}^{n}(n-2_i+1)\lambda_i\in \mathbb{Z}.\]
\end{Def}

\begin{Def}
    A \textit{Grosshans filtration} of $M$ is a filtration 
    \[0=M_{-1}\subset M_{\leq0}\subset \dots \subset \bigcup_{i\geq0}M_{\leq i}=M\]
    where $M_{\leq j}$ is the largest $G$--submodule of $M$ which weight vectors of $M_{\leq j}$ have Grosshans height at most $i$. We will denote its associated graded by $\mathrm{gr}M$.  The \textit{Grosshans hull} of $M$ is $\mathrm{ind}_B^{\GL_n}(M)^{U^+}$ and it is denoted by $\mathrm{hull}_\nabla \mathrm{gr}M$. Given a $G$--algebra $A$, we let $\mathrm{gr}_i A$ denote $A_{\leq i}/A_{<i}$. The \textit{Grosshans graded algebra} $\mathrm{gr}A$ is given by $\bigoplus_{i\leq 0}\mathrm{gr}_i A$. 
\end{Def}

\begin{Lemma}\label{embedding into grosshans}
    Let $M$ be a $G$--module. Then there is a natural injective map from the associated graded of $M$  to the Grosshans hull of $M$, in symbols,   $\mathrm{gr} M\hookrightarrow \mathrm{hull}_\nabla \mathrm{gr}M$. 
\end{Lemma}

\begin{proof}
We first describe an embedding of  $\bigoplus_i (\mathrm{gr}_i M)^{U^+}$ into $\mathrm{hull}_\nabla \mathrm{gr}M$, and then we will obtain the desired injective map by identifying $\bigoplus_i (\mathrm{gr}_i M)^{U^+}$ with $M^{U^+}$ as $T$--modules. For $\mathrm{gr}_i M$, consider the quotient map of $T$--modules $\mathrm{gr}_iM \to (\mathrm{gr_iM})^{U^+}$ and inflate it to obtain a map of $B$--modules. Consequently, this induces a map of $G$--modules $\mathrm{gr}_i M\to \mathrm{ind}_B^G((\mathrm{gr}_iM)^{U^+})$. But $\mathrm{ind}_B^G((\mathrm{gr}_iM)^{U^+})$ can be identified with $(\mathrm{gr}_i M)^{U^+}$ since the the restriction map $\mathrm{ind}_B^G k_\lambda \to k_\lambda$ induces an isomorphism of $T$--modules on the weight space of weight $\lambda$, for any dominant weight (see \cite[Proposition 20]{FvdK10}). As a consequence, the kernel of $\mathrm{gr}_i M\to \mathrm{ind}_B^G((\mathrm{gr}_iM)^{U^+})$ must have weights with Grosshans height less that $i$, and therefore it must vanish. 

  It remains to show that $\bigoplus_i (\mathrm{gr}_iM)^{U^+}$ can be identified with $M^{U^+}$ as $T$--modules.  First, note that any weight of $M^{U^+}$ is dominant. Indeed, a non-dominant weight $\lambda$ is such that $\mathrm{ind}_B^G k_\lambda$ vanishes, hence  
  \[
  \mathrm{Hom}_G(k_{-\lambda},M)=\mathrm{Hom}_B(k, \mathrm{ind}_B^G k_\lambda \otimes M)
  \]
  vanishes as well. As a consequence  $-\lambda$ is not a weight of $M^{U}$, therefore $\lambda$ is not a weight of $M^{U^+}$. 
     Returning to the claim, for a dominant weight $\lambda$ of $M^{U^+}$ with $\mathrm{ht} \lambda=i$, the so called Universal property of Weyl modules (see \cite[Proposition 21]{FvdK10}) implies that the weight space $(M^{U^+})_\lambda$ must be contained in $M_{\leq i}$, and since $\lambda$ has Grosshans height $i$, it is not contained in $M_{<i}$. This completes the claim and the proof. 
\end{proof}

\begin{Def}
    Let $M$ be a $G$--module. We say that $M$ has a \textit{good Grosshans filtration} if the above embedding from $\mathrm{gr}M$ to  $\mathrm{hull}_\nabla \mathrm{gr} M$ is an isomorphism. 
\end{Def}

An interesting property of $G$--modules with a good Grosshans filtration is that they become acyclic after tensoring them with costandar modules, just as in the case of modules with a good filtration when $k$ is a field (see Theorem \ref{modules with a good filtration are acyclic}). In fact, $G$--modules with a good Grosshans filtration are detected via the vanishing of certain cohomology groups as we will see. In particular, we obtain that  good Grosshans filtrations agree with good filtrations over fields.  

\begin{Th}\label{cohomological characterization of good Grosshans}
    A $G$--module $M$ has a good Grosshans filtration if and only if $H^i(G,M\otimes k[G/U])$ vanishes for all $i\geq1$. In fact, both conditions are equivalent to the vanishing of $H^1(G,M\otimes k[G/U])$. 
\end{Th}

\begin{Rem}
In order to connect the statement of this theorem with our previous discussion, we remind ourselves that 
\[
k[G/U]=\bigoplus_\lambda \nabla_\lambda
\]
where $\lambda$ runs over all dominant weights.     
\end{Rem}

\begin{proof}[Sketch of a proof of Theorem \ref{cohomological characterization of good Grosshans}]
Assume that $M$ has a good Grosshans filtration. Consider the associated graded $\mathrm{gr} M$, and note that $\mathrm{gr}_i M \otimes k[G/H]$ is a direct sum of modules of the form  $\mathrm{ind}_B^G k_\lambda \otimes \mathrm{ind}_B^G k_\mu\otimes N$ for each $i$, where $N$ has trivial action of $G$, but each of those modules is acyclic; for this, use the  Künneth formula and Kempf's theorem, i.e., the fact that costandard modules have cohomology just in degree 0. Hence $\mathrm{gr}_i M \otimes k[G/H]$ is acyclic as well. This implies that each term in the Grosshans filtration after tensoring with $k[G/H]$ is acyclic, so does $M\otimes k[G/H]$.

For the converse, we will show the contrapositive. Let $i$ such that $M_{<i}$ has a good Grosshans filtration but $M_{\leq i}$ does not. Moreover, we can identify
\[
M_{\leq i}^{U^+}\simeq  \bigoplus_j (\mathrm{gr}_j M_{\leq i})^{U^+}
\]
  as $T$--modules just as in the proof of Theorem \ref{embedding into grosshans}, and hence we deduce that $\mathrm{hull}_\nabla(\mathrm{gr}_i M)/\mathrm{gr}_i M$ is not-trivial. Thus we can choose $\lambda$ such that $\mathrm{Hom}_G(\nabla_\lambda, \mathrm{hull}_\nabla (\mathrm{gr}_i M)/\mathrm{gr}_i M)$ is non-trivial. In particular,   $\lambda$ must have Grosshans height at most  $i-1$, and then $\mathrm{Hom}_G(\Delta_\lambda, \mathrm{hull}_\nabla(\mathrm{gr}_i M))$ is trivial. Considering the short exact sequence 
\[0\to \mathrm{gr}_i M\to \mathrm{hull}_\nabla (\mathrm{gr}_i M)\to \mathrm{hull}_\nabla (\mathrm{gr}_i M)/ \mathrm{gr}_i M \to 0\]
we deduce that 
\[
\Ext_G^1(\Delta_\lambda, \mathrm{hull}_\nabla (\mathrm{gr}_i M))=H^1(G,\mathrm{hull}_\nabla (\mathrm{gr}_i M)\otimes \nabla_\lambda)
\]
does not vanish. Using the same reasoning, we obtain that $H^1(G, M_{\leq i}\otimes \nabla_\lambda)$ is not trivial since the $G$--module $M_{<i}$ has a good Grosshans filtration, therefore $M_{<i}\otimes k[G/U]$ is a sum of acyclic modules, and hence it is acyclic as well. Finally, another iteration of this argument using the exact sequence 
\[
0\to M_{\leq i} \to M \to M/M_{\leq i}\to 0
\]
will give us that $H^1(G,M\otimes k[G/U])$ is not  trivial, here we use that $\mathrm{Hom}_G(\Delta_\lambda, M/M{\leq i})$ vanishes.
\end{proof}

\begin{Rem}
In particular, from the previous proof we obtain that any  $G$--module $M$ with a good Grosshans filtration satisfies that its associated graded $\mathrm{gr}M$ is a direct sum of modules of the form $\nabla_\lambda\otimes J(\lambda) $, where $J(\lambda)$ has a trivial action of $G$. Of course, this generalizes the behaviour of modules with a good filtration over a field  and allows us to not place further restrictions on the modules when working over arbitrary  commutative Noetherian rings. We stress that the key fact is that the modules $\nabla_\lambda\otimes J(\lambda) $ are still acyclic.     
\end{Rem}

In order to introduce further  relevant consequences of modules with good Grosshans filtrations (at least from the perspective of this document) we first need some preparation. 

\begin{Rec}
 For a moment, we restrict our attention to  $k=\mathbb{Z}$. For a given dominant weight $\lambda\in \chi(T)$, we denote by $S'(\lambda)$ the graded algebra given in degree $n$ by $\nabla_{n\lambda}=\Gamma(G/B,\mathcal{L}(n\lambda))$. Also, we let $S(\lambda)$ denote the graded subalgebra of $S'(\lambda)$ generated by $\Delta_\lambda$, here we are considering $\Delta_\lambda$ as a submodule of $\nabla_\lambda$ with common $\lambda$ weight space. It has been shown in \cite[Lemma 39]{FvdK10} that the inclusion $S'(\lambda)\to S(\lambda)$ is \textit{universally power surjective}. In other words, the map $S'(\lambda)\to S(\lambda)$ remains power surjective after tensoring with any algebra.   
\end{Rec}

\begin{Th}\label{embedding is power surjective}
Let $A$ be a $G$--algebra. Then the embedding from $\mathrm{gr}A$ to  $\mathrm{hull}_\nabla \mathrm{gr}A$ is power surjective. 
\end{Th}

\begin{proof}
Consider a weight vector $b$ in $A^{U^+}$ of weight $\lambda$. We recall ourselves that all the weights of $A^{U^+}$ are dominant. Let $\sigma$ be a generator of the $\lambda$ weight space of $\nabla_\lambda$. We obtain a $B$--algebra map 
\[
\sigma_b\colon S'(\lambda)\otimes k\to A^{U^+} 
\]
by sending $\sigma$ to $b$. Then, using the adjuntion restriction-induction, we obtain an algebra  map 
\[
\psi_b(\lambda) \colon S'(\lambda)\otimes k\to \mathrm{hull}_\nabla(\mathrm{gr}A). 
\]
Let $b$ in $\mathrm{hull}_\nabla(\mathrm{gr}A)$. Then there are weight vectors $b_1,\ldots,b_n$ in $A^{U^+}$ of weight $\lambda_1,\ldots,\lambda_n$ respectively, such that  $b$ is in the image of $\psi_{b_1}(\lambda_1)\otimes\ldots\otimes \psi_{b_n}(\lambda_n)$. Moreover, there is a commutative diagram 

\begin{center}
   \begin{tikzcd}
   \bigotimes_{i=1}^n S(\lambda_i)\otimes k \arrow[r,"\varphi"] & \bigotimes_{i=1}^n S'(\lambda_i)\otimes k \arrow[rr, "\otimes \psi_{b_i}"] \arrow[rd] & & \mathrm{hull}_\nabla(\mathrm{gr}A) \\
   & & \mathrm{gr} A  \arrow[ru] & 
\end{tikzcd} 
\end{center}
Since each $S(\lambda_i) \to S'(\lambda_i)$ is universally power surjective, the map $\varphi$ remains power surjective. In particular, there is an integer $m$ and an element $x$ in  $\bigotimes_{i=1}^n S(\lambda_i)\otimes k$ such that  $\varphi(x)=(b_1\otimes \ldots \otimes b_n)^m$. The commutativity of the diagram implies that $b^m$ is in $\mathrm{hull}_\nabla(\mathrm{gr}A)$ as we wanted. 
\end{proof}

 As an application of the power surjectivity of the embedding $\mathrm{gr}A\to \mathrm{hull}_\nabla(\mathrm{gr}A)$, we obtain a criterion for finite generation of a $G$--algebra $A$ in terms its associated graded or its Grosshans hull.

\begin{Th}\label{A fg grA fg and hullA fg}
    Let $A$ be a $G$--algebra. Then the following conditions are equivalent. 
    \begin{enumerate}
        \item $A$ is finitely generated as $k$--algebra.
        \item $\mathrm{gr}A$  is finitely generated as $k$--algebra.
        \item $\mathrm{hull}_\nabla(\mathrm{gr}A)$ is finitely generated as $k$--algebra.
    \end{enumerate}
\end{Th}

\begin{proof}
First, assume that $A$ is finitely generated. In this case, we known that $A^U$ is finitely generated since  $G$ satisfies (FG), see Theorem  \ref{Hilberts 14th}.  By identifying $A^U$ with 
\[
\mathrm{Hom}_U(k,A)=\mathrm{Hom}_G(k,\mathrm{ind}_U^G A )= (A\otimes k[G/U])^G
\]
plus the fact that $k[G/U]$ is a finitely generated $k$--algebra, we obtain that $A^U$ is a finitely generated $k$--algebra. Now, we can use that $A^U$ and $A^{U^+}$ are isomorphic as $k$--algebras, and therefore $\mathrm{hull}_\nabla (\mathrm{gr}A)=\mathrm{ind}_B^G A^{U^+}$ is a finitely generated $k$--algebra. 

Now, assume that $\mathrm{hull}_\nabla(\mathrm{gr}A)$ is finitely generated. Recall that the embedding  $\mathrm{gr} A\hookrightarrow \mathrm{hull}_\nabla \mathrm{gr}A$ is power surjective, hence $\mathrm{hull}_\nabla \mathrm{gr}A$ is integral over $\mathrm{gr} A$ by Lemma \ref{integral1}. Hence, using  Lemma \ref{Lemma Artin-Tate} we obtain that $\mathrm{gr}A$ is finitely generated $k$--algebra as well.  

Finally, suppose that $\mathrm{gr}A$ is finitely generated. For a given finite set of generators for $\mathrm{gr} A$, it is not too hard to show that  a finite set of representatives of these generates $A$. 
\end{proof}

\begin{Th}\label{generic good Grosshans}
    For any $G$--algebra $A$ which is finitely generated as $k$--algebra, there is an integer $n>0$ such that $n\cdot  \mathrm{hull}_\nabla(\mathrm{gr} A)$ is contained in $\mathrm{gr}A$. In particular, $H^i(G,\mathrm{gr}A)$ is annihilated by $n$ for any $i>0$.
\end{Th}

\begin{proof}
   By the previous theorem, we obtain that  $\mathrm{hull}_\nabla(\mathrm{gr} A)$ is a finitely generated $k$--algebra, since $A$ is a finitely generated $k$--algebra. Thus, we can choose weight vectors $b_1,\ldots,b_s$ in $A^U$ of weight $\lambda_1,\ldots,\lambda_s$ respectively, such that  the image of $\psi_{b_1}(\lambda_1)\otimes\ldots\otimes \psi_{b_s}(\lambda_s)$ is all $\mathrm{hull}_\nabla(\mathrm{gr} A)$  (see the proof of Theorem \ref{embedding is power surjective} for the construction of the maps $\psi_{b_i}$). Now, since 
   \[
   \bigotimes_{i=1}^s S'(\lambda_i)\otimes \mathbb{Q}=\bigotimes_{i=1}^s S(\lambda_i)\otimes \mathbb{Q}
   \]
   and each graded algebra $S'(\lambda)$ is integral over its subalgebra $S(\lambda)$ (see  \cite[Lemma 37]{FvdK10}), we obtain that there is a positive integer $n$ such that 
   \[
   n\cdot\bigotimes_{i=1}^n S'(\lambda_i)\subseteq \bigotimes_{i=1}^n S(\lambda_i)
   \]
   Just as before, the map 
   \[
   \bigotimes_{i=1}^n S(\lambda_i)\otimes k \to \bigotimes_{i=1}^s S'(\lambda_i)\otimes k\to \mathrm{hull}_\nabla(\mathrm{gr}A)
   \]
   factors through $\mathrm{gr} A $. We deduce that $n\cdot \mathrm{hull}_\nabla(\mathrm{gr}A)\subseteq \mathrm{gr}A$. 

   It remains to show the second claim.  By the long exact sequence on cohomology associated to the embedding $\mathrm{gr}A\to \mathrm{hull}_\nabla\mathrm{gr}A$ and the fact that $\mathrm{hull}_\nabla\mathrm{gr}A$ is acyclic since $H^1(G,\mathrm{hull}_\nabla\mathrm{gr}A\otimes k[G/U])=0$ (see Theorem \ref{cohomological characterization of good Grosshans}), we deduce that the connecting morphism 
   \[
   H^{i-1}(G,\mathrm{hull}_\nabla\mathrm{gr}A/\mathrm{gr}A)\to H^i(G,\mathrm{gr}A)
   \]
   is surjective for $i>0$. The previous claim gives us that $H^i(G,\mathrm{gr}A)$ must be annihilated by $n$. 
\end{proof}

\begin{Cor}\label{Cor: A1/m is acyclic for some m}
    Let $A$ be a $G$--algebra finitely generated as $k$--algebra. Then there is a natural number $n>0$ such that $A[1/n]$ has a good Grosshans filtration and $H^i(G,A)\otimes \mathbb{Z}[1/n]$ vanishes for all $i>0$.
\end{Cor}

\begin{proof}
    By Theorem \ref{generic good Grosshans}, there is $n>0$ such that $n\cdot \mathrm{hull}_\nabla(\mathrm{gr}A)\subset \mathrm{gr}A$. Using exactness of localizations, we  deduce that $A[1/n]$ has a good Grosshans filtration. In particular, Theorem \ref{cohomological characterization of good Grosshans} gives us the vanishing of $H^i(G,A[1/n])$, which agrees with $H^i(G,A)\otimes \mathbb{Z}[1/n]$, for all $i>0$.  
\end{proof}

\begin{Th}\label{finite resolution by acyclic modules}
 Let $A$ be a finitely generated $G$--algebra, and $M$ a Noetherian $AG$--module. If $A$ has a good Grosshans filtration, then there is a finite resolution 
 \[
 0\to M\to N_0\to \dots \to N_d \to 0
 \]
 where the $N_i$ are Noetherian $AG$--modules with good grosshans filtrations.
\end{Th}

We refer the interested reader to \cite{vdK15} for a proof. 
As we mentioned at the beginning of this section, the previous theorem has a predecessor when the ring $k$ is a field of positive characteristic (see \cite[Theorem 1.1]{svdK09}), and some similar consequences can be obtained from it as we will see. 

\begin{Cor}\label{Cor: resolution by modules with a good filtration}
   Let $A$ be a finitely generated $G$--algebra, and $M$ a Noetherian $AG$--module. If $A$ has a good Grosshans filtration, then $H^i(G, M)$ are Noetherian $A^G$--modules and they vanish for $i>>0$.
\end{Cor}

\begin{proof}
 Recall that if $N$ has a good Grosshans filtration, then $N$ is acyclic by Theorem \ref{cohomological characterization of good Grosshans} and $N^G$  is Noetherian over $A^G$ by Theorem \ref{Hilberts 14th}. Hence the exact sequence  
 \[
 0\to M\to N_0\to \dots \to N_d \to 0
 \]
 from Theorem \ref{finite resolution by acyclic modules} give us a finite resolution of $M$ by acyclic $G$--modules, hence it computes $H^\ast(G,M)$. The conclusion is clear now.
\end{proof}

Finally, we can state the main consequence of the previous results (see \cite[Theorem 10.1]{vdK15}).

\begin{Th}\label{Th:CFG when k contains a field}
   Let $k$ be a Noetherian ring containing a field $\mathbb{F}$. Assume that $G$ is the group scheme over $k$ obtained from a geometrically reductive affine algebraic group scheme $G_\mathbb{F}$ over $\mathbb{F}$ by base change along the inclusion $\mathbb{F} \to k$. Then $G$ has the (CFG) property.
\end{Th}

\begin{proof}
   By Proposition  \ref{Thm: Group schemes over fields are embedding in GL_n}, we obtain a closed embedding $G_\mathbb{F} \to \GL_{n,\mathbb{F}}$ for some $n \geq 1$, and $\GL_{n,\mathbb{F}}/G_\mathbb{F}$ remains affine over $\mathbb{F}$. Base change preserves these properties. Now, for a $G$--algebra finitely generated as a $k$--module, we have that $\mathrm{ind}_G^{\GL_n}(A)$ remains finitely generated. Moreover, the Reduction Lemma \ref{vdK's Reduction Lemma} gives us that 
   \[
   H^\ast(\GL_n,\mathrm{ind}_G^{\GL_n}(A)) \cong H^\ast(G,A).
   \]
    Hence, the problem reduces to showing that $\GL_n$ has the (CFG) property. This will be proved in the following section.
\end{proof}

\pagebreak

\section{CFG for $\GL_n$}\label{section:CFGforGL}

This part follows closely Antoine Touz\'e's lectures given during the Master class on \textit{New Developments in Finite Generation of Cohomology} that took place in Bielefeld University in September 2023. We also refer to Touz\'e's \cite{Tou2022} for a more detailed treatment.

The goal of this section is to prove the following theorem. 

\begin{Th}\label{CFG for GLn}
    Let $k$ be a Noetherian commutative ring, and $A$ be $\GL_n$--algebra finitely generated as $k$--algebra.  If $k$ contains a field, then $H^\ast(\GL_n,A)$ is a finitely generated $k$--algebra. In other words, $\GL_n$ has the (CFG) property. Moreover, if $M$ is a Noetherian $A$--module with compatible $\GL_n$--action, then $H^\ast(G,M)$ is a Noetherian $H^\ast(G,A)$--module.
\end{Th}

We will consider two cases in order to prove (CFG) for $\GL_n$: first when the ring $k$ contains $\mathbb{Q}$; second, when $k$ contains $\mathbb{F}_p$ for some prime $p$. The second claim will follow by Proposition \ref{Prop: CFG implies finite genarion of modules}, we also refer to \cite[Theorem 1.5]{TvdK10} for further details.

\subsubsection*{$k$ contains $\mathbb{Q}$} In this case, \textit{there are no 
 higher cohomology groups} as we will see.

\begin{Lemma}\label{no cohomology in char 0}
    Let $M$ be a $\GL_n$--algebra. Then $H^i(\GL_n,M)=0$ for all $i>0$.  
\end{Lemma}

\begin{proof}
    Recall that $\GL_n=({\GL_n}_\mathbb{Q})_k$, and hence by base change (see Lemma \ref{base change}) we obtain 
    \[
    H^\ast({\GL_n},M)=H^\ast({\GL_n}_\mathbb{Q},M)
    \]
    Moreover, since ${\GL_n}_\mathbb{Q}$ is linearly reductive we deduce that $H^\ast({\GL_n}_\mathbb{Q},M)=H^0({\GL_n}_\mathbb{Q},M)$. 
\end{proof}

\begin{Prop}\label{CFG for GLn k containing Q}
    Let $k$ be a Noetherian ring containing $\mathbb{Q}$. Then $\GL_n$ satisfies the (CFG) property.
\end{Prop}

\begin{proof}
    Let $A$ be a $\GL_n$--algebra finitely generated over $k$. By Lemma \ref{no cohomology in char 0} we have that $H^\ast(\GL_n,A)=H^0(\GL_n,A)$, and the latter is finitely generated since $\GL_n$ has (FG) by Theorem \ref{Hilberts 14th}. 
\end{proof}

\subsubsection*{$k$ contains $\mathbb{F}_p$}

We first need to recall some useful definitions and establish notation that will be needed in this section. In order to avoid confusion, when we want to emphasize that $k$ is a field, we use $\mathbb{F}$ instead. Note that $p\cdot R=0$ and $\mathbb{F}$ is of characteristic $p$ by our initial assumption. 

\begin{Def}
Let $r$ be a non-negative integer. The \textit{$r$th-Frobenius morphism} of a $k$--algebra $A$ is the morphism  $F^r_A\colon A\to A$, $a\mapsto a^{p^r}$. The \textit{Frobenius twist of a $k$--module $M$} is $M\otimes_k  {}_rk$ given by the the base change along $F^r_k$ and it is denoted by $M^{(n)}$, note that we use the notation ${}_r k$ to emphasize that $k$ if viewed as a $k$--module via $F^t_k$. If furthermore $M$ is a $k$--algebra, then $M^{(r)}$ is a $k$--algebra as well. For a $k[\GL_n]$--comodule $(M,\Delta_M)$, we will consider the Frobenius twist $M^{(r)}$ as a $k[\GL_n]$--comodule via the comodule map
\[
M\xrightarrow[]{\Delta_M^{(r)}} (M\otimes k[\GL_n])^{(r)}\xrightarrow[]{\simeq} M^{(r)}\otimes (k[\GL_n])^{(r)}\xrightarrow[]{\mathrm{id}\otimes P^r} (M^{(r)})\otimes k[\GL_n] 
\]
 where $P^r\colon (k[\GL_n])^{(r)}\to k[\GL_n]$ is the algebra map given by sending $a$ to $a^{p^r}$. From this, we obtain a $\GL_n$--module $M^{(r)}$ from a $\GL_n$--module $M$.
\end{Def}

\begin{Rem}
    The Frobenius twist of a finite dimensional $\mathbb{F}$-vector space $V$ can be constructed as the subspace 
$\langle v^{p^r}\mid v\in V\rangle \subset S^{p^r}(V)$, here $S^{p^r}(V)$ denotes the $p^r$th symmetric power of $V$. Let us denote this construction by $V^{(r)'}$. The map 
\[V^{(r)}\to V^{(r)'}, \quad v\otimes1\mapsto v^{p^r}\]
induces an isomorphism. 
\end{Rem}

\begin{Rem}
  Assume that $\mathbb{F}$ is perfect.  Let $v^{(r)}$ denote $v^{p^r}$ when  viewed in $V^{(r)}$. If the $\mathbb{F}$--vector space $V$ has a basis $b_1,\ldots,b_d$, then $b_1^{(r)},\ldots,b_d^{(r)}$ is a basis for $V^{(r)}$.  
\end{Rem}

Recall that the embedding $\mathrm{gr}A\to \mathrm{hull}_\nabla(\mathrm{gr}A)$ is power surjective by Theorem \ref{embedding is power surjective} and both $\mathrm{gr}A$ and $\mathrm{hull}_\nabla(\mathrm{gr}A)$ are finitely generated $k$--algebras by Theorem \ref{A fg grA fg and hullA fg}. In particular, there is  $r$ big enough such that $ x^{p^r}$ lies in $\mathrm{gr}A$, for all $x\in \mathrm{hull}_\nabla(\mathrm{gr}A)$. We obtain a  morphism of $\GL_n$--modules
\[
 (\mathrm{hull}_\nabla(\mathrm{gr}A))^{(r)} \to \mathrm{hull}_\nabla(\mathrm{gr}A), \quad
    a^{(r)}  \mapsto a^{p^r}. 
\]
\noindent and of course, it is a morphism of $k$--algebras as well. By our previous observation, we know that $a^{p^r}$ lies in $\mathrm{gr}A$, so we can consider it as a morphism from $(\mathrm{hull}_\nabla(\mathrm{gr}A))^{(r)}$ to $\mathrm{gr}A$.   This morphism is Noetherian\footnote{Recall that a morphism $f\colon A\to B$ of $G$--algebras is \textit{Noetherian} if  $B$ is a Noetherian $A$--module along $f$.} when $\mathbb{F}$ is perfect, e.g. $\mathbb{F}=\mathbb{F}_p$. Indeed, $(\mathrm{hull}_\nabla(\mathrm{gr}A))^{(r)}$ is finitely generated over $\mathbb{F}$, hence Noetherian by Hilbert's Basis Theorem, and $\mathrm{gr}A$ is finitely generated over $(\mathrm{hull}_\nabla(\mathrm{gr}A))^{(r)}$, for instance, a basis as $\mathbb{F}$-vector space is a generating set.  

Let us summarize our findings so far.  For a general $k$ containing $\mathbb{F}_p$ we obtain a surjective map
\begin{equation*}
\mathrm{hull}_\nabla(\mathrm{gr}A)\otimes_{\mathbb{F}_p}k\xrightarrow[]{\pi} \mathrm{hull}_\nabla(\mathrm{gr}A)    
\end{equation*}
 that is a morphism of ${\GL}_{n,k}$--algebras. Moreover, $\mathrm{hull}_\nabla(\mathrm{gr}A)$ is a $\GL_{n,{\mathbb{F}_p}}$--algebra with a good Grosshans filtration, hence a good filtration since we are over a field.  Thus we get a morphisms of $\GL_{n,k}$--algebras 
\begin{equation}\label{definition of the map Xi}
    (\mathrm{hull}_\nabla(\mathrm{gr}A))^{(r)}\otimes_{\mathbb{F}_p}k \xrightarrow[]{\Xi} \mathrm{gr}A, \quad
    b^{(r)}\otimes \lambda \mapsto \pi(b)^{p^r}\lambda
\end{equation}
\noindent which is Noetherian. This map, together with the main results from Section \ref{section:GrosshansFiltrations}, Friedlander-Sulslin results in \cite{FS97} and some homological algebra will allow us to show that the cohomology of $\mathrm{gr}A$ is a finitely generated algebra over $k$. This will be used to conclude the same property but for $A$. 

We need some preparations. 

\begin{Def}
  Let $G$ be a group scheme defined over $k$. We let $G^{(r)}$ denote the group scheme given by $G^{(r)}(A)=A^{(r)}$, for all $k$--algebras $A$. The \textit{$r$-th Frobenius morphism on $G$}, is the morphism of groups schemes $F_G^r\colon G\to G^{(r)}$ defined by $G(P^r_A)\colon G(A)\to G^{(r)}(A)$, where $P^r_A$ is the algebra map $A\to A^{(r)}$, $a\mapsto a^{(p^r)}$.   The \textit{$r$th Frobenius kernel of $G$} is the normal subgroup scheme of $G$ given as the kernel of the map $F_G^r\colon G\to G^{(r)}$. 
\end{Def}

\begin{Rem}
    The group scheme $G^{(r)}$ corresponds to the fiber product of $G\to \mathrm{Spec}(k)$ and 
    \[
    \mathrm{Spec}(k)\xrightarrow[]{\mathrm{Spec}(F^r_k)}\mathrm{Spec}(k).
    \]
    In other words, it is obtained by base change of schemes. 
\end{Rem}

\begin{Ex}
    The $r$th Frobenius kernel of $\GL_{n,{\mathbb{F}_p}}$ is given by the group scheme represented by the Hopf algebra $\mathbb{F}[\GL_n]/\langle x_{i,j}^{p^r}=\delta_{i,j}\rangle.$
\end{Ex}

Each $r$th Frobenius kernel of $\GL_n$ will give us an extension of group schemes 
\begin{equation}\label{Extension of group schemes by Frobenius kernels}
  1\to   (\GL_n)_r \to \GL_n\to \GL_n/ (\GL_n)_r \to 1.
\end{equation} 
Our next step is to study the cohomology of $\mathrm{gr}A$ making use of the so called Lyndon–Hochschild–Serre spectral sequence associated to the above extension of group schemes. We will give a quick overview of the tools we need. 

\begin{Def}
    A \textit{spectral sequence of algebras} is a sequence of bigraded differential algebras $(E^{\ast,\ast}_r,d_r)_{r\geq r_0}$ with differentials 
    \[d_r\colon E^{s,t}_r\to E^{s+r,t+1-r}_r\]
    and such that $H(E^{\ast,\ast}_r,d)=E^{\ast,\ast}_{r+1}.$ We refer to $E^{\ast,\ast}_{r_0}$ as the \textit{initial page} of the spectral sequence. A spectral sequence  $(E^{\ast,\ast}_r,d_r)_{r\geq r_0}$ \textit{converges} to the graded $k$--algebra $H^\ast$ if the following properties hold. 
    \begin{enumerate}
        \item $E^{\ast,\ast}_\infty$ exists. That is, for any $r,s$ there exists $r_{s,t}$ such that $E^{s,t}_r=E^{s,t}_{r_{s,t}}$ for every $r\geq r_{s,t}$. In this case, $E^{s,t}_\infty:=E^{s,t}_{r_{s,t}}$.
        \item There is an algebra filtration on $H^\ast$ such that 
        \[\bigoplus_{s+t=n}E^{s,t}_\infty\simeq\mathrm{gr}H^n.\] 
    \end{enumerate}
     We will write $E^{\ast,\ast}_r\Rightarrow H^\ast$ if $(E^{\ast,\ast}_r,d_r)_{r\geq r_0}$ converges to $H^\ast$.
\end{Def}

We will record a version of the Lyndon–Hochschild–Serre spectral sequence that is needed in this section. For more details, we refer to \cite[Section I.6.6]{jantzen2003representations}.

\begin{Lemma}
    Let $1\to H\to G\to G/H\to 1$ be an extension of group schemes. Then for any $G$--module $M$ there is a spectral sequence  
    \[
    {}^{LHS}E^{s,t}_2(M)=H^s(G/H,H^t(H,M))\Rightarrow H^{s+t}(G,M)
    \]
    here we are considering the $H^t(H,M)$ as $G/H$--modules by regarding them as the derived functors of 
    \[(-)^H\colon \mathbf{Rep}(G)\to \mathbf{Rep}(G/H).\]
    Moreover, if $A$ is a $G$--algebra, then ${}^{LHS}E^{s,t}_2(A)$ is a spectral sequence of algebras.  
\end{Lemma}

\begin{Th}\label{cohomology of grA is finitely generated}
    Let $A$ be a $\GL_n$--algebra finitely generated as $k$--algebra. Then $H^\ast(\GL_n,\mathrm{gr}A)$ is finitely generated as $k$--algebra.  
\end{Th}

The proof will consist in verifying that the spectral sequence ${}^{LHS}E^{s,t}_2(\mathrm{gr}A)$ corresponding to the extension of group schemes given in Equation \ref{Extension of group schemes by Frobenius kernels}, for a big enough $r$, satisfies the conditions of the following lemma. 
\begin{Lemma}\label{Lemma analysis of ss}
    Let $k$ be a Noetherian ring containing $\mathbb{F}_p$ and $(E^{\ast,\ast}_r,d_r)_{r\geq r_0}$ be a spectral sequence of $k$--algebras which converges to $H^\ast$. Assume that: 
    \begin{itemize}
        \item[(C1)] $E^{\ast,\ast}_{r_0}$ is finitely generated as a $k$--algebra.
        \item[(C2)] There is $r_\mathrm{max}$ such that $E^{\ast,\ast}_{r_\mathrm{max}}=E^{\ast,\ast}_\infty$.
        \item[(C3)] The filtration on $H^\ast$ is finite and restrict to a finite filtration on $H^n$ for every degree $n$.
    \end{itemize}
    Then $H^\ast$ is finitely generated as a $k$--algebra. 
\end{Lemma}

\begin{proof}
    For a differential graded algebra $(A^\ast,d)$, consider the subalgebra $B^\ast$ generated by $p$-powers of elements in $A^\ast$. By definition, $A^\ast$ is integral over $B^\ast$. Since $pR=0$ we get that $d(a^p)=pd(a)=0$. In other words, $B^\ast\subseteq \mathrm{ker}(d)$, and hence $A^\ast$ is integral over $\mathrm{ker}(d)$ as well. By the Artin-Tate Lemma we obtain that $\mathrm{ker}(d)$ is a finitely generated algebra over $k$, then $H^\ast(A^\ast,d)$ is finitely generated over $k$. Thus the pages $E_r^{\ast,\ast}$ are finitely generated algebras over $k$, for all $r\geq r_0$. Since the spectral sequence collapses, the algebra $\bigoplus_{s+t=n}E^{s,t}_\infty$ must be finitely generated over $k$ and it agrees with $\mathrm{gr}(H^\ast)$ for some algebra filtration on $H^\ast$. If furthermore the filtration on $H^\ast$ is finite and restricts to a finite filtration on each degree, then we can use a finite set of generators of $\mathrm{gr}(H^\ast)$ to generate $H^\ast$. 
\end{proof}

In order to verify conditions (C1) and (C2) we need to invoke some results from \cite{FS97} and, for convenience, we will include the statements that we need here. Let $\mathfrak{gl}_n$ denote the \textit{adjoint representation} of $\GL_{n,{\mathbb{F}_p}}$, that is, the space of $n\times n$--matrices with the conjugation action. Recall that $\mathfrak{gl}_n^{(r)}$ denotes the $r$th-Frobinius twist of $\mathfrak{gl}_n$. Note that $\mathfrak{gl}_n^{(r)}$ is a vector space of dimension $n^2$ with trivial action of $(\GL_{n,{\mathbb{F}_p}})_r$ as we will see in the following remark. 

\begin{Rem}
Let  $l$ be a ring, and let $V$ be a $l$--module with an action $\rho$ of $\GL_{n,l}$. Consider \[
V^{[r]}=(V,\rho\circ F^r)
\] 
where $F^r\colon \GL_{n,l}\to \GL_{n,l}$ is given by $[a_{ij}]\mapsto [a_{ij}^{p^r}]$. If $l=\mathbb{F}_p$, then $V^{[r]}=V^{(r)}$. Since $(\GL_n)_r=\mathrm{Ker}F^r$, we obtain that the action of $(\GL_n)_r$ on $V^{[r]}$  is trivial.    
\end{Rem}

Let  
\[
\mathbb{S}_r=\bigotimes_{i=1}^r \mathbb{F}_p[\mathfrak{gl}_n^{(r)}\langle 2p^{i-1}\rangle]
\]
be the tensor product of polynomial algebras over copies of $\mathfrak{gl}_n^{(r)}$. This is a  graded $\mathbb{F}_p$--algebra and the symbol $\langle 2p^{i-1}\rangle$ indicates that the $i$th copy of $\mathfrak{gl}_n^{(r)}$ is in degree $2p^{i-1}$. In particular, $\mathbb{S}_r$ is a $\GL_{n,{\mathbb{F}_p}}$--algebra that is finitely generated as $\mathbb{F}_p$--algebra. Moreover, the action of $\GL_{n,{\mathbb{F}_p}}$ on $\mathbb{S}_r$ restricts to a trivial action of $(\GL_{n,{\mathbb{F}_p}})_r$.

For the rest of this section, let $G$ denote $\GL_n$ over $k$. 

\begin{Th}\label{FS full strength}
    There exists a Noetherian morphism of $G_{\mathbb{F}_p}$--algebras 
    \[
    \upphi_{FS}\colon \mathbb{S}_r\to H^\ast(G_r,\mathbb{F}_p).
    \]
    Moreover, if $C$ and $D$ are two $G_{\mathbb{F}_p}$--algebras and $f\colon C\to D$ is a Noetherian map, then the morphism  
\begin{align*}
    \upphi_{FS}\colon \mathbb{S}_r\otimes H^0(G_r,C) & \to H^\ast(G_r,D)\\
    x\otimes g & \mapsto \upphi_{FS}(x)\cup f(g)
\end{align*} is Noetherian as well. 
\end{Th}

Recall from Equation \ref{definition of the map Xi} that we have a Noetherian morphism 
\[
\Xi\colon (\mathrm{hull}_\nabla(\mathrm{gr}A))^{(r)}\otimes_{\mathbb{F}_p}k\xrightarrow[]{\pi} \mathrm{hull}_\nabla(\mathrm{gr}A)
\]
for some $r$ large enough. Thus we can apply Theorem \ref{FS full strength} to $D=\mathrm{gr}A$, $C=(\mathrm{hull}_\nabla(\mathrm{gr}A))^{(r)}\otimes_{\mathbb{F}_p} k$ and $f=\Xi$. Hence we obtain a Noetherian  map 
\[
\upphi_{FS}(\mathrm{gr}A)\colon \mathbb{S}_r\otimes_{\mathbb{F}_p}H^0(G_{r,\mathbb{F}_p}, (\mathrm{hull}_\nabla(\mathrm{gr}A))^{(r)}\otimes_{\mathbb{F}_p}k)\to H^\ast({G_r}_{\mathbb{F}_p},\mathrm{gr}A).
\]
But, as we already observed, the action of $G_{r,\mathbb{F}_p}$ on $(\mathrm{hull}_\nabla(\mathrm{gr}A))^{(r)}=(\mathrm{hull}_\nabla(\mathrm{gr}A))^{[r]}$ is trivial, then it is also trivial on $(\mathrm{hull}_\nabla(\mathrm{gr}A))^{(r)}\otimes_{\mathbb{F}_p} k$. Then we obtain a map 
\[
\upphi_{FS}(\mathrm{gr}A)\colon \mathbb{S}_r\otimes_{\mathbb{F}_p} (\mathrm{hull}_\nabla(\mathrm{gr}A))^{(r)}\otimes_{\mathbb{F}_p}k\to H^\ast({G_r},\mathrm{gr}A).
\]
Here we have used the Base Change Lemma \ref{base change} to remove the $\mathbb{F}_p$ from the right hand side. We are almost ready to prove Theorem \ref{cohomology of grA is finitely generated}, but first we need to \textit{untwist}  $\mathbb{S}_r\otimes_{\mathbb{F}_p} (\mathrm{hull}_\nabla(\mathrm{gr}A))^{(r)}\otimes_{\mathbb{F}_p}k$. For this, consider the following diagram

\begin{center}
        \begin{tikzcd}
           G_r \arrow[r, hook] & G \arrow[rd,"F^r"']\arrow[r]&G/G_r \arrow[d,"\simeq"] \\
            & &\mathbf{G}
        \end{tikzcd}
    \end{center}

\noindent where $\mathbf{G}$ is simply $G$, however, we want to emphasize that we are referring to the one on the lower-right corner of the above diagram. Then, regarding $\mathbb{S}_r\otimes H^0({G_r}_{\mathbb{F}_p},B^{(r)}\otimes k)$ as a $\mathbf{G}$--module is isomorphic to 
\[
\bigotimes_{i=n}^r\mathbb{F}_p[\mathfrak{gl}_n\langle 2p^{i-1}\rangle]\otimes_{\mathbb{F}_p} \mathrm{hull}_\nabla(\mathrm{gr}A) \otimes_{\mathbb{F}_p}k. 
\]
Indeed, we know that  $\mathfrak{gl}_n^{(n)}=\mathfrak{gl}_n^{[r]}$, and this corresponds to the adjoint representation of $\mathbf{G}$ under the pullback along $F^r$.  

\begin{proof}[Proof of Theorem \ref{cohomology of grA is finitely generated}]
    By \cite{ABW82} we know that $\bigotimes_{i=n}^r\mathbb{F}_p[\mathfrak{gl}_n\langle 2p^{i-1}\rangle]$ has a good filtration. Moreover, we already observed that $\mathrm{hull}_\nabla(\mathrm{gr}A)$ has a good filtration as well. It follows that 
    \[
    \bigotimes_{i=n}^r\mathbb{F}_p[\mathfrak{gl}_n\langle 2p^{i-1}\rangle]\otimes_{\mathbb{F}_p} \mathrm{hull}_\nabla(\mathrm{gr}A) \otimes_{\mathbb{F}_p}k
    \]
    has a good filtration. Thus we can use our Noetherian map $\upphi_{FS}(\mathrm{gr}A)$ combined with the fact that $\mathbf{G}=GL_n$ and Corollary \ref{Cor: resolution by modules with a good filtration} to obtain:
    
    \begin{itemize}
    \item $H^s(G/G_r,H^t(G_r,\mathrm{gr}A))=0$ if $s>>0$, this is condition (C2).
    \item $H^\ast(G/G_r,H^\ast(G_r,\mathrm{gr}A))$ is finitely generated as a $k$--algebra, this is condition (C1).
\end{itemize}
 Condition (C3) always holds for the Lyndon–Hochschild–Serre spectral sequence. By Lemma \ref{Lemma analysis of ss} we conclude that $H^\ast(G,\mathrm{gr}A)$ is finitely generated as $k$--algebra.
\end{proof}

\begin{Rem}
    We stress that the previous proof makes use of the results obtained in Section \ref{section:GrosshansFiltrations} regarding good Grosshans filtrations. Hence, as promised, this highlights their relevance in the (CFG) property for $\GL_n$ and for finite group schemes.
\end{Rem}


We will use analysis of  the spectral sequence of algebras associated to a filtered algebra $A$:
\begin{equation}\label{Eq: 2nd spectral sequence needed}
E_1^{s,t}(A)=H^{s+t}(G,\mathrm{gr}_{-s}A)\Rightarrow H^{s+t}(G,A)    
\end{equation}

to prove Theorem \ref{CFG for GLn}. In particular, we will consider $A$ with its Grosshans filtration. Hence, condition (C1) is satisfied by Theorem \ref{cohomology of grA is finitely generated}, so it remains to prove condition (C2). Let us recall ourselves the following result due to Evens, see \cite{Eve61}. 

\begin{Lemma}
    If there exists a ring $R$ such that  $(E^{\ast,\ast}_r,d_r)_{r\geq1}$ is a spectral sequence of $R$--modules and $E^{\ast,\ast}_1$ is a Noetherian module over $R$, then there is $r_\mathrm{max}$ such that $E^{\ast,\ast}_{r_\mathrm{max}}=E^{\ast,\ast}_\infty$. 
\end{Lemma}

Hence we have to find a suitable ring $R$ such that the spectral sequence from \ref{Eq: 2nd spectral sequence needed} becomes a spectral sequence of $R$--modules satisfying the conditions from the previous lemma. A natural candidate is 
\[
\bigoplus_{i\geq0}t^iA_{\leq i}\subset A[t]
\]
where $A_{\leq i} $ denotes the Grosshans filtration of $A$. Let $\mathcal{A}$ denote $\bigoplus_{i\geq0}t^iA_{\leq i}$.

Evaluation on $t=1$ gives us 
\[
\mathrm{ev}\colon\mathcal{A}\to A
\]
which is a map of $k$--algebras that is compatible with the filtration on $\mathcal{A}$ given by $F_j\mathcal{A}=\bigoplus_{0\leq i\leq j} t^i A_{\leq i}$. Hence, the evaluation map induces a morphism of spectral sequences (all maps are of $k$--algebras) 
\[
E^{\ast,\ast}_r(\mathrm{ev})\colon E^{\ast,\ast}_r(\mathcal{A}) \to E^{\ast,\ast}_r(A)
\]
and note that $E^{\ast,\ast}_r(\mathcal{A})$ is a trivial spectral sequence; all the differentials are zero and all its pages are isomorphic to $H^\ast(G,\mathcal{A})$. Let $R$ denote $H^\ast(G,\mathcal{A})=H^\ast(G,\mathrm{gr}A)$. Then $E^{\ast,\ast}(A)$ is a spectral sequence of $R$--modules. It remains to prove the following claim.

\begin{Claim}\label{E_1 is Noetherian over R}
    $E^{\ast,\ast}_1(A)=H^\ast(G,\mathrm{gr}A)$ is Noetherian over $R$.    
\end{Claim}

First, we will need the following lemma. 

\begin{Lemma}
    $E^{\ast,\ast}_1(A)=H^\ast(G,\mathrm{gr}A)$ is Noetherian over $R$ if and only there is a Noetherian map $E$ making the following diagram commutative. 
\begin{center}
        \begin{tikzcd}
            & H^\ast(G,\mathcal{A}) \arrow[rd, dashrightarrow,"E"']\arrow[r]&H^\ast(G_r,\mathcal{A})^{G/G_r} \arrow[d] \\
            & &H^\ast(G_r,\mathrm{gr}A)^{G/G_r}
        \end{tikzcd}
    \end{center}
    \noindent where the horizontal and vertical maps are induced by the corresponding restrictions. 
\end{Lemma}

\begin{proof}
    First, note that the terms $H^\ast(G_r,\mathcal{A})^{G/G_r}$ and $H^\ast(G_r,\mathrm{gr}A)^{G/G_r}$ correspond to ${}^{LHS}E^{0,\ast}_2(\mathcal{A})$ and ${}^{LHS}E^{0,\ast}_2(\mathrm{gr} A)$ from the spectral sequence associated to \ref{Extension of group schemes by Frobenius kernels}, respectively. Here we are using that $G_r$ acts trivially on $H^\ast(G_r,M)$ for any $G$--module $M$. 
    
    Moreover, the map $\mathcal{A}\to \mathrm{gr}A$ induces a morphism of spectral sequences ${}^{LHS}E^{\ast,\ast}(\mathcal{A})\to {}^{LHS}E^{\ast,\ast}(\mathrm{gr}A)$ from where we obtain the following commutative diagram 
\[\begin{tikzcd}
    H^\ast(G,\mathcal{A}) \arrow[r,twoheadrightarrow] \arrow[d] & {}^{LHS}E^{0,\ast}_2(\mathcal{A}) \arrow[d,"f"] \arrow[r,hookrightarrow] & {}^{LHS}E^{\ast,\ast}_\infty(\mathcal{A}) \arrow[d]  \\
    H^\ast(G,\mathrm{gr}A) \arrow[r,twoheadrightarrow] & {}^{LHS}E^{0,\ast}_2(\mathrm{gr} A)\arrow[r,hookrightarrow] & {}^{LHS}E^{\ast,\ast}_\infty(\mathrm{gr} A)
\end{tikzcd}\]
and since the bottom right map is Noetherian by the proof of Theorem \ref{cohomology of grA is finitely generated},  then all vertical maps are Noetherian if and only if any of the vertical maps is Noetherian. 

   The point is that for an integer $r$  such that ${}^{LHS}E^{\ast,\ast}_r(\mathrm{gr} A)={}^{LHS}E^{\ast,\ast}_\infty(\mathrm{gr} A)$ one obtains that the  map of raising to the $p^r$-power 
   \[
   p^r\colon {}^{LHS}E^{0,\ast}_2(\mathrm{gr} A)^{(r)}\to {}^{LHS}E^{0,\ast}_2(\mathrm{gr} A) 
   \]
    must factors as 
\[
{}^{LHS}E^{0,\ast}_2(\mathrm{gr} A)^{(r)}\to {}^{LHS}E^{0,\ast}_\infty(\mathrm{gr} A)\hookrightarrow {}^{LHS}E^{0,\ast}_2(\mathrm{gr} A)
\]
where both maps are Noetherian, and hence the composite is Noetherian as well. Here we used that $k$ has characteristic $p$, and hence the $p$-powers of elements in the page $s$ are cycles in the same page of the spectral sequence. Now, consider the following commutative diagram
\[\begin{tikzcd}
    H^\ast(G,\mathcal{A}) \arrow[r] \arrow[rd] & {}^{LHS}E^{0,\ast}_2(\mathcal{A})  \arrow[r] & {}^{LHS}E^{0,\ast}_2(\mathrm{gr}A)   \\
     & {}^{LHS}E^{0,\ast}_\infty(\mathcal{A})\arrow[r,"f"] \arrow[u,hookrightarrow] & {}^{LHS}E^{\ast,\ast}_\infty(\mathrm{gr} A) \arrow[u,hookrightarrow]
\end{tikzcd}\]
Hence the conclusion follows since $E$ is the composition of the top maps. 
\end{proof}

We sketch a strategy to prove that the map $E$ is Noetherian. Consider the following diagram 
\begin{center}
\begin{tikzcd}
            & H^\ast(G,\mathcal{A}) \arrow[r,"E"]&H^\ast(G_r,\mathrm{gr}A)^{G/G_r}  \\
            & & (\$_r \otimes_{\mathbb{F}_p} H^0(G_r,\mathcal{A}))^{G/G_r} \arrow[u,"(\upphi_{FS})^{G/G_r}"']\arrow[lu, dashrightarrow,"\upphi_{vdK}"]
        \end{tikzcd}
    \end{center}
\noindent where $\upphi_{FS}$ denotes Friedlander-Suslin map which is Noetherian. Since $G/G_r$ satisfies (FG), we obtain that $(\upphi_{FS})^{G/G_r}$ is Noetherian as well. Note that if $\upphi_{vdK}$ exists, then $E$ is a Noetherian map. 

The map $\upphi_{vdK}$ is analogous to $\upphi_{FS}$, but \textit{why $\upphi_{vdK}$ exists?}  For this, we need classes in $H^\ast(\GL_n)$ generalizing Friedlander-Suslin's classes. The following result is due to Touz\'e, see \cite{Tou10}.

\begin{Th}\label{Thm Touze}
There are classes 
\[
c[i]\in H^{2i}(\GL_n,\Gamma^i(\mathfrak{gl}_n^{(1)}))
\]
where $\Gamma^i(\mathfrak{gl}_n^{(1)})$ denotes $((\mathfrak{gl}_n^{(1)})^\otimes i)^{\mathfrak{S}_i}$, satisfying the following properties. 
\begin{enumerate}
    \item $c[1]=e_1$ of Friedlander-Suslin.
    \item For every $i,j$ we have 
    \begin{align*}
        H^{2(i+j)}(\GL_n,\Gamma^{i+1}(\mathfrak{gl}_n^{(1)})) & \to H^{2(i+j)}(\GL_n, \Gamma^i(\mathfrak{gl}_n^{(1)})\otimes \Gamma^j(\mathfrak{gl}_n^{(1)}))\\ 
        c[i+j] & \mapsto c[i]\cup c[j].
    \end{align*}
\end{enumerate}
\end{Th}

\pagebreak

\section{Provisional CFG}\label{section:provisionalCFG}

Let $k$ be a commutative Noetherian ring. 
For a $\GL_n$--algebra $A$ finitely generated as $k$--algebra, we will show that the existence of a uniform bound on the torsion submodule of $H^*(\GL_n, A)$ implies that it is in fact a finitely generated algebra. In fact, the same arguments work for a Chevalley group scheme $G$.  This is Theorem \ref{torsion}, but first we need some intermediate results. 

\begin{Lemma}\label{lemma:reduction-mn}
    Let $m, n>1$. Consider the reduction  map $f\colon  A/mnA\to A/nA$. Then the induced map  $f^\ast\colon H^{\even}(G, A/mnA) \to H^{\even}(G, A/nA)$ is power surjective.
\end{Lemma}

\begin{proof}
    Let $x\in H^{\even}(G, A/nA)$ be homogeneous. We want to show that some power of $x$ lifts along $f^\ast$. Let $I$ be the kernel of $f$.  Clearly $mI=0$, and hence also $mH^*(G, I)=0$. The long exact sequence on invariants gives us a connecting homomorphism 
    \[
    \partial_i :H^i(G, A/nA)\to H^{i+1}(G, I)
    \]
    which satisfies the Leibniz rule with respect to the graded $H^*(G, A/nA)$ multiplication on $H^*(G, I)$. Therefore 
    \[
    \partial_{2nm}(x^m)=\partial_1(x)x^{m-1}+x(\partial_1(x)x^{m-2}+x(\dots))=mx^{m-1}\partial_{2n}(x)=0
    \]
    where we use the fact that $x$ has even degree to eliminate minus signs appearing from the $(-1)^{|x|}$ part of the Leibniz rule. In other words $x^m\in \ker \partial_{2nm}$, which implies $x^m\in \im H^{2nm}(f)$ by exactness.
\end{proof}

\begin{Def}
    Let $A$ be an abelian group, we will say that \textit{$A$ has bounded torsion} if there is a positive integer that annihilates the torsion subgroup $A_{\tors}$ of $A$.
\end{Def}

\begin{Prop}\label{Prop: power surjectivity on cohomology at a prime}
    If $H^*(G, A)$ has bounded torsion, then the reduction map 
    \[
    H^{\even}(G, A) \to H^{\even}(G, A/pA)
    \]
    is power surjective for every prime number $p$.
\end{Prop}

\begin{proof}
    Assume $H^*(G, A)$ has bounded torsion. Note that $A_{\tors}$ is an ideal of the finitely generated algebra $A$, and $A$ itself is Noetherian since it is a finitely generated algebra over a Noetherian base. Therefore we can see that $A_{\tors}$ is finitely generated and thus also has bounded torsion. Now let $p$ be any prime number, and pick $k$ large enough such that multiplication by $n=kp$ annihilates both $H^{>0}(G, A)$ and $A_{\tors}$. 

    The short exact sequence $0\to A_{\tors} \to A\xrightarrow{\pi}{} A/A_{\tors}\to 0$ gives us a long exact sequence on cohomology:
    \[
    \dots \to H^i(G, A) \xrightarrow[]{H^i(\pi)} H^i(G, A/A_{\tors}) \xrightarrow[]{\partial} H^{i+1}(G, A_{\tors}) \to \dots
    \]
Since the sequence is exact, we have $\im H^i(\pi)=\ker \partial$. In particular, both $\im H^i(\pi)$ and $\im \partial$ are annihilated by $n$ (since module homomorphisms are linear and $n$ annihilates $H^i(G, A)$ and $H^{i+1}(G, A_{\tors})$ respectively). Further, 
\[
\coker H^i(\pi) = H^i(G, A/A_{\tors})/\im H^i(\pi)=  H^i(G, A/A_{\tors})/\ker \partial \simeq \im \partial
\]
where the last isomorphism comes from the first isomorphism theorem. Thus we have that $\coker H^i(\pi)$ is also annihilated by $n$. We can write $H^i(G, A/A_{\tors})\simeq \coker H^i(\pi) \oplus \im H^i(\pi)$, and both of these summands are annihilated by $n$, thus we see that $H^i(G, A/A_{\tors})$ is itself annihilated by $n$. 

We have another short exact sequence $0\to n^2A \to A/A_{\tors}\to A/(n^2A+A_{\tors})\to 0$, which induces a map in cohomology 
\[
H^i(G, A/A_{\tors}) \xrightarrow{}{} H^i(G, A/(n^2A +A_{\tors}))
\]
whose kernel is equal to the image of the map $H^i(G, n^2A)\to H^i(G, A/A_{\tors})$, but this map in cohomology is zero since $H^i(G, A/A_{\tors})$ is annihilated by $n^2$. Thus we see that the map $H^i(G, A/A_{\tors}) \hookrightarrow H^i(G, A/(n^2A +A_{\tors}))$ is actually an inclusion.

Now we have two more short exact sequences:
\begin{align*}
    0 \to A/A_{\tors} \xrightarrow[]{\times n^2}  A\to A/n^2A \to 0,  \textrm{ and } \\
    0 \to A/(n^2A +A_{\tors}) \xrightarrow[]{\times n^2} A/n^4A \to A/n^2A \to 0.
\end{align*}

Which induce a commutative diagram on cohomology:
\begin{center}
   \begin{tikzcd}
    &H^{2i}(G, A) \arrow[d]\arrow[r]&H^{2i}(G, A/n^2A) \arrow[d, equals]\arrow[r, "\partial_{2i}"]&H^{2i+1}(G, A/A_{\tors})\arrow[d, hook] \\
    &H^{2i}(G, A/n^4A) \arrow[r]&H^{2i}(G, A/n^2A) \arrow[r, "\partial_{2i}'"]&H^{2i+1}(G, A/n^2A + A_{\tors}) 
\end{tikzcd} 
\end{center}
Now let $x\in H^{2j}(G, A/n^2A)$, and put $i=jn^2$. The element $x^{n^2}$ is then in degree $2jn^2=2i$, and we hit it with the differential $\partial'_{2i}$ to see that $\partial'_{2i}(x^{n^2})=n^2x^{n^2-1}\partial_{2j}'(x)$, which is zero in $H^{2i+1}(G, A/n^2A + A_{\tors})$ because of the $n^2$ torsion. In particular, using the above commutative diagram, this means that $\partial_{2i}(x^{n^2})$ is also zero, and the exactness of the top row allows us to pull $x^{n^2}$ back to an element of $H^{2i}(G, A)$. Thus we have shown that $H^{\even}(G, A)\to H^{\even}(G, A/n^2A)$ is power surjective.

We deduce that $H^{\even}(G, A/n^2A) \to H^{\even}(G, A/pA)$ is power surjective. Indeed,  apply  Lemma \ref{lemma:reduction-mn} with $m'=k^2p$ and $n'=p$ so that $m'n'=n^2$. The composite of power surjective maps is clearly power surjective, so we are done.
\end{proof}

\begin{Th}\label{torsion}
    Let $G$ be a Chevalley group scheme over a Noetherian ring $k$. If $H^*(G, A)$ has bounded torsion, then $H^*(G, A)$ is a finitely generated algebra $k$--algebra.
\end{Th}

\begin{Ex}
    If either $k$ contains $\mathbb{Q}$ or if $H^i(G, A)=0$ for $i>>0$, then $H^*(G, A)$ has bounded torsion.
\end{Ex}

\begin{proof}[Proof of Theorem \ref{torsion}]
    Recall that in  Corollary \ref{Cor: A1/m is acyclic for some m}, we showed that there exists $m>0$ such that $H^\ast(G,A)$ localized away from $m$ is trivial. Hence we only need to show that $H^\ast(G,A)$ is a finitely generated after localization at a finite number of primes, namely the primes dividing $m$. In other words, there is no loss of generality by focusing on the $p$-primary part of $H^\ast(G,A)$, one prime at the time. Fix a prime $p$. Then the $p$-primary part of $H^\ast(G,A)$ corresponds to $H^\ast(G,A)\otimes \mathbb{Z}_{(p)}$. Moreover, using the Base Change Lemma  \ref{base change}, we obtain that the latter is isomorphic to $H^\ast(G_{k\otimes \mathbb{Z}_{(p)}},A\otimes \mathbb{Z}_{(p)})$, and now $A\otimes \mathbb{Z}_{(p)}$ is a finitely generated $k\otimes \mathbb{Z}_{(p)}$--algebra. Therefore, we may assume that $k=\mathbb{Z}_{(p)}$ and $G=G_{k\mathbb{Z}_{(p)}}$ which remains a Chevalley group scheme.  

    By Proposition \ref{Prop: power surjectivity on cohomology at a prime}, we obtain that 
    \[
    H^{\even}(G, A) \to H^{\even}(G, A/pA)
    \]
    is power surjective, in particular, $H^{\even}(G, A/pA)$ is integral over $H^{\even}(G, A)$; see Lemma \ref{integral1}. On the other hand, note that $A/pA$ has now a structure of $k/pk$--algebra, and hence (CFG) for $\GL_n$ over a ring containing a field (see Theorem \ref{CFG for GLn}) together with other layer of the Base Change Lemma \ref{base change} give us that $H^\ast(G,A/pA)$ is a finitely generated $A^G$--algebra. Here we use also (FG) for $G$, see Theorem \ref{Hilberts 14th}. Thus we can choose a $A^G$--subalgebra of $H^\ast(G,A)$, say $\mathscr{A}$, such that it is finitely generated by homogeneous elements and the map $\mathscr{A}\to H^\ast(G,A/pA)$ is Noetherian. 

    A similar reasoning by applying the Base Change Lemma together with (CFG) for $\GL_n$, gives us that the map $H^\ast(G,A/pA)\to H^\ast(G,A/pA+A_{tors})$ is Noetherian, and hence the composition $\mathscr{A}\to H^\ast(G,A/pA+A_{tors})$ is Noetherian as well. 

    Assume that $H^{>0}(G,A/A_{tors})$ is a Noetherian $\mathscr{A}$--module. 
    Again, by (CFG) we obtain that $H^{>0}(G,p^iA/p^{i+1})$ is a Noetherian $H^\ast(G,A/pA)$--module for each $i\geq 0$, and hence it is also a Noetherian $\mathscr{A}$--module.  Since $A_{tors}$ has bounded torsion, we have a finite filtration 
    \[
    A_{tors}\supseteq pA_{tors}\supseteq p^2A_{tors}\supseteq\ldots \supseteq p^sA_{tors}=0
    \]
     and hence we deduce that $H^{>0}(G,A_{tors})$ is a Noetherian $\mathscr{A}$--module. It follows that $H^\ast(G,A)$ is also a Noetherian $\mathscr{A}$--module. Therefore, $H^\ast(G,A)$ is a finitely generated $A^G$--algebra, and by (FG) for $G$, we deduce that it is also a finitely generated $k$--algebra. 

    It remains to show that $H^{>0}(G,A/A_{tors})$ is a Noetherian $\mathscr{A}$--module. Indeed, choose homogeneous elements $v_i\in H^{>0}(G,A/A_{tors})$ generating such that its image generate 
    \[
    W=\mathrm{Im}(H^{>0}(G,A/A_{tors})\to H^{>0}(G,A/pA+A_{tor})).
    \]
    Now, let $V$ denote the $\mathscr{A}$--span of the $v_i$. We claim that $H^{>0}(G,A/A_{tors})=V$.  Since $W$ agrees with $H^{>0}(G,A/A_{tors})/pH^{>0}(G,A/A_{tors})$ as $\mathscr{A}$--modules, we deduce that 
    \[
    H^{>0}(G,A/A_{tors})+V\subseteq p^rH^{>0}(G,A/A_{tors}) +V 
    \]
    for each $r>0$. Finally,  note that $H^{>0}(G,A/A_{tors})$ has bounded torsion since both $H^\ast(G,A)$ and $H^\ast(G,A_{tors})$ have bounded torsion. Then the claim follows. 
\end{proof}

\pagebreak

\section{Bounded Torsion}\label{section:boundedtorsion}

In this section we will establish the existence of a uniform bound on the torsion of the graded cohomology for all finite  group schemes, as in the following theorem:

\begin{Th}\label{bounded torsion}
      Let $G$ be a finite group scheme over Noetherian ring $k$. Then there is a positive integer $n$ that annihilates $H^i(G, M)$ for all $i>0$ and for all $G$--modules $M$.
\end{Th}

Firstly, we show that it suffices to check that $H^1(G, M)$ has bounded torsion:

\begin{Lemma}\label{reduction to H1}
    Suppose there is a positive integer $n$ such that $n$ annihilates $H^1(G, M)$ for all $G$--modules $M$, then $n$ also annihilates $H^i(G, M)$ for all $i>0$.
\end{Lemma}

\begin{proof}
    This is a standard dimension shift argument. Embed $M$ into an injective $G$--module $A$, and let $B$ be the cokernel of this embedding, that is, we have a short exact sequence $0\to M\to A\to B\to 0$ of $G$--modules. Taking cohomology, we get a long exact sequence 
    \[
    0\to M^G\to A^G\to B^G\to H^1(G, M)\to H^1(G, A)\to H^1(G, B)\to H^2(G, M)\to H^2(G, A) \to \dots
    \]
Since $A$ is injective, we see that the $H^i(G, A)$ for $i>0$ all vanish, and thus that $H^i(G, B)\simeq H^{i+1}(G, M)$. By assumption $n$ kills all the first cohomologies in the above sequence.
Thus we see that $H^2(G, M)\simeq H^1(G, B)$ is killed by $n$. Since $M$ was arbitrary, we also see that $H^2(G, B)$ is killed by $n$, and thus that $H^3(G, M)$ also is since these are isomorphic, and so on inductively.
\end{proof}

\subsection*{Reduction to the local case}

The full proof of Theorem \ref{bounded torsion} consists of a few reductions to simpler settings, in the following we reduce to the case where we are working with free modules over local rings:

\begin{Lemma}\label{open cover}
    Let $f_1, \dots, f_s \in k$ have the property that the distinguished open sets $D(f_i)$ cover $\Spec(k)$. If $M_{f_i}$ has bounded torsion for all $i$, then $M$ has bounded torsion.
\end{Lemma}

\begin{proof}
    Suppose that for each $i$, $n_i$ kills the torsion submodule of $M_{f_i}$. Let $m$ be a torsion element of $M$, then clearly $m/1$ is torsion in $M_{f_i}$ for all $i$, and there exists a $k_i$ such that $f_i^{k_i}n_im=0$ in $M$. Now for each $i$ we have that $f_i^{k_i}$ is in the annihilator of $n_1\cdots n_s m$. Since $D(f_i)$ cover $\Spec(k)$, we can write $1=\sum_{i=1}^s a_i f_i^{b_i}$. Repeatedly multiplying by $1$, we can see that $1\in \Ann(n_1\cdots n_s m)$, and hence that  $n_1\cdots n_s m=0$.
\end{proof}

If a finitely generated $R$--module $M$ is free at some stalk, then we can find an open set containing this point on which $M$ is still free.
\begin{Lemma}\label{free over open set}
    Let $\mathscr{F}$ be a coherent sheaf on a scheme $X$ such that $\mathscr{F}_{x}$ a free $\mathcal{O}_x$--module for some $x\in X$. Then there exists an open set $U\subseteq X$ containing $x$ such that $\mathscr{F}|_U$ is free.
\end{Lemma}
\begin{proof}
 
This is question can be answered on an affine open set. More precisely, suppose that $M_{\mf{p}}$ is a free $k_\mf{p}$--module for some $\mf{p}$, then we can find an open set $D(f) \subseteq \Spec(k)$ on which $M_f$ is a free $k_f$--module. Let $m_1,\dots, m_n$ be generators for $M$, and let $x_1, \dots, x_s$ be free generators for $M_{\mf{p}}$, and let $D(g)$ be the distinguished open set over which the $x_i$ are all simultaneously sections of (just take the $g$ to be the product over $i$ of the denominators of the $x_i$). Clearly the images of the $m_i$ generate $M_{\mf{p}}$ as an $k_\mf{p}$--module. Write $m_j=\sum_{i=1}^s \alpha_{ij} x_i$, where $\alpha_{ij}=a_{ij}/b_{ij}$ with $a_{ij}\in M, b_{ij}\notin \mf{p}$. Let $h=\prod_{i,j} b_{ij}$, then on $D(gh)$ the expressions above for the $m_j$ make sense, and this shows that these $m_j$ are in the span of $x_1,\dots, x_s$. That is to say that the collection $x_1,\dots, x_s$ does indeed generate $M_{gh}$ as an $k_{gh}$--module. In particular the $x_j$ are free generators for $M_{gh}$, since they  were free generators for $M_\mf{p}$. 
\end{proof}

For the following we adopt the notation of Section \ref{Sec: Hopf algebras}. Let $G$ be a finite flat group scheme, and $H=k[G]$ be its coordinate Hopf algebra, which is finitely generated and projective over the Noetherian ring $k$, in particular, its dual $H^*$ is also a Hopf algebra which is a right Hopf module over $H$ by Lemma \ref{H* is a right Hopf Module}. From Lemma \ref{summand} and \ref{coinvariants are rank 1} we see that the module of coinvariants $(H^*)^{\text{co}H}$ is a rank 1 projective summand of $H^*$ as a $k$--module. Moreover, since $H$ is projective, we see from Lemma \ref{invariants and coinvariants} that  $(H^*)^{\text{co}H}=(H^*)^{H^*}$. In particular from Lemma \ref{right mult}, we see that $(H^*)^{\text{co}H} M\subseteq M^G$ for any $G$--module $M$.

Of course, the $k$--module $(H^*)^{\text{co}H}$ is free of rank 1 at every stalk, that is, $(H^*)^{\text{co}H}_\mf{p} \simeq k_\mf{p}$, this is since it is a rank 1 projective module over a local ring. Using Lemma \ref{free over open set} we can then see that there is a finite cover of $\Spec(k)$ over which $(H^*)^{\text{co}H}$ is free of rank $1$. Therefore, if for any $G$--module $M$ we can show that $M$ has bounded torsion whenever $(H^*)^{\text{co}H}$ is free, we can use the combination of Lemma \ref{open cover} and Lemma \ref{free over open set} to conclude the result in general.

This is to say that when we move on to the proof of Theorem \ref{bounded torsion} we may always assume that $(H^*)^{\text{co}H}$ is free of rank $1$.

\subsection*{Reduction to characteristic $0$}

We now may also reduce to the case of characteristic zero. Of course if $1\in k$ has finite additive order $n$, then $n$ is sufficient to bound the torsion for $H^i(G, M)$. So we may now always suppose that $k$ contains $\mathbb{Z}$ and consider $\mathbb{Q}\otimes k$.

\subsection*{Reduction to geometric points}

There is a crucial step in the proof of Theorem \ref{bounded torsion} where our goal is to show that an element $a\in \mathbb{Q}\otimes k $  is a unit. What we show here is that it suffices to check that $a$ is a unit at every geometric point of $\Spec (\mathbb{Q}\otimes k)$.

Let $k'$ denote a commutative ring. Recall that a geometric point of a $\Spec(k')$ is a morphism $\varphi \colon \Spec F\to \Spec(k')$, where $F$ is an algebraically closed field containing the residue field of the unique point in the image of this map. More precisely, let $\mf{p}\in \Spec(k')$ be the unique point in the image of this morphism, i.e., the image of the $0$ ideal under the morphism $\varphi$, it has residue field $k'_\mf{p}/\mf{m}_\mf{p}$, and we demand that $F$ contains $k'_\mf{p}/\mf{m}_\mf{p}$. We get the commutative diagram: 
\begin{center}
    \begin{tikzcd}[column sep=small]
        & &\Spec (k'_\mf{p}/\mf{m}_\mf{p}) \arrow[dr] & \\
        &\Spec (F) \arrow[ur]\arrow[rr, "\varphi"']& &\Spec (k') 
    \end{tikzcd}
\end{center}

Now let $a\in k'$. To check that $a$ is a unit, it suffices to check that $a$ is a unit at every geometric point of $\Spec(k')$. Clearly, $a$ is a unit in $k'$ if and only if it is not contained in any maximal ideal of $k$, which is the same as the element $a$ being a unit in every local ring $k'_\mf{p}$, which is then clearly the same as $a$ being non-zero in every residue field $k'_\mf{p}/\mf{m}_\mf{p}$. Of course, if $a$ is non-zero in each residue field, then this is equivalent to it being be non-zero at every geometric point, since maps of fields are always injective. Therefore, to check that some element $a\in k'$ is a unit, it suffices to check that it is a unit at an arbitrary geometric point of $\Spec(k')$.

\begin{proof}[Proof of Theorem \ref{bounded torsion}]
    Let $H=k[G]$. Our \textit{reduction to characteristic $0$} allow us to assume that $k$ contains $\mathbb{Z}$ and work with $k_1=\mathbb{Q}\otimes k$. Our \textit{reduction to the local case} allows us to assume that $(H^*)^{\text{co}H}$ is free of rank 1 and is a summand of $H^*$. 
    
    Since $(H^*)^{\text{co}H}$ is free of rank $1$ we may choose a generator $\psi \in (H^*)^{\text{co}H}$.

    Firstly we will prove that $\psi(1)$ is a unit in $k_1$. By the \textit{reduction to geometric points} it suffices to show this for any geometric point $\Spec F$ of $\Spec (k_1)$. Now since we are in characteristic $0$ and $F$ is assumed to be algebraically closed we see from Theorem \ref{Thm:Cartier} and Proposition \ref{Prop:etale in charc 0 are constant} that all finite group schemes are equivalent to constant group schemes associated to the finite group $G(F)=\Hom_{\textbf{CAlg}_F}(F[G], F)$. Clearly $G(F)$ embeds into $F[G]^*=\Hom_{\textbf{Mod}_F}(F[G], F)$ since maps of $F$--algebras are also maps of the underlying $F$--modules. Let $\psi_0=\sum_{g\in G(F)} g$. From the description of $(H^*)^{\text{co}H}$ it follows that we may write $(H^*)^{\text{co}H}=F\sum_{g\in G(F)} g$.  Thus we conclude that $\psi=s\psi_0$ for some $s\in F$. Since each $g\in G(F)$ is a map of algebras, we have that $g(1)=1$, and thus that $\psi_0(1)=|G(F)|$, which is invertible in $F$. In particular $\psi(1)=s|G(F)|$ is also invertible in $F$. So now we see that $\psi(1)$ is a unit in $k_1$ from our \textit{reduction to geometric points}.

     Let $\psi(1)^{-1}=m/n \otimes r \in \mathbb{Q}\otimes k$. Now set $\psi_1=\psi(1)^{-1}\psi$. Clearly $\psi_1(1)=1$. Therefore we conclude that $1\in (\mathbb{Q}\otimes k)\psi(1)$. Now we have that:
    \[
    \left(\frac{m}{n}\otimes r\right)\psi(1)=1 \implies (1\otimes mnr)\psi(1)=n
    \]
    That is to say that there is an element $t\in k$ such that $t\psi(1)$ is a positive integer $n$. If we put $\phi=t\psi$, then the map $\phi-n$ sends $1$ to $0$, in particular, it is zero on $k \subseteq k[G]$.

    Now, let $M$ be any $k[G]$--comodule with a comultiplication $v$, as in Theorem \ref{equivalence}, we have seen that we can view it as a $k[G]^*$--module via the action $f\cdot m := (\id_M\otimes f)(v(m))$. Moreover, if $m\in M^G$, then we see that $\phi \cdot m=nm$. It is then easy to see that multiplication by $\phi-n$ annihilates $M^G$.

    Now embed $M$ into any injective $G$--module $A$, and let $B$ be the cokernel of this embedding. Since $A$ is injective all the higher cohomology vanish, and in particular from the long exact sequence we get that 
    \[
    \coker(A^G\xrightarrow{\pi}{} B^G) \simeq H^1(M, G).
    \]
    Now take some $b\in B^G$ and some lift $a\in A$. Note that we are not assuming that $a$ is an invariant.  Then 
    \[\phi \cdot b =nb =\pi(na)=\pi(\phi\cdot a)\]
    Moreover, since $\phi$ is left invariant, we have that $\phi \cdot A \subseteq A^G$ from Lemma \ref{right mult}. In particular, $\phi \cdot a\in A^G$.  Thus we see that $nb$  is in the image of $\pi\colon A^G\to B^G$. We deduce that $n$ annihilates 
    $\coker(A^G\xrightarrow{\pi}{} B^G)\cong H^1(G,M). $
    Since $M$ was arbitrary, we have shown that $n$ kills $H^1(M, G)$ for all $M$, now applying Lemma \ref{reduction to H1} we are done.
\end{proof}

\pagebreak

\section{CFG for Finite Group Schemes}

We now have all the ingredients needed to prove van der Kallen's (CFG) theorem over a Noetherian base, we have done the hard work so here is the payoff:

\begin{Th}
    Let $k$ be a commutative Noetherian ring. Let $G$ be a finite group scheme over $k$. Then $G$ satisfies the cohomological finite generation property. 
\end{Th}

\begin{proof}
    Let $A$ be a finitely generated $G$--algebra.  We have seen that finite group schemes satisfy (FG) over a Noetherian base, see Corollary \ref{FG for Chevalley}, so we can apply the Reduction Lemma. We embed $G$ into $\GL_n$ as in Proposition  \ref{embedding of G}, and we see that the quotient is affine by Remark \ref{flat subgroup scheme}. The conclusion of the Reduction Lemma gives us that 
    \[
    H^*(G, A)=H^\ast(\text{GL}_n,\mathrm{ind}_G^{\text{GL}_n} A).
    \]
    Theorem \ref{bounded torsion} shows us that $H^{>0}(G, A)$ has bounded torsion. Note that $A^G$ has bounded torsion since it is Noetherian. Thus we conclude that $H^*(G, A)$ also has bounded torsion. Now since $\GL_n$ is a Chevalley group scheme we can apply Theorem \ref{torsion} and conclude that $H^*(G, A)=H^\ast(\text{GL}_n,\mathrm{ind}_G^{\text{GL}_n} A)$ is a finitely generated algebra.
\end{proof}

By Proposition \ref{Prop: CFG implies finite genarion of modules} we obtain the following corollary: 

\begin{Cor}
    Let $G$ be a finite group scheme over $k$, $A$ be a finitely generated $G$--algebra, and $M$ be an $AG$--module finitely generated over $A$. Then $H^\ast(G,M)$ is finitely generated over $H^\ast(G,A)$. 
\end{Cor}

\pagebreak 

\section{Alternative Approach to Bounded Torsion}\label{sec: alternative bounded}

Let $G$ be a finite group scheme over a commutative Noetherian base $k$, and let $A$ be a $G$--algebra finitely generated over $k$. In this section, we present an alternative argument to show that $H^\ast(G,A)$ has bounded torsion; giving a different proof of Theorem \ref{bounded torsion}. 

The argument we present is due to Eike Lau \cite{Lau2023}. While the result holds in greater generality, as explained in \textit{loc. cit.}, it requires familiarity with algebraic stacks. Thus, this section can be viewed as a translation of Lau's approach to the context of representations over the group scheme $G$. We are deeply grateful to Lau for generously taking the time to explain his approach to us. We would like to emphasize that any mistakes in this section are entirely our own.   

\begin{Claim}\label{claim-trace-equivariat}
    Let $G$ be a finite group scheme over a Noetherian ring $k$. Then there is $G$--equivariant morphism $\mathrm{tr}\colon k[G]\to k$, such that the composite $k\xrightarrow[]{\eta} k[G] \xrightarrow[]{\mathrm{tr}} k$ is multiplication by $\mathrm{rk}_k(k[G])$. 
\end{Claim}
 
We need some preparation.  All rings are assumed to be commutative and with unity unless we specify otherwise. 

\begin{Rec}
Let $A$ be a ring and $P$ be a finitely generated projective $A$--module. Recall that  in this case, the canonical map 
\[
\varphi\colon  P^\ast\otimes_A P \to \mathrm{End}_A(P)
\]
is an isomorphism. The \textit{trace map of $P$} is defined as the composite 
\[
\mathrm{tr}_P\colon  \mathrm{End}_A(P) \xrightarrow[]{\varphi^{-1}} P^\ast\otimes_A P \xrightarrow[]{\mathrm{ev}} A
\]
where $\mathrm{ev}$ is simply the evaluation map associated to the adjunction induced by tensoring with $P$. In fact, it is not too difficult to verify that this construction agrees with the usual trace whenever $P$ is a free module.  
\end{Rec}

\begin{Rec}\label{rec-def-trace}
    Let $A\to B$ be a finite flat map of Noetherian rings. In other words, $B$ is a finite projective  $A$--module. In this case, we write   $\mathrm{tr}_{B/A}$ to denote the trace of $B$ from the previous recollection. Moreover, for any $b\in B$, we write $\mathrm{tr}_{B/A}(b)$ to denote the trace of the endomorphism $B\to B$ given by multiplying by $b$. In this way, we obtain an $A$--linear map $\mathrm{tr}_{B/A}\colon B\to A$. 
\end{Rec}

Interestingly, one can characterize the trace map as follows. We refer to \cite[Tag 0BSY]{stacks-project} for a proof. 

\begin{Lemma}\label{lemma-trace-universalprop}
    Let $A\to B$ be a finite flat map of Noetherian rings. Then the trace map $\mathrm{tr}_{B/A}$ is the unique map in $\omega_{B/A}\coloneqq\Hom_A(B,A)$ with the following  property. 
    \begin{itemize}
        \item[($\star$)] Let $A'$ be a Noetherian $A$--algebra such that $B'=B\otimes_AA'$   comes with a product decomposition $B'=C\times D$, with $A'\to C$ finite. Then the image of $\mathrm{tr}_{B/A}$ under the composite $\omega_{B/A} \to \omega_{B'/A'}\to \omega_{C/A'}$ is precisely $\mathrm{tr}_{C/A'}$.      
    \end{itemize}
\end{Lemma}

We will need the following easy formal consequence of the definition of trace maps.  

\begin{Lemma}\label{lemma-basechange-trace}
      Let $A\to B$ be a finite flat map of Noetherian rings. Let $\varphi\colon A\to A' $ be a ring map, and write $\psi$ to denote the canonical map $B\to B'\coloneqq B\otimes_AA'$. Then 
      \begin{equation}\label{eq-trace}
          \varphi\circ \mathrm{tr}_{B/A} =\mathrm{tr}_{B'/A'}\circ \psi.
      \end{equation}
\end{Lemma}

\begin{proof}
   Recall that finiteness and flatness are preserved under base change, and hence talking about  $\mathrm{tr}_{B'/A'}$ makes sense. Now, note that the maps involved in the definition of the trace map are well behaved under base change. Thus the result follows. 
\end{proof}

\begin{Lemma}\label{lemma-trace-equivariant}
    Let $G$ be a finite group scheme over a Noetherian ring $k$. Then the trace map 
    \[
    \mathrm{tr}_{k[G]/k}\colon k[G]\to k
    \]
    from Recollection \ref{rec-def-trace}, is a map of $k[G]$--comodules. 
\end{Lemma}

\begin{proof}
    Recall that the $k[G]$-coaction on $k$ is given by the unit map $\eta\colon k\to k[G]$. Hence we need to verify that the following square is commutative. 
\begin{equation}\label{eq-squares-multiplication}
    \begin{tikzcd}
            k[G] \arrow[r,"\Delta"] \arrow[d,"\mathrm{tr}_{k[G]/k}"'] & k[G]\otimes k[G] \arrow[d,"\mathrm{tr}_{k[G]/k}\otimes \id"] \\
            k \arrow[r,"\eta"] & k[G].
        \end{tikzcd}
\end{equation}
Now, via the equivalence $(\textbf{AffSch}_k)\simeq (\textbf{CAlg}_k)^{\text{op}}$ we have the corresponding diagrams 
\begin{center}
    \begin{tikzcd}
      G\underset{S}{\times} G \arrow[d,"\mathrm{pr}_2"'] \arrow[r,"m"] & G \arrow[d]               \\
        G \arrow[r,""] & S  
    \end{tikzcd}
$\leftrightsquigarrow
$
    \begin{tikzcd}
        k[G]\otimes  k[G] &  k[G] \arrow[l,"\Delta"'] 
 \\  
  k[G] \arrow[u] &  k \arrow[u,"j_2"'] \arrow[l,"\eta"']
    \end{tikzcd}
\end{center}
where $S$ is short for $\Spec(k)$.  Note that the left hand side square is cartesian. Indeed, consider the following commutative diagram 
\begin{center}
    \begin{tikzcd}
         G\underset{S}{\times} G \arrow[r,"\varphi"] &  G\underset{S}{\times} G \arrow[r,"\mathrm{pr}_1"] \arrow[d,"\mathrm{pr}_2"'] & G \arrow[d] \\ 
                  &  G   \arrow[r]               & S
    \end{tikzcd}
\end{center}
where $\varphi$ is the map given by $(g,h)\mapsto (m(g,h),h)$. The square is cartesian. Moreover,  the left top arrow is an isomorphism. Hence the left hand square in \ref{eq-squares-multiplication} is cartesian. We deduce that the  the right hand square in \ref{eq-squares-multiplication} is cocartesian. Then we can invoke Lemma \ref{base change} for base change along the map $\eta$. In particular, for $a\in k[G]$, Equation \ref{eq-trace} gives us 
\[
\eta(\mathrm{tr}_{k[G]/k}(a)) = \mathrm{tr}_{k[G]\otimes k[G]/k[G]}(\Delta (a)) = (\mathrm{tr}_{k[G]/k}\otimes \id_{k[G]})(\Delta(a))
\]
where the second equality follows by unpacking $j_2$. 
\end{proof}


\begin{Prop}
Let $G$ be a finite group scheme defined over a commutative Noetherian ring $k$. Let $A$ be a $G$--algebra finitely generated over $k$. Then $H^\ast(G,A)$ has bounded torsion.     
\end{Prop}

\begin{proof}
    First, by Corollary  \ref{FG for Chevalley} we have that $A^G$ is finitely generated $k$--algebra, and hence it is enough to show that $H^{>0}(G,A)$ has bounded torsion. Now,  Lemma \ref{lemma-trace-equivariant} tells us that the trace map $\mathrm{tr}_{k[G]/k}$ is $G$--equivariant, and one verifies that the composite  
    \begin{equation}\label{eq-trace-unit}
        k\xrightarrow[]{\eta}k[G]\xrightarrow[]{\mathrm{tr}_{k[G]/k}}k
    \end{equation}
    is indeed multiplication by $d$, where $d$ is the rank of $k[G]$ over $k$. In other words, this verifies Claim \ref{claim-trace-equivariat}. 
    Tensoring Equation \ref{eq-trace-unit}, and applying the cohomology functor we obtain a diagram 
    \begin{center}
        \begin{tikzcd}
           H^{>0}(G, A) \arrow[rr,"\cdot d"] \arrow[rd] && H^{>0}(G, A) \\
               &  H^{>0}(G,A\otimes k[G]) \arrow[ru] & 
        \end{tikzcd}
    \end{center}
    But $H^{>0}(G,A\otimes k[G])=0$ by Corollary \ref{vanishing of cohomology for induced modules}. In particular, multiplication by $d$ on $H^{>0}(G, A)$ factors though zero, and so is itself zero. In other words, $H^{>0}(G,A)$ has bounded torsion, and by our previous observation so does $H^\ast(G,A)$. 
\end{proof}

\pagebreak

\bibliographystyle{alpha}
\bibliography{bibfile}

\end{document}